\RequirePackage[l2tabu, orthodox]{nag}
\documentclass{amsart}
\usepackage{etex}

\usepackage{mathtools}
\usepackage{verbatim}
\usepackage[top = 1in, left = 1.25in, right = 1in, bottom = 1in]{geometry}
\usepackage{graphicx}
\usepackage[all]{xy}
\usepackage{tikz}
\usepackage{enumerate}
\usepackage{multirow}
\usepackage{array}
\usepackage{longtable}
\usepackage{booktabs}
\usepackage{mathrsfs}
\usepackage{amsfonts}
\usepackage{amsmath}
\usepackage{amsthm}
\usepackage{amssymb}
\usepackage{setspace}
\usepackage{dsfont}
\usepackage{caption}
\usepackage{subcaption}
\usepackage{upgreek}

\newtheorem{theorem}{Theorem}[section]
\newtheorem{lemma}[theorem]{Lemma}
\newtheorem{proposition}[theorem]{Proposition}
\newtheorem{corollary}[theorem]{Corollary}

\theoremstyle{definition}
\newtheorem{definition}[theorem]{Definition}
\newtheorem{example}[theorem]{Example}

\newtheorem*{ack*}{Acknowledgments}
\newtheorem{conjecture}[theorem]{Conjecture}
\newtheorem{rem}[theorem]{Remark}

\theoremstyle{remark}

\numberwithin{equation}{section}
\numberwithin{table}{section}
\numberwithin{figure}{section}

\newcommand{\be}{\begin{equation}}
\newcommand{\en}{\end{equation}}

\DeclareMathOperator{\Int}{Int}
\DeclareMathOperator{\Asc}{Asc}
\DeclareMathOperator{\Acc}{Acc}
\DeclareMathOperator{\Asp}{Asp}
\DeclareMathOperator{\Alt}{Alt}
\DeclareMathOperator{\Asw}{Asw}
\newcommand{\htop}{h_{\text{\normalfont top}}}
\DeclareMathOperator{\diam}{diam}
\DeclareMathOperator{\card}{card}

\newcommand{\setcomp}[1]{{#1}^{\mathsf{c}}}

\DeclarePairedDelimiter{\ceil}{\lceil}{\rceil}
\DeclarePairedDelimiter{\floor}{\lfloor}{\rfloor}

\begin{document}

\title{Dynamical intricacy and average sample complexity}

%    Remove any unused author tags.

%    author one information
\author{Karl Petersen}
\address{Department of Mathematics,
CB 3250 Phillips Hall,
University of North Carolina,
Chapel Hill, NC 27599 USA}
%\curraddr{}
\email{petersen@math.unc.edu}
%\thanks{}

%    author two information
\author{Benjamin Wilson}
\address{Department of Mathematics,
CB 3250 Phillips Hall,
University of North Carolina,
Chapel Hill, NC 27599 USA}
\curraddr{Department of Applied Mathematics, Stevenson University,
1525 Greenspring Valley Rd,
Stevenson, MD 21117 USA}
\email{bwilson4@stevenson.edu}
%\thanks{}



\date{\today}

\dedicatory{}

\begin{abstract}
We propose a new way to measure the balance between freedom and coherence in a dynamical system and a new measure of its internal variability. Based on the concept of entropy and ideas from neuroscience and information theory, we define \emph{intricacy} and \emph{average sample complexity} for topological and measure-preserving dynamical systems. We establish basic properties of these quantities, show that their suprema over covers or partitions equal the ordinary entropies, compute them for many shifts of finite type, and indicate natural directions for further research.
\end{abstract}

\maketitle

\section{Introduction}

In their study of high-level neural networks~\cite{tononi1994measure}, G. Edelman, O. Sporns, and G. Tononi introduced a quantitative measure that they call \emph{neural complexity} to try to capture the interplay between two fundamental aspects of brain organization: the functional segregation of local areas and their global integration. Neural complexity is high when functional segregation coexists with integration and is low when the components of a system are either completely independent (segregated) or completely dependent (integrated).
J. Buzzi and L. Zambotti~\cite{BZ12} provided a mathematical foundation for neural complexity by placing it in a natural class of functionals: the averages of mutual information satisfying exchangeability and weak additivity. The former property means that the functional is invariant under permutations of the system, the latter that it is additive when independent systems are combined. They gave a unified probabilistic representation of these functionals, which they called \emph{intricacies}.

In this paper we define and then study \emph{intricacy in dynamical systems}, based on the classical definition of topological entropy in dynamical systems and intricacy as defined by Buzzi and Zambotti. We define \emph{topological intricacy} and the closely related \emph{topological average sample complexity} for a general topological dynamical system $(X,T)$ with respect to an open cover $\mathscr{U}$ of $X$. More specifically, denote by $n^*$ the set of integers $\{0,1,\dots, n-1\}$, let $S=\{s_0,s_1,\dots,s_{|S|-1}\}\subset n^*$, let $\setcomp{S}=n^*\setminus S$, let $c_S^n$ be a weighting function that depends on $S$ and $n$, let $\mathscr{U}_S=\bigvee_{i=0}^{|S|-1}T^{-s_i}\mathscr{U}$, and let $N(\mathscr{U})$ be the minimum cardinality of a subcover of $\mathscr{U}$. Then the topological intricacy of $(X,T)$ with respect to the open cover $\mathscr{U}$ is defined to be
\be
\Int(X,\mathscr{U},T)=\lim_{n\rightarrow\infty}\frac{1}{n}\sum_{S\subset n^*}c_S^n\log\left(\frac{N(\mathscr{U}_S)N(\mathscr{U}_{\setcomp{S}})}{N(\mathscr{U}_{n^*})}\right).
\en
Breaking up the logarithm and sum shows that one should study the {\em average sample complexity},
\be
\Asc(X,\mathscr{U},T):=\lim_{n\rightarrow\infty}\frac{1}{n}\sum_{S\subset n^* } c_S^n\log N(\mathscr{U}_S).
\en
We will initially let $c_S^n=2^{-n}$ for all $S$. Since we are averaging the quantity $\log(N(\mathscr{U}_S)N(\mathscr{U}_{\setcomp{S}})/N(\mathscr{U}_{n^*}))$ over all subsets $S\subset n^*$, topological intricacy takes on high values for systems in which for most $S$ the product $N(\mathscr{U}_S)N(\mathscr{U}_{\setcomp{S}})$ is large compared to $N(\mathscr{U}_{n^*})$. We will see that this happens in systems that are far from from both total order and total disorder. Intricacy may be thought of as a measure of something like organized flexibility within a system, and average sample complexity as a measure of possible internal variability. 

We define intricacy and average sample complexity for measure-preserving systems by taking probabilities of configurations into account rather than just counting them. 
Let $(X,\mathscr{B},\mu,T)$ be a measure-preserving system and $\alpha=\{A_1,\dots, A_k\}$ a finite measurable partition of $X$. Given $S\subset n^*$ and $c_S^n$ as above, let $\alpha_S=\bigvee_{i=0}^{n-1}T^{-s_i}\alpha$ and $H_\mu(\alpha)=-\sum_{i=1}^k\mu(A_i)\log\mu(A_i)$. Then the measure-theoretic intricacy of $X$ and $T$ with respect to $\alpha$ is defined to be
\be
\Int_\mu(X,\alpha,T)=\lim_{n\rightarrow\infty}\frac{1}{n}\sum_{S\subset n^*}c_S^n[H_\mu(\alpha_S)+H_\mu(\alpha_{\setcomp{S}})-H_\mu(\alpha_{{n^*}})].
\en
For similar reasons as in the topological case, measure-theoretic intricacy takes high values for systems that are far from both order and disorder.
As in the topological case, measure-theoretic intricacy also involves a component interesting in its own right, the {\em measure-theoretic average sample complexity}:
\be
\Asc_\mu(X,\alpha,T)=\lim_{n\rightarrow\infty}\frac{1}{n}\sum_{S\subset n^*}c_S^nH_\mu(\alpha_S).
\en
Existing concepts such as sequence entropy \cite{kushnirenko1967metric, newton1970sequenceI, newton1970sequenceII, krug1972sequence, saleski1977sequence} and {maximal pattern complexity} \cite{KTZseq02, KTZmax02} also involve sampling a system at a selected set of times, but $\Int$ and $\Asc$ include all possibilities over all finite sets of sampling times. 

Two of the main results in this paper (Theorems \ref{asctoentthm} and  \ref{mtheoreticasctoent}) establish a relationship between topological intricacy and topological entropy as well as between measure-theoretic intricacy and measure-theoretic entropy. 
We show that intricacy is bounded above by entropy in both the topological and measure-theoretic settings. One result of this is that systems of zero entropy also have zero intricacy, so intricacy takes on low values for integrated systems. It is also easy to see that independent systems have zero intricacy.
Entropy in dynamics classically is first defined with respect to either a specific cover of a topological space or a specific partition of a measure space. To define the entropy of a transformation as an invariant under topological conjugacy or measure-theoretic isomorphism, one then takes the supremum over all open covers or over all partitions. 
We define intricacy with respect to a cover and with respect to a partition, but in a corollary of Theorem~\ref{asctoentthm} we show, for $c_S^n=2^{-n}$, that 
$\sup_\mathscr{U}\Int(X,\mathscr{U},T)=\sup_{\mathscr U} \Asc(X, \mathscr U,T)=
\htop(X,T)$, the usual topological entropy of the system. Similarly in the measure-theoretic setting, in Theorem~\ref{mtheoreticasctoent} and Corollary \ref{cor:supintent} we show for $c_S^n=2^{-n}$ that 
$\sup_\alpha \Int_\mu(X,\alpha,T)=\sup_\alpha \Asc_\mu (X, \alpha,T)=h_\mu(X,T)$, 
the usual measure-theoretic entropy. 
Thus attempts to define conjugacy invariants from these quantities lead to nothing new. However, looking at these measurements for specific partitions and open covers provides finer information of a new kind about interactions within dynamical systems (see, for example, Examples \ref{1317example}, \ref{ex:2}, and \ref{ex:3}).

The main topological examples we examine are subshifts, which are closed shift-invariant collections of infinite sequences of elements from a finite alphabet. The topological entropy of a subshift is the exponential growth rate of the number of words of each length found in sequences in the subshift. To find the intricacy 
or average sample complexity of a subshift, rather than counting all words of length $n$, we find an average of the number of words seen at the places in a subset $S\subset n^*$. Averaging in this manner creates a measurement that is more sensitive to the structure of the sequences in a subshift than is the entropy.

While we can approximate intricacy and average sample complexity  for subshifts, computing the actual quantities is difficult in general, since in principle for each $n$ we have to make computations on all $2^n$ subsets of $n^*$. Theorem~\ref{squarethm} provides a formula for the average sample complexity for particular covers of certain shifts of finite type.
In the measure-theoretic setting Theorems~\ref{Asccompenteqthm} and \ref{thm:FirstReturns} and Proposition \ref{infocor} give a relationship between measure-theoretic average sample complexity with respect to a finite partition $\alpha$, the fiber entropy of 
the first-return (or skew product) map on a cross product, and a series involving the conditional entropies $H_\mu(\alpha\mid\alpha_i)$. More specifically, we show that
for 1-step Markov shifts
\be
\Asc_\mu(X,\alpha,T)  = \sum_{i=1}^\infty2^{-i-1}H_\mu(\alpha\mid\alpha_i).
\en
   We use this equation to compute the measure-theoretic average sample complexity and measure-theoretic intricacy for $1$-step Markov measures on the full $2$-shift and $1$-step and $2$-step Markov measures on the golden mean shift. 
  Analysis of these data leads to conjectures about measures that maximize average sample complexity and measures that maximize intricacy. 
  Appendix A presents the generalizations to average sample pressure, and Appendix B extends the results to general weights.
 We have defined some new quantities and found out only the first few new things about them; 
 we conclude by mentioning some questions raised by this work that we think deserve further study.
 
\subsection{Some terminology and notation}  
We assume the basic terminology and notation of topological dynamics, symbolic dynamics, and ergodic theory, as found for example in \cite{lind1995an}, \cite{petersen1989ergodic}, and \cite{waltersergodic}.
For us a
 \emph{topological dynamical system} $(X,T)$ is a compact Hausdorff (often metric) space $X$ with a continuous transformation $T:X\rightarrow X$, and 
  a  {\em measure-preserving system} $(X,\mathscr{B},\mu, T)$ consists of a complete probability space $(X, \mathscr B, \mu)$ and a one-to-one onto map $T:X\rightarrow X$ such that $T$ and $T^{-1}$ are both measurable. 
 We denote by $n^*$ the set of integers from $0$ to $n-1$. i.e.
$n^*=\{0,1,\dots, n-1\}$.
Given a subset $S\subset n^*$, we denote its complement by $\setcomp{S}=n^*\setminus S$. We denote the number of elements in a set $A$ by either card$(A)$ or $|A|$. Unless otherwise specified, logarithms will be taken base $e$. We take the convention that $0\log 0=0$.

 The (two-sided) \emph{full shift space} $\Sigma(\mathcal{A})$ over an alphabet $\mathcal A$ is defined to be $\Sigma(\mathcal{A})=\prod_{-\infty}^\infty \mathcal{A}=\{x=(x_i)_{-\infty}^\infty:x_i\in \mathcal{A}\text{ for each }i\}$
 and is given the product topology. For us $\mathcal A$ is finite and has the discrete topology. The one-sided full shift space is $\Sigma(\mathcal{A})^+=\{x=(x_i)_0^\infty:x_i\in \mathcal{A}\text{ for each }i\}$. The shift transformation $\sigma:\Sigma(\mathcal{A})\rightarrow\Sigma(\mathcal{A})$ is defined by $(\sigma x)_i=x_{i+1}\quad\text{for }-\infty<i<\infty$, and $\sigma:\Sigma(\mathcal{A})^+\rightarrow\Sigma(\mathcal{A})^+$ is defined by $(\sigma x)_i=x_{i+1}$ for $0\le i<\infty$.
  If $\mathcal{A}=\{0,1,\dots, r-1\}$ then we denote $\Sigma(\mathcal{A})$ or $\Sigma(\mathcal{A})^+$ by $\Sigma_r$ or $\Sigma_r^+$ and call it the \emph{full $r$-shift}. We will deal only with two-sided shift spaces over a finite alphabet $\mathcal{A}=\{0,1,\dots, r-1\}$ unless otherwise stated. 
  	A \emph{subshift} is a pair $(X,\sigma)$, where $X\subset \Sigma_r$ is a nonempty, closed, shift-invariant ($\sigma X=X$) set. A \emph{block} or \emph{word} is an element of $\mathcal{A}^r$ for some $r=0,1,2\dots$, i.e. a finite string on the alphabet $\mathcal{A}$. If $x$ is a sequence in a subshift $X$, we will sometimes denote the block in $x$ from position $i$ to position $j$ by $x_{[i,j]}=x_{i}x_{i+1}\cdots x_{j}$. We denote the empty block by $\epsilon$. 
 Denote the set of words of length $n$ in a subshift $X$ by $\mathscr{L}_n(X)$, i.e, $\mathscr{L}_n(X)=\left\{x_{[i,i+n-1]}:x\in X, i\in\mathbb{Z}\right\}.$ The \emph{language} of a subshift $X$ is $\mathscr{L}(X)=\bigcup_{n=0}^\infty\mathscr{L}_n(X)$.
 
 Let $S\subset n^*$, $S=\{s_0,s_1,\dots,s_{|S|-1}\}$, and suppose $w\in\mathscr{L}_n(X)$ such that $w_{s_i}= a_{s_i}$ for $i=0,\dots,|S|-1$ and $a_{s_i}\in\mathcal{A}$. Then we call $a_{s_0}a_{s_1}\cdots a_{s_{|S|-1}}$ a \emph{word at the places in $S$}. Denote the set of words we can see at the places in $S$ for all words in $\mathscr{L}_n(X)$ by $\mathscr{L}_S(X)$. More formally, if $S=\{s_0,s_1,\dots,s_{|S|-1}\}$, then
 \begin{equation}\label{bseq}
 \mathscr{L}_S(X)=\{x_{s_0}x_{s_1}\dots x_{s_{|S|-1}}:x\in X\}.
 \end{equation}
  Notice that $\mathscr{L}_{n^*}(X)=\mathscr{L}_n(X)$. Given a subshift $X\subset\Sigma(\mathcal{A})$, we will often consider the cover $\mathscr{U}_n$ consisting of \emph{rank $n$ cylinder sets}
 \begin{equation}
   C_{-n}[i_{-n}, \dots, i_n] =  \{x\in X:x_{-n}=i_{-n},x_{-n+1}=i_{-n+1},\dots,x_0=i_0,\dots,x_n=i_n\}
 \end{equation}
 for some choices of $i_{-n},i_{-n+1},\dots,i_n\in\mathcal{A}$, and similarly for covers $\mathscr{U}_n$ of one-sided subshifts.
 	A \emph{shift of finite type} (SFT) is defined by specifying a finite collection, $\mathcal{F}$, of forbidden words on a given alphabet, $\mathcal{A}=\{0,1,\dots,r\}$. Given such a collection $\mathcal{F}$, define $X_\mathcal{F}\subset \Sigma_r$ to be the set of all sequences none of whose subblocks are in $\mathcal{F}$. i.e.
 	\begin{equation}
 	X_\mathcal{F}=\{x\in\Sigma(\mathcal{A}):
 	\text{ for all }i,j\in\mathbb{Z},
    x_{[i,j]}	\not\in\mathcal{F}\}.
 	\end{equation}

\section{Topological intricacy and average sample complexity}\label{Sec:IntAsc}
Our definitions of intricacy and average sample complexity are based on the idea of neurological complexity proposed by Edelman, Sporns, and Tononi \cite{Tononi94} and its probabilistic generalizations by Buzzi and Zambotti \cite{BZ12}. 
An important initial consideration is the identification of the families of weights that are appropriate to use for the averaging over subsets involved in the basic definitions.
A \emph{system of coefficients} is defined (in \cite{BZ12}) to be a family of numbers
$\{c_S^n:n\in \mathbb{N}, S\subset n^*\}$
satisfying, for all $n\in \mathbb{N}$ and $S\subset n^*$,
 $c_S^n\ge 0$,
 $\sum_{S\subset n^* }c_S^n=1$, and
 $c_{\setcomp{S}}^n=c_S^n$.
Some examples of systems of coefficients are
$ c_S^n={1}/{2^n}$ (uniform),
 $c_S^n= {1}/[n+1)C({n},{|S|})]$ (neural complexity, $C(n,k)$ are the binomial coefficients), and
$ c_S^n={1}/[{2}\left(p^{|S|}(1-p)^{|\setcomp{S}|}+(1-p)^{|S|}p^{|\setcomp{S}|}\right)]$ for  fixed $0<p<1$ ($p$-symmetric).

Given a system of coefficients $c_S^n$ and a finite set of random variables $\{\mathbf{x}_i: i\in n^* \}$, for each $S\subset n^*$ let $\mathbf{x}_S:=\{\mathbf{x}_i:i\in S\}$. The corresponding {\em mutual information functional} $\mathcal{I}^c$ is defined by
\begin{equation}
\mathcal{I}^c(\mathbf{x}):=\sum_{S\subset n^* }c_S^n MI(\mathbf{x}_S,\mathbf{x}_{\setcomp{S}})=\sum_{S\subset n^*}c_S^n\left[H(\mathbf{x}_S)+H\left(\mathbf{x}_{\setcomp{S}}\right)-H\left(\mathbf{x}_S,\mathbf{x}_{\setcomp{S}}\right)\right].
\end{equation}
An \emph{intricacy} is a mutual information functional satisfying
\begin{enumerate}[(1)]
\item \emph{exchangeability}: if $n,m\in{\mathbb{N}}$ and $\phi:n^* \rightarrow m^*$ is a bijection, then $\mathcal{I}^c(\mathbf{x})=\mathcal{I}^c(\mathbf{y})$ for any $\mathbf{x}:=\{\mathbf{x}_i: {i\in n^* }\}$, $\mathbf{y}:=\{\mathbf{x}_{\phi^{-1}(j)}:{j\in m^*}\}$;
\item \emph{weak additivity}: $\mathcal{I}^c(\mathbf{x},\mathbf{y})=\mathcal{I}^c(\mathbf{x})+\mathcal{I}^c(\mathbf{y})$ for any two independent systems $\{\mathbf{x}_i:{i\in n^* }\}$, $\{\mathbf{y}_j:{j\in m^*}\}$.
\end{enumerate}
The following result from~\cite{BZ12} characterizes systems of coefficients that generate intricacies.  A probability measure $\lambda$ on $[0,1]$ is {\em symmetric} if $\int_{[0,1]}f(x)\lambda(dx)=\int_{[0,1]}f(1-x)\lambda(dx)$ for all bounded measurable functions $f$ on $[0,1]$.
\begin{theorem}\label{coeffprop}
Let $c_S^n$ be a system of coefficients and $\mathcal{I}^c$ the associated mutual information functional.
\begin{enumerate}[{\normalfont}]
\item $\mathcal{I}^c$ is an intricacy if and only if there exists a symmetric probability measure $\lambda_c$ on $[0,1]$ such that for all $S\subset n^*$,
\be\label{eq:intricacyweights}
c_S^n=\int_{[0,1]}x^{|S|}(1-x)^{n-|S|}\lambda_c(dx).
\en
\item The measure $\lambda_c$ is uniquely determined by $\mathcal{I}^c$. Moreover $\mathcal{I}^c$ is non-null, i.e. there exists some nonzero $c_S^n$ for $S\not\in\{\emptyset, n^*\}$ if and only if $\lambda_c\{(0,1)\}>0$. In this case $c_S^n>0$ for all $S\subset n^*$, $S\not\in\{\emptyset, n^*\}$.
\item For the neural complexity weights we have
\be
c_S^n=\frac{1}{n+1}\frac{1}{\binom{n}{|S|}}=\int_{[0,1]}x^{|S|}(1-x)^{n-|S|}dx \text{ for all } S\subset n^*,
\en
i.e., $\lambda_c$ is Lebesgue measure on $[0,1]$ and neural complexity is an intricacy.
\end{enumerate}
\end{theorem}

We now formulate definitions of topological intricacy and topological average sample complexity, based on the  definition of topological entropy given by Adler, Konheim, and McAndrew in terms of open covers \cite{AdlerKonheimMcAndrew1965}. We could just as well use the definition of Bowen \cite{Bowen1971}, 
and do so below in (\ref{sec:bowendef}) and 
for the generalization to average sample pressure in Section  \ref{asp}.  %Recall that $n^*=\{0,1,\dots, n-1\}$. Given a topological system $(X,T)$ and an open cover $\mathscr{U}$, for each subset $S\subset n^*$ let
%\be
%\mathscr{U}_S=\bigvee_{i\in S} T^{-i}\mathscr{U}.
%\en
%Recall that $N(\mathscr{U})$ denotes the minimum cardinality of the subcovers of $\mathscr{U}$.
\begin{definition}\label{def:TopInt}%[Topological Intricacy]
Let $T:X\rightarrow X$ be a continuous map on a compact Hausdorff space $X$, let  $\mathscr{U}$ be an open cover of $X$, and let $c_S^n$ be a system of coefficients as defined above. Define the \emph{topological intricacy of $T$ with respect to the open cover} $\mathscr{U}$ to be
\begin{equation}\label{topintdef}
\Int(X,\mathscr{U},T):=\lim_{n\rightarrow\infty}\frac{1}{n}\sum_{S\subset n^* }c_S^n\log\left( \frac{N(\mathscr{U}_S)N(\mathscr{U}_{\setcomp{S}})}{N(\mathscr{U}_{n^*})}\right).
\end{equation}
\end{definition}
\noindent We will see later that this limit exists.

Next we define the topological average sample complexity. Note that
\be
\begin{aligned}
\frac{1}{n}\sum_{S\subset n^* }c_S^n\log\left( \frac{N(\mathscr{U}_S)N(\mathscr{U}_{\setcomp{S}})}{N(\mathscr{U}_{n^*})}\right)&=\frac{1}{n}\sum_{S\subset n^* }\left(c_S^n\log N(\mathscr{U}_S)+c_S^n\log N(\mathscr{U}_{\setcomp{S}})-c_S^n\log N(\mathscr{U}_{n^*})\right)\\
&=2\left(\frac{1}{n}\sum_{S\subset n^* } c_S^n\log N(\mathscr{U}_S)\right)-\frac{1}{n}\log N(\mathscr{U}_{n^*})
\end{aligned}
\en
and
\be
\lim_{n\rightarrow\infty}\frac{1}{n}\log N(\mathscr{U}_{n^*})=\inf_n\frac{1}{n}\log N(\mathscr{U}_{n^*})=\htop(X,\mathscr{U},T),
\en
the ordinary topological entropy of $T$ with respect to the open cover $\mathscr{U}$. Thus, in order to calculate intricacy we must find
\be
\lim_{n\rightarrow\infty}\frac{1}{n}\sum_{S\subset n^* } c_S^n\log N(\mathscr{U}_S).
\en
Since this quantity is interesting on its own, we make the following definition.
\begin{definition}\label{def:TopAsc}%[Average Sample Complexity]
Let $T:X\rightarrow X$ be a continuous map on a compact Hausdorff space $X$, let $\mathscr{U}$ be an open cover of $X$ and let $c_S^n$ be a system of coefficients. 
The \emph{topological average sample complexity of $T$ with respect to the open cover $\mathscr{U}$} is defined to be
\begin{equation}\label{topascdef}
\Asc(X,\mathscr{U},T):=\lim_{n\rightarrow\infty}\frac{1}{n}\sum_{S\subset n^* } c_S^n\log N(\mathscr{U}_S).
\end{equation}
%Then, define the \emph{average sample complexity} as
%\begin{equation}
%\Asc(X,T):=\sup_{\mathscr{U}} \Asc(X,\mathscr{U},T)
%\end{equation}
\end{definition}
Thus 
\be\label{Eq:AscInt}
\Int(X\mathscr U,T) = 2 \Asc(X, \mathscr U,T) -\htop(X,\mathscr U,T).
\en

Suppose $S=\{s_0,\dots, s_{|S|-1}\}$ with $s_0<s_1<\cdots<s_{|S|-1}$. If we let $S'=\{0,s_1-s_0,\dots, s_{|S|-1}-s_0\}$ then 
\be
N(\mathscr{U}_{S'})=N(T^{s_0}\mathscr{U}_{S})=N(\mathscr{U}_{S}).
\en
Thus, when averaging $\log N(\mathscr{U}_S)$ over all subsets $S\subset n^*$ we end up counting the contribution from some subsets many times. If we restrict to subsets $S\subset n^*$ such that $0\in S$, then we count each configuration only once. This leads to the next definition, where we are concerned only with the configuration that a subset $S\subset n^*$ exhibits. 
\begin{definition}%[Average Configuration Complexity]
	Let $T:X\rightarrow X$ be a continuous map on a compact Hausdorff space $X$, let $\mathscr{U}$ be an open cover of $X$, and let $c_S^n$ be a system of coefficients. 
	The \emph{average configuration complexity} of $T$ with respect to the open cover $\mathscr{U}$ is
	\begin{equation}\label{acc1}
	\Acc(X,\mathscr{U},T):=\lim_{n\rightarrow\infty}\frac{1}{n}\sum_{\substack{S\subset n^*\\ 0\in S }} c_S^n\log N(\mathscr{U}_S).
	\end{equation}
\end{definition}
\begin{proposition}\label{accascprop1}
	Let $(X,T)$ be a topological dynamical system and fix the system of coefficients $c_S^n=2^{-n}$. Then for any open cover $\mathscr{U}$ of $X$, 
	\be
	\Acc(X,\mathscr{U},T)=\frac{1}{2}\Asc(X,\mathscr{U},T).
	\en
\end{proposition}
\begin{proof}
	\begin{align*}
	\Asc(X,\mathscr{U},T)
	&=\lim_{n\rightarrow\infty}\frac{1}{n}\frac{1}{2^n}\sum_{\substack{S\subset n^*\\ 0\in S }}\log N(\mathscr{U}_S)+\lim_{n\rightarrow\infty}\frac{1}{n}\frac{1}{2^n}\sum_{\substack{S\subset n^*\\ 0\not\in S }}\log N(\mathscr{U}_S)\\
	&=\lim_{n\rightarrow\infty}\frac{1}{n}\frac{1}{2^n}\sum_{\substack{S\subset n^*\\ 0\in S }}\log N(\mathscr{U}_S)+\lim_{n\rightarrow\infty}\frac{1}{n}\frac{1}{2^n}\sum_{S\subset (n-1)^*}\log N(\mathscr{U}_S)\\
	&=\Acc(X,\mathscr{U},T)+\lim_{n\rightarrow\infty}\frac{1}{2}\left(\frac{n-1}{n}\right)\left[\frac{1}{n-1}\frac{1}{2^{n-1}}\sum_{S\subset (n-1)^*}\log N(\mathscr{U}_S)\right]\\
	&=\Acc(X,\mathscr{U},T)+\frac{1}{2}\Asc(X,\mathscr{U},T).
	\end{align*}
\end{proof}

We also consider the average sample complexity and intricacy as functions of $n$. 
\begin{definition}\label{def:ascintfunctions}
Let $(X,T)$ be a topological dynamical system, $\mathscr{U}$ an open cover of $X$, and $c_S^n$ a system of coefficients. The \emph{topological average sample complexity function of $T$ with respect to the open cover $\mathscr{U}$} is defined by
\be
\Asc(X,\mathscr{U},T,n)=\frac{1}{n}\sum_{S\subset n^*}c_S^n\log N(\mathscr{U}_S).
\en
The \emph{topological intricacy function of $T$ with respect to the open cover $\mathscr{U}$} is defined by
\be
\Int(X,\mathscr{U},T,n)=\frac{1}{n}\sum_{S\subset n^*}c_S^n\log\left(\frac{N(\mathscr{U}_S)N(\mathscr{U}_{\setcomp{S}})}{N(\mathscr{U}_{n^*})}\right).
\en
When the context is clear we will sometimes write these as $\Asc(n)$ and $\Int(n)$.
\end{definition}

\begin{rem}
	Suppose that $(X,\sigma)$ is a subshift, $\mathscr U=\mathscr U_0$ is the standard time-$0$ cover (and partition) by cylinder sets determined by the initial symbol, and $S \subset n^*$. Then $N(\mathscr U_S)$ is the number of different words of length $|S|$ seen at the places in $S$ among all sequences in $X$.
\end{rem}

%\subsection{Preliminary results}
In order to show that the limits in Equations~\ref{topintdef} and \ref{topascdef} exist, we show that $ b_n:=\sum_{S\subset n^* } c_S^n\log N(\mathscr{U}_S)$ is subadditive for the class of systems of coefficients that define an intricacy functional as in Theorem~\ref{coeffprop}.

\begin{theorem}\label{syscoeffsubadd}
 Let $c_S^n$ be a system of coefficients and define
$b_n:=\sum_{S\subset n^* }c_S^n\log N(\mathscr{U}_S)$.
Then $b_{n+m}\le b_n+b_m$ for all $n,m\in\mathbb{N}$.
\end{theorem}
\begin{proof}
Let $S\subset (n+m)^*$ and define $U(S)=S\cap n^*$ and $V(S)=S\cap [(n+m)^*\setminus n^*]$. We see that
\be
N(\mathscr{U}_S)\le N\left(\mathscr{U}_{U(S)}\right)N\left(\mathscr{U}_{V(S)}\right),
\en
so
\be
\sum_{S\subset (n+m)^*}c_S^{n+m}\log N(\mathscr{U}_S)\le \sum_{S\subset (n+m)^*}c_S^{n+m}\log N(\mathscr{U}_{U(S)})+\sum_{S\subset (n+m)^*}c_S^{n+m}\log N(\mathscr{U}_{V(S)}).
\en
Abbreviate $U(S)=U$ and $V(S)=V$. For $W\subset m^*$ let $W+n=\{w+n:w\in W\}$. Note that each $W\subset m^*$ corresponds uniquely to $W+n=V$, and for the corresponding sets $W$ and $V$ $N(\mathscr{U}_W)=N(\mathscr{U}_{W+n})=N(\mathscr{U}_V)$. Then for all $n,m\in\mathbb{N}$
\begin{align*}
b_{n+m}&=\sum_{S\subset (n+m)^*}\int_{[0,1]}x^{|S|}(1-x)^{n+m-|S|}\lambda_c(dx)\log N(\mathscr{U}_S)\\
&=\int_{[0,1]}\sum_{S\subset (n+m)^*}x^{|S|}(1-x)^{n+m-|S|}\log N(\mathscr{U}_S)\lambda_c(dx)\\
&\le\int_{[0,1]}\left(\sum_{U\subset n^* }x^{|U|}(1-x)^{n-|U|}\log N(\mathscr{U}_U)\right. \left.+\sum_{W\subset m^*}x^{|W|}(1-x)^{m-|W|}\log N(\mathscr{U}_W)\right)\lambda_c(dx)\\
&=\sum_{U\subset n^* }\int_{[0,1]}x^{|U|}(1-x)^{n-|U|}\lambda_c(dx)\log N(\mathscr{U}_U)+\sum_{W\subset m^*}\int_{[0,1]}x^{|W|}(1-x)^{m-|W|}\lambda_c(dx)\log N(\mathscr{U}_W)\\
&=b_n+b_m.
\end{align*}
\end{proof}
\begin{corollary}
If $c_S^n$ is a system of coefficients, then the limits in the definitions of $\Asc(X,\mathscr{U},T)$ and $\Int(X,\mathscr{U},T)$ (Definitions \ref{topintdef} and \ref{topascdef}) exist and 
\be
\Asc(X,\mathscr{U},T)=\inf_n\frac{1}{n}\sum_{S\subset n^*}c_S^n\log N(\mathscr{U}_S).
\en
\end{corollary}
\begin{proof}
This follows from Fekete's Lemma \cite{fekete1923verteilung} and Theorem~\ref{syscoeffsubadd}.
\end{proof}
\begin{proposition}\label{htopboundprop}
 For each open cover $\mathscr{U}$, $\Asc(X,\mathscr{U},T)\le \htop(X,T)$, and hence 
 \be
 \Int(X,\mathscr{U},T)\le \htop(X,\mathscr{U},T)\le \htop(X,T).
 \en
\end{proposition}
\begin{proof}
For every finite open cover $\mathscr U$ and every subset $S\subset n^*$, $N(\mathscr{U}_S)\le N(\mathscr{U}_{n^*})$.
\end{proof}

\subsection[Definitions based on Bowen's definition of entropy]{Definitions of intricacy and average sample complexity based on Bowen's definition of entropy}\label{sec:bowendef}

\begin{definition}
	Given a dynamical system $(X,T)$, where $d$ is a metric on $X$, and a subset $S\subset n^*$, a set $E\subset X$ is \emph{$(S,\varepsilon)$ spanning} if for each $x\in X$ there is $y\in E$ with $d(T^{s_i}x,T^{s_i}y)\le\varepsilon$ for all $i=0,\dots, |S|-1$. Let $r(S,\varepsilon)$ be the minimum cardinality of an $(S,\varepsilon)$ spanning set of $X$.
	\begin{definition}
		Fix a system of coefficients $c_S^n$. For each $\varepsilon>0$ define the \emph{$\varepsilon$-topological intricacy of $(X,T)$} by
		\begin{equation}
		\Int_\varepsilon(X,T)=\limsup_{n\rightarrow\infty}\frac{1}{n}\sum_{S\subset n^*}c_S^n\log\left(\frac{r(S,\varepsilon)r(\setcomp{S},\varepsilon)}{r(n^*,\varepsilon)}\right),
		\end{equation}
		the \emph{$\varepsilon$-topological average sample complexity of $(X,T)$} by
		\begin{equation}
		\Asc_\varepsilon(X,T)=\limsup_{n\rightarrow\infty}\frac{1}{n}\sum_{S\subset n^*}c_S^n\log r(S,\varepsilon),
		\end{equation}
		and the \emph{$\varepsilon$-topological average configuration complexity of $(X,T)$} by
		\begin{equation}
		\Acc_\varepsilon(X,T)=\limsup_{n\rightarrow\infty}\frac{1}{n}\sum_{\substack{S\subset n^*\\0\in S}}c_S^n\log r(S,\varepsilon).
		\end{equation}
	\end{definition}
	We also give the definitions of topological intricacy, topological average sample complexity, and topological average configuration complexity in terms of $(S,\varepsilon)$ separated sets. A set $E\subset X$ is {\em$(S,\varepsilon)$ separated} if for each pair of distinct points $x,y\in E$, $d(T^{s_i}x,T^{s_i}y)>\varepsilon$ for some $i=0,\dots, |S|-1$. Let $s(S,\varepsilon)$ be the maximum cardinality of a set $E\subset X$ such that $E$ is $(S,\varepsilon)$ separated. Fix a system of coefficients $c_S^n$. For each $\varepsilon>0$ define the \emph{$(\varepsilon$-topological intricacy$)'$ of $(X,T)$} by
	\begin{equation}
	\Int'_\varepsilon(X,T)=\limsup_{n\rightarrow\infty}\frac{1}{n}\sum_{S\subset n^*}c_S^n\log\left(\frac{s(S,\varepsilon)s(\setcomp{S},\varepsilon)}{s(n^*,\varepsilon)}\right),
	\end{equation}
	the \emph{$(\varepsilon$-topological average sample complexity$)'$ of $(X,T)$} by
	\begin{equation}
	\Asc'_\varepsilon(X,T)=\limsup_{n\rightarrow\infty}\frac{1}{n}\sum_{S\subset n^*}c_S^n\log s(S,\varepsilon),
	\end{equation}
	and the \emph{$(\varepsilon$-topological average configuration complexity$)'$ of $(X,T)$} by
	\begin{equation}
	\Acc'_\varepsilon(X,T)=\limsup_{n\rightarrow\infty}\frac{1}{n}\sum_{\substack{S\subset n^*\\0\in S}}c_S^n\log s(S,\varepsilon).
	\end{equation}
\end{definition}
We use different notations for the definitions based on $(S,\varepsilon)$ separating sets and those based on $(S,\varepsilon)$ spanning sets because, in general, for a given $\varepsilon$ the two definitions may not be equivalent.
But the limits as $\varepsilon \to 0$ and the suprema over open covers $\mathscr U$ are the same, and similarly for the pressure versions: see 
Theorem \ref{asppresssupthm}, Corollary \ref{asppresssupcor}, Theorem \ref{aspepslimthm}, and Corollary \ref{ascepslimcor}.

%\subsection*{Relating two sets of definition of topological intricacy and average sample complexity}
%\begin{proposition}
%Let $T:X\rightarrow X$ be a continuous map on a compact metric space $X$. Given $\varepsilon>0$, we can find an open cover $\mathscr{U_{\varepsilon}}$ of $X$ such that
%\begin{equation}
%\Asc_\varepsilon(X,T)=\Asc(X,\mathscr{U_{\varepsilon}},T).
%\end{equation}
%Similarly, given an open cover $\mathscr{U}$ of $X$, we can find an $\varepsilon(\mathscr{U})>0$ such that
%\begin{equation}
%\Asc(X,\mathscr{U},T)=\Asc_{\varepsilon(\mathscr{U})}(X,T).
%\end{equation}
%\end{proposition}

\section{The supremum over open covers equals topological entropy}\label{sec:supequalsent}
To calculate the topological entropy of a system using the Adler, Konheim, and McAndrew definition with open covers, one finds the supremum over all open covers, $\mathscr{U}$, of $\htop(X,\mathscr{U},T)$, and this defines an invariant for topological conjugacy. The following theorem shows that, with $c_S^n=2^{-n}$ for all $S$, if suprema over all open covers are taken in calculating topological average sample complexity then we get just the usual topological entropy.
 See (\ref{finobs}) below for further comments about this and Theorem \ref{mtheoreticasctoent}. 
 Therefore we are motivated to compute and study intricacy and average sample complexity for specific open covers, and also as functions of $n$ (see Definition \ref{def:ascintfunctions}) before taking the limits in Definitions \ref{def:TopInt} and \ref{def:TopAsc}.

\begin{theorem}\label{asctoentthm}
Let $(X,T)$ be a topological dynamical system and fix a system of coefficients $c_S^n=2^{-n}$. Then
\be
\sup_{\mathscr{U}}\Asc(X,\mathscr{U},T)=\htop(X,T) \quad\text{ and }\quad {\sup_{\mathscr{U}}\Int(X,\mathscr{U},T)=\htop(X,T)} .
\en

\end{theorem}
The idea of the proof is 
{that for most subsets $S \subset n^*$, $S+k^*$ is large, so that for any open cover $\mathscr{U}$ we have $\log N(\mathscr{U}_{S+k^*})/n=\log N({(\mathscr{U}_{k^*})}_S)/n$ close to $\log N(\mathscr{U}_{n^*})/n\approx \htop(X,\mathscr{U},T)$.}
Averaging over $S$, it then follows that $\Asc(X,\mathscr{U}_{k^*},T)$ is close to $\htop(X,\mathscr{U},T)$, and taking the suprema over $\mathscr{U}$ concludes the argument. 

In the following Lemma, for a given $n,k\in\mathbb{N}$ with $k<n$, we break the interval $n^*$ into intervals $K_i$ of length $k/2$; then for $S\subset n^*$ and $s\in S$, we know that if $s\in K_i$ then the interval $s+k^*:=\{s,s+1,\dots, s+k-1\}$ will be long enough to contain $K_{i+1}$. This helps us count the subsets $S$ for which $S+k^*:=\{s_0+k^*,\dots,s_{|S|-1}+k^*\}$ is large enough so that $\log N(S+k^*)/n$ is a good approximation to the topological entropy of $T$ with respect to $\mathscr{U}$.  

\begin{lemma}\label{goodsetlemma}
Given $n,k\in\mathbb{N}$ such that $k$ is even and less than $n$, break $n^*$ into $\ceil{2n/k}-1$ sets of $k/2$ consecutive integers and one set of at most $k/2$ consecutive integers, by defining
\begin{equation}\label{keq}
K_i=\left\{\frac{i-1}{2}k,\dots,\frac{i}{2}k-1\right\} \text{ for } i=1,2,\dots, \ceil{2n/k}-1, K_{\ceil{2n/k}}=\left\{\frac{\ceil{2n/k}-1}{2}k,\dots,n-1\right\}.
\end{equation}
%\begin{figure}
%\begin{center}
%\begin{tikzpicture}
%\draw(-6,0)--(6,0);
%\draw(-6,.25)--(-6,-.25);
%\draw(-5,.25)--(-5,-.25);
%\draw(-4,.25)--(-4,-.25);
%\draw(-3,.25)--(-3,-.25);
%\draw(-2,.25)--(-2,-.25);
%\draw(-1,.25)--(-1,-.25);
%\draw(4,.25)--(4,-.25);
%\draw(5,.25)--(5,-.25);
%\draw(6,.25)--(6,-.25);
%\node at (-6,-.6){$0$};
%\node at (-5,-.6){$\frac{k}{2}$};
%\node at (-4,-.6){$k$};
%\node at (-3,-.6){$\frac{3}{2}k$};
%\node at (-2,-.6){$2k$};
%\node at (-1,-.6){$\frac{5}{2}k$};
%\node at (1.6,-.6){$\cdots\quad\cdots\quad\cdots$};
%\node at (3.75,-.6){$n-k$};
%\node at (5,-.6){$n-\frac{k}{2}$};
%\node at (6.4,-.55){$n-1$};
%\node at (-5.4,.4){$K_1$};
%\node at (-4.4,.4){$K_2$};
%\node at (-3.4,.4){$K_3$};
%\node at (-2.4,.4){$K_4$};
%\node at (-1.4,.4){$K_5$};
%%\node at (4.4,.4){$K_{\frac{2n}{k}-1}$};
%\node at (5.5,.4){$K_{\frac{2n}{k}}$};
%
%
%\end{tikzpicture}
%\end{center}
%\caption{Visualization of sets $K_i$.}
%\end{figure}
 For each subset $S\subset n^*$, let 
 \be
B(S)=\card\{i:S\cap K_i\ne\emptyset\}
\en
denote the number of intervals $K_i, i=1,\dots,\ceil{2n/k}$, that contain at least one element of $S$. Given $0<\varepsilon<1$, define $\mathcal{B}$, the set of ``bad" subsets $S\subset n^*$, by
\be
\mathcal{B}=\mathcal{B}(n,k,\varepsilon)=\{S\subset n^*:B(S)\le (2n/k)(1-\varepsilon)\}.
\en
 Then there exists an even $k\in\mathbb{N}$ such that
\begin{equation}\label{goodsetlimn}
\lim_{n\rightarrow\infty}\frac{\textnormal{card}\left(\mathcal{B}(n,k,\varepsilon)\right)}{2^n}=0.
\end{equation}

\end{lemma}
\begin{proof}
	Note that $S\in\mathcal{B}$ if it intersects at most $\floor{(1-\varepsilon)(2n/k)}$ of the $\ceil{2n/k}$ sets $K_i$.
To create any subset $S\in \mathcal{B}$ we choose $\ceil{(2n/k)\varepsilon}$ intervals $K_i$ for $S$ to not intersect and then pick a subset (could be empty) from the rest of the $\floor{2n/k(1-\varepsilon)}$ intervals $K_i$ to intersect $S$. 
The same subset can be produced this way many times. 
Thus
\begin{align*}
\card(\mathcal{B})&\leq\binom{\ceil{2n/k}}{\ceil{2n\varepsilon/k}}\left(2^{k/2}\right)^{\floor{(2n/k)(1-\varepsilon)}}
=\binom{\ceil{2n/k}}{\ceil{2n\varepsilon/k}}2^{\floor{n(1-\varepsilon)}}.
\end{align*}
According to Stirling's approximation, there is a constant $c$ such that
\be
{\binom{m}{m\varepsilon}} \leq \frac{c} {\sqrt{m}}{\varepsilon^{-m\varepsilon}(1-\varepsilon)^{-m(1-\varepsilon)}}
\en
for all $m$.
This implies
\be
\begin{aligned}
\lim_{n\rightarrow\infty}\frac{\card(\mathcal{B})}{2^n}&\leq \lim_{n\rightarrow\infty}
\frac{c}{ \sqrt{\ceil{2n/k}}}
2^{-n\varepsilon}\varepsilon^{-(2n/k)\varepsilon}(1-\varepsilon)^{-(2n/k)(1-\varepsilon)}\\
&=\lim_{n\rightarrow\infty}
\frac{c}{ \sqrt{\ceil{2n/k}}}
\left(\frac{1}{2^\varepsilon\varepsilon^{(2/k)\varepsilon}(1-\varepsilon)^{(2/k)(1-\varepsilon)}}\right)^n.
\end{aligned}
\en
We will show $\lim_{n\rightarrow\infty}(\card(\mathcal{B})/2^n)=0$ by showing that for each $\varepsilon>0$ we can find a $k$ such that $2^\varepsilon\varepsilon^{(2/k)\varepsilon}(1-\varepsilon)^{(2/k)(1-\varepsilon)}>1$. Denote the binary entropy function by
\be H(x)=-x\log x-(1-x)\log(1-x).\en To show $2^\varepsilon\varepsilon^{(2/k)\varepsilon}(1-\varepsilon)^{(2/k)(1-\varepsilon)}>1$, we take the logarithm of both sides of the inequality and show 
\begin{equation}\label{kepseq1}
\varepsilon\log 2+\frac{2}{k}\varepsilon\log \varepsilon+\frac{2}{k}(1-\varepsilon)\log (1-\varepsilon)>0.
\end{equation}
This would follow from 
\begin{equation}\label{kepseq2}
k>\frac{2}{\varepsilon\log 2}H(\varepsilon).
\end{equation}
By basic calculus $H(\varepsilon)\le\log (2)$; thus if $k>2/\varepsilon$ Equation~\ref{kepseq2} is satisfied and therefore Equation~\ref{kepseq1} is satisfied.
\end{proof}

Next we note some properties of $N(\mathscr{U}_S)$ that are needed for the proof of Theorem~\ref{asctoentthm}. To simplify notation we sometimes replace $N(\mathscr{U}_S)$ by $N(S)$ when the context is clear.

 \begin{lemma}\label{Nsetstoplemma}
Let $(X,T)$ be a topological dynamical system and $\mathscr{U}$ an open cover of $X$. Given  $n\in\mathbb{N}$ and $S\subset n^*$ the following properties hold:
\begin{enumerate}[\normalfont 1.]
\item $N((\mathscr{U}_{k^*})_S)=N(\mathscr{U}_{S+k^*})$.
\item Given $S_1,S_2,\dots, S_m\subset n^*$,  $\log N(\bigcup_iS_i)\le\sum_{i}\log N({S_i})$.
\end{enumerate}
\end{lemma}
\begin{proof}
\begin{enumerate}
\item This follows from the fact that
\be
(\mathscr{U}_{k^*})_S=\bigvee_{i\in S}T^{-i}\mathscr{U}_{k^*}=\bigvee_{i\in S+k^*}T^{-i}\mathscr{U}=\mathscr{U}_{S+k^*}.
\en
\item We show this for two sets $S_1$ and $S_2$ and use induction. Because $N(\mathscr{U}\vee\mathscr{V})\le N(\mathscr{U})N(\mathscr{V})$,
\be
N({S_1\cup S_2})=N(\mathscr{U}_{S_1}\vee\mathscr{U}_{S_2})\le N({S_1})N({S_2}).
\en
\end{enumerate}
\end{proof}
\begin{proof}[Proof of Theorem~\ref{asctoentthm}]
Recall that $\htop(X,\mathscr{U},T)=\lim_{n\rightarrow\infty}\log N(\mathscr{U}_{n^*})/n$ and $\htop(X,T)=\sup_{\mathscr{U}}h(X,\mathscr{U},T)$. We prove the statement by showing for each open cover $\mathscr{U}$ of $X$,
\begin{equation}\label{presupasch}
\lim_{k\rightarrow\infty}\Asc(X,\mathscr{U}_{k^*},T)=\htop(X,\mathscr{U},T).
\end{equation}

Recall that by Proposition~\ref{htopboundprop} for every cover $\mathscr{U}$ of $X$, $\Asc(X,\mathscr{U},T)\le\htop(X,\mathscr{U},T)$. We would like to show that 
\be
\lim_{k\rightarrow\infty}\lim_{n\rightarrow\infty}\frac{1}{n}\frac{1}{2^n}\sum_{S\subset n^*}\log N(S+k^*)=\lim_{n\rightarrow\infty}\frac{1}{n}\log N(n^*).
\en
Let $0<\varepsilon<1$ be given. By Fekete's Lemma, $\htop(X,\mathscr{U},T)=\inf_k\left(\log N(k^*)/k\right)$. Thus there is a $k_0$ such that for every $k>k_0$,
\begin{equation}\label{htoplimeq}
0\le \frac{\log N(k^*)}{k}-\htop(X,\mathscr{U},T)<\varepsilon.
\end{equation}
Let $k>\max\{2k_0,2/\varepsilon\}$ and let $n>k$. Form the family of bad sets $\mathcal{B}(n,k,\varepsilon)$ as in the statement of Lemma~\ref{goodsetlemma} and let the sets $K_i$ be as in Equation~\ref{keq}.  The main idea behind the construction of the intervals $K_i$ is that for $S\subset n^*$ and $s\in S$, if $s\in K_i$ then $K_{i+1}\subset S+k^*$. Suppose $S\not\in\mathcal{B}$. Then $S$ intersects at least $(2n/k)(1-\varepsilon)$ of the sets $K_i$ so we have $\card(S+k^*)\ge (k/2)(2n/k)(1-\varepsilon)=n(1-\varepsilon)$. 

Let $E=n^* \setminus (S+k^*)$, and notice that 
\be
N(E) \leq |\mathscr{U}|^{|E|} \leq |\mathscr{U}|^{n \varepsilon}.
	\en
	By Lemma \ref{Nsetstoplemma}, 
	\be
	\log N(S+k^*) \geq \log N(n^*) - \log N(E) \geq \log N(n^*) - n \varepsilon \log |\mathscr{U}|,
\en
so that if $S \notin \mathcal B$
\be
\frac{\log N(S+k^*)}{n} \geq \htop(X,\mathscr{U},T) - \varepsilon \log |\mathscr{U}|.
\en

We then conclude that for any $0<\varepsilon<1$ we can find $k$ such that for all $n$ large enough that 
 $|\mathcal B^c|/2^n \geq 1 - \varepsilon$,
\be\label{ascboundeq}
\begin{aligned}
\frac{1}{n}\frac{1}{2^n}\sum_{S\subset n^*}\log N(S+k^*)& \geq \frac{|\setcomp{\mathcal{B}}|}{2^n}\frac{1}{|\setcomp{\mathcal{B}}|}\sum_{S\in\setcomp{\mathcal{B}}}\frac{\log N(S+k^*)}{n}\\
&\ge (1-\varepsilon)\frac{1}{|\setcomp{\mathcal{B}}|}\sum_{S\in\setcomp{\mathcal{B}}}(\htop(X,\mathscr{U},T)-\varepsilon \log  |\mathscr{U}|)\\
&= (1-\varepsilon)\frac{1}{|\setcomp{\mathcal{B}}|}|\setcomp{\mathcal{B}}|(\htop(X,\mathscr{U},T)-\varepsilon \log |\mathscr{U}|)\\
&=\htop(X,\mathscr{U},T) - \varepsilon \htop (X,\mathscr{U},T)-(1-\varepsilon)\varepsilon
\log |\mathscr{U}|.
\end{aligned}
\en

Letting $n\rightarrow\infty$ and then $k\rightarrow\infty$ in Equation~\ref{ascboundeq} gives
\begin{equation}\label{presupaschm}
\begin{aligned}
\lim_{k\rightarrow\infty}\Asc(X,\mathscr{U}_k,T)&=\lim_{k\rightarrow\infty}\lim_{n\rightarrow\infty}\frac{1}{n}\frac{1}{2^n}\sum_{S\subset n^*}\log N(S+k^*)\\
&\ge
\htop(X,\mathscr{U},T) - \varepsilon \htop (X,\mathscr{U},T)-(1-\varepsilon)\varepsilon
\log |\mathscr{U}|,
\end{aligned}
\end{equation}
and hence, by Proposition~\ref{htopboundprop}, $\lim_{k\rightarrow\infty}\Asc(X,\mathscr{U}_k,T)=\htop(X,\mathscr{U},T)$. {Then take the supremum over all covers $\mathscr{U}$ of $X$ on both sides of Equation~\ref{presupasch}.}

{To prove the statement about $\Int$, take an increasing sequence of covers $\mathscr U^{(n)}$ with $\Asc(X,\mathscr U^{(n)}n,T) \nearrow \htop(X,T)$ and apply (\ref{Eq:AscInt}), noting that $\lim \htop(X,\mathscr U^{(n)},T) \leq \htop(X,T)$.}
\end{proof}

  %-------------------------------------------------------------------------------------------------------------------------

\section{Complexity calculations for shifts of finite type}

In this section we calculate the intricacy, average sample complexity, and average sample pressure for some shifts of finite type $X\subset \Sigma_r$.
 Unless otherwise noted, we will use the uniform system of coefficients $c_S^n=2^{-n}$ and open covers by rank $0$ cylinder sets. Recall that for a subset $S\subset n^*$, $N(S)$ counts the number of words seen at the places in $S$ over all sequences $x \in X$.

\begin{proposition}\label{nsprop1}
Let $X$ be a shift of finite type over the alphabet $\mathcal{A}$ with adjacency matrix $M$ such that $M^2>0$. Given $S\subset n^*$ denote the disjoint maximal subsets of consecutive integers that compose $S$ by $I_1,\dots, I_k$ with $|I_j|=t_j$ for  $t_j\in\mathbb{N}$. Then
\begin{equation}\label{prodeq}
N(S)=|\mathscr{L}_{t_1^*}(X)||\mathscr{L}_{t_2^*}(X)|\cdots|\mathscr{L}_{t_k^*}(X)|=N(t_1^*)N(t_2^*)\cdots N(t_k^*).
\end{equation}
In particular, for $\ell=1,2, \dots, n$,
\begin{equation}\label{sumeq}
\sum_{\substack{S\subset n^*\\\{n-\ell,n-\ell+1,\dots,n-1\}\in S\\n-\ell-1\not\in S}}\log N(S)=\sum_{S\subset (n-\ell-1)^*}\log\left( N(S) N(\ell^*)\right) 
\end{equation}
and
\begin{equation}\label{sumeq2}
\sum_{\substack{S\subset n^*\\n-1\not\in S}}\log N(S)=\sum_{S\subset (n-1)^*}\log N(S).
\end{equation}
\end{proposition}

\begin{theorem}\label{squarethm}
 Let $X$ be a shift of finite type with adjacency matrix $M$ such that $M^2>0$. Let $c_S^n=2^{-n}$ for all $S$. Then
 \begin{equation}\label{zrasca}
 \Asc(X,\mathscr{U}_0,\sigma)=\frac{1}{4}\sum_{k=1}^\infty\frac{\log |\mathscr{L}_{k^*}(X)|}{2^k}.
 \end{equation}
 \end{theorem}
 \begin{proof} 
We break the sum over all subsets $S\subset n^*$ in the definition of average sample complexity into the sum over those $S\subset n^*$ that contain $n-1$ and those that do not. Since $c_S^n$ has no dependence on $S$, the sum over $S\subset n^*$ that do not contain $n-1$ is equivalent to the sum over $S\subset (n-1)^*$. The sum over the sets containing $n-1$  is then broken into a sum over sets that contain $n-2$ and those that do not contain $n-2$. We simplify the sum over sets that do not contain $n-2$ using Proposition~\ref{nsprop1} and continue the process inductively.

We prove first that
\begin{equation}\label{eq:intrecform}
\frac{1}{n}\frac{1}{2^n}\sum_{S\subset n^* }\log N(S)=\frac{1}{n}\frac{1}{2^n}\log N(n^*) +\frac{1}{4n}\sum_{k=1}^{n-1}\frac{n-k+3}{2^k}\log N(k^*).
\end{equation}
Let
\be
a_n=\sum_{S\subset n^* }\log N(S)\quad\text{and}\quad \lambda_k=\log N(k^*).
\en
Using the process described above and Equations \ref{sumeq} and \ref{sumeq2}, it can be shown that for $n>1$
\be
a_n=\lambda_n+\sum_{k=1}^{n-1}\left(2^{n-k-1}\lambda_k+a_k\right).
\en
Therefore, for $n>1$
\begin{equation}
a_n-a_{n-1}=\lambda_n+\lambda_{n-2}+2\lambda_{n-3}+4\lambda_{n-4}+\cdots+2^{n-3}\lambda_1+a_{n-1},
\end{equation}
which gives
\be
\begin{aligned}\label{intrecform2}
a_1&=\lambda_1\text{ and}\\
a_n&=\lambda_n+2a_{n-1}+2^{n-2}\sum_{k=1}^{n-2}\frac{\lambda_k}{2^k}\text{ for }n>1.
\end{aligned}
\en

We use induction to prove Equation~\ref{eq:intrecform} by showing
\begin{equation}\label{intrecform3}
a_n=\lambda_n+2^{n-2}\sum_{k=1}^{n-1}\frac{n-k+3}{2^k}\lambda_k.
\end{equation}
%We see that $a_1=\lambda_1$ in Equation~\ref{intrecform3}. Now we assume \ref{intrecform3} for all $k<n$ and prove it for $n$. By the induction hypothesis, \ref{intrecform2}, and \ref{intrecform3},
%\begin{align*}
%a_n&=\lambda_n+2a_{n-1}+2^{n-2}\sum_{k=1}^{n-2}\frac{\lambda_k}{2^k}\\
%&=\lambda_n+2\left(\lambda_{n-1}+2^{n-3}\sum_{k=1}^{n-2}\frac{n-j+2}{2^j}\lambda_k\right)+2^{n-2}\sum_{k=1}^{n-2}\frac{\lambda_k}{2^k}\\
%&=\lambda_n+2\lambda_{n-1}+2^{n-2}\sum_{k=1}^{n-2}\frac{n-k+3}{2^j}\lambda_k\\
%&=\lambda_n+2^{n-2}\sum_{k=1}^{n-1}\frac{n-k+3}{2^k}\lambda_k.
%\end{align*}
Now we show 
\begin{equation}\label{unifint}
\Asc(X,\mathscr{U}_0,\sigma)=\frac{1}{n}\frac{1}{2^n}\sum_{S\subset n^* }\log N(S)=\frac{1}{4}\sum_{j=1}^{\infty}\frac{\log N(k^*)}{2^k},
\end{equation}
which would follow from
\be
\lim_{n\rightarrow\infty}\left(\frac{1}{n}\frac{1}{2^n}\log N(n^*) +\frac{1}{4n}\sum_{k=1}^{n-1}\frac{n-k+3}{2^k}\log N(k^*)\right)=\frac{1}{4}\sum_{k=1}^{\infty}\frac{\log N(k^*)}{2^k}.
\en
We know that $H(n)=\log N(n^*)/n$ converges to the topological entropy of $(X,\sigma)$, so 
\be
\lim_{n\rightarrow\infty}\frac{1}{n}\frac{1}{2^n}\log N(n^*) =0.
\en
Now 
\be
\sum_{k=1}^\infty\frac{(3-k)k\log|\mathcal{A}|}{2^k}
\en
converges and $N(k^*)\le |\mathcal{A}|^k$, so $\log N(k^*)\le k\log |\mathcal{A}|$. Thus,
\be
\lim_{n\rightarrow\infty}\frac{1}{4n}\sum_{k=1}^\infty\frac{3-k}{2^j}\log N(k^*)=0.
\en
\end{proof}

\begin{corollary}\label{intcor}
 Let $X$ be a shift of finite type with adjacency matrix $M$ such that $M^2>0$. Let $c_S^n=2^{-n}$ for all $S$. Then
 \begin{equation}\label{zrasci}
 \Int(X,\mathscr{U}_0,\sigma)=\frac{1}{2}\sum_{k=1}^\infty\frac{\log |\mathscr{L}_{k^*}(X)|}{2^k}-\htop(X,T).
 \end{equation}

\end{corollary}

\begin{corollary}
Two shifts of finite type, $X_1$ and $X_2$, that have positive square adjacency matrices 
and 
have the same complexity functions ($|\mathscr{L}_{n^*}(X_1)|=|\mathscr{L}_{n^*}(X_2)|$ for all $n\in\mathbb{N}$) 
have the same average sample complexity and intricacy of rank $0$ open covers using the uniform system of coefficients $c_S^n=2^{-n}$.   %and only if 
\end{corollary}

\begin{example}
The full $r$-shift has a positive square adjacency matrix and $N(k^*)=r^k$, so 
\begin{equation}
\Asc(\Sigma_r,\mathscr{U}_0,\sigma)=\frac{\log r}{2}.
\end{equation}
We also have $\Int(\Sigma_r,\mathscr{U}_0,\sigma)=2\Asc(\Sigma_r,\mathscr{U}_0,\sigma)-\htop(\Sigma_r,\sigma)$ and $\htop(\Sigma_r,\sigma)=\log r$, so we find
\begin{equation}
\Int(\Sigma_r,\mathscr{U}_0,\sigma)=0.
\end{equation}
This example shows that a completely independent (segregated) shift system has zero intricacy when it is taken over rank $0$ cylinder sets with the uniform system of coefficients.
\end{example}

%-------------------------------------------------------------------------------------------------------------------------------------
%							
%						INTERESTING EXAMPLES OF SFTS
%
%-------------------------------------------------------------------------------------------------------------------------------------

 Next we present intricacy and average sample complexity calculations for a few interesting shifts of finite type. $M$ is the adjacency matrix for each SFT and $\rho(M)$ is the smallest power for which $M$ is positive. The calculations are done using the uniform system of coefficients $c_S^n=2^{-n}$, the open covers are by rank $0$ cylinder sets, 
the computations were made using {\em Mathematica}, tables show values rounded to $3$ decimal places, and when applicable the sums in Equations~\ref{zrasca} and  \ref{zrasci} are computed using the first $20$ terms of the series.
 
 \begin{example}\label{1317example}
\begin{table}[h]
\begin{tabular}{@{}m{.05\textwidth}m{.18\textwidth}m{.05\textwidth}m{.25\textwidth}m{.08\textwidth}m{.08\textwidth}m{.08\textwidth}m{.08\textwidth}@{}}
\toprule
Label&\hspace{.08\textwidth}$M$ & $\rho(M)$
   &\ \hspace{.25in} \text{Graph} & \text{Entropy} & $H(10)$ & $\Asc(10)$ & $\Int(10)$
   \\
   \midrule
I&$\left(
\begin{array}{ccc}
 1 & 1 & 0 \\
 0 & 0 & 1 \\
 1 & 1 & 0 \\
\end{array}
\right) $& 3 &   \includegraphics[width=1in,keepaspectratio]{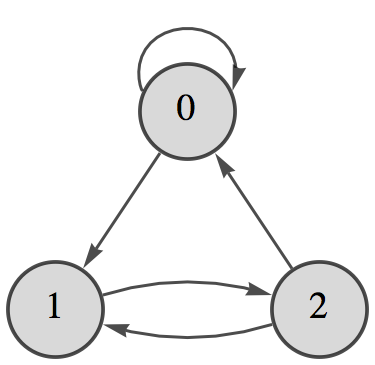}& 0.481 & 0.545 & 0.399 & 0.254 \\
II&$ \left(
\begin{array}{ccc}
 0 & 1 & 1 \\
 1 & 0 & 1 \\
 1 & 0 & 0 \\
\end{array}
\right)$ & 4 &  \includegraphics[width=1in,keepaspectratio]{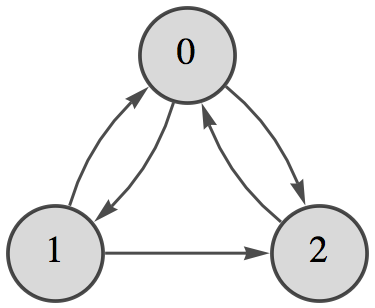} & 0.481 & 0.545 & 0.377 & 0.208 \\
\bottomrule

\end{tabular}
\caption[Two shifts of finite type with the same entropy]{Two shifts of finite type with the same entropy and complexity functions, but different average sample complexity and intricacy functions\label{1317table}}
\end{table}
In this example we compare two shifts of finite type that have the same entropy and complexity functions but different average sample complexity and intricacy functions: see Table \ref{1317table}. When looking at comparisons of $N(S)$ for each SFT for all $S\subset 4^*$ we can see where the differences occur in the average sample complexity and intricacy functions. For instance, the first SFT has $13$ words that appear at $\{0,1,3\}$, whereas the second SFT has $11$ words on those indices. 
Figure \ref{fig:AscPlot} shows part of the graphs of $\Asc (n)$ for these two systems. 
\begin{figure}
	\begin{center}
	\includegraphics[width=3.5in]{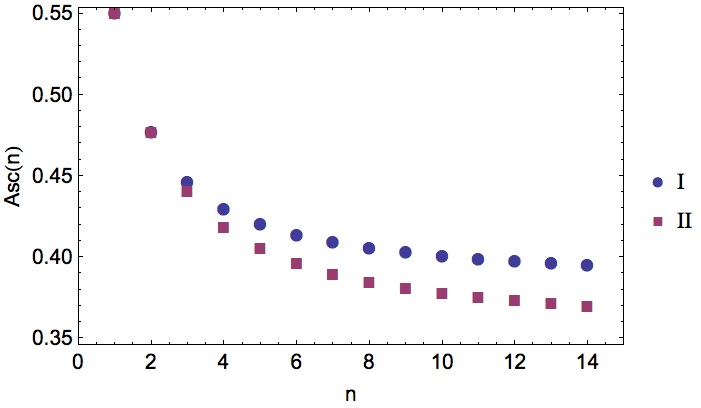}
	\caption{$\Asc (n)$ versus $n$ for the two SFTs in Table \ref{1317table}}\label{fig:AscPlot}
\end{center}
\end{figure}
Notice that the smallest power for which the adjacency matrix for the first SFT is positive is $3$, while it is $4$ for the second SFT. This gives us a clue as to what $\Asc(n)$ and $\Int(n)$ measure. 
Even though both SFTs have the same number of words of each length, the structure of these words is different. The words that appear in sequences for the first SFT are more complex in some sense because there is more freedom to build them.
%
%\begin{table}
%\be
%\begin{array}{@{}ccc@{}}
%\toprule
%&\multicolumn{2}{c}{N(S)}\\
%\cline{2-3}
% S\subset 4^*&\text{SFT 13}&\text{SFT 17} \\
% \midrule
% \{\} & 1 & 1 \\
% \{0\} & 3 & 3 \\
% \{1\} & 3 & 3 \\
% \{2\} & 3 & 3 \\
% \{3\} & 3 & 3 \\
% \{0,1\} & 5 & 5 \\
% \{0,2\} & 8 & 7 \\
% \{0,3\} & 9 & 8 \\
% \{1,2\} & 5 & 5 \\
% \{1,3\} & 8 & 7 \\
% \{2,3\} & 5 & 5 \\
% \{0,1,2\} & 8 & 8 \\
% \{0,1,3\} & 13 & 11 \\
% \{0,2,3\} & 13 & 11 \\
% \{1,2,3\} & 8 & 8 \\
% \{0,1,2,3\} & 13 & 13 \\
% \bottomrule
%\end{array}
%\en
%\caption[Comparison of $N(S)$ for two shifts of finite type]{Comparison of $N(S)$ for SFT 13 and SFT 17 over all $S\subset 4^*$.\label{comp1317table}}
%\end{table}
 \end{example}
 
 \begin{example}\label{ex:2}
  \begin{table}[h]
\begin{tabular}{@{}m{.18\textwidth}m{.05\textwidth}m{.25\textwidth}m{.08\textwidth}m{.08\textwidth}m{.08\textwidth}m{.08\textwidth}@{}}
\toprule
\hspace{.08\textwidth}$M$ & $\rho(M)$
   &\ \hspace{.25in} \text{Graph} & \text{Entropy} & $H(10)$ & $\Asc(10)$ & $\Int(10)$
   \\
   \midrule
$\left(
\begin{array}{ccc}
 0 & 1 & 1 \\
 1 & 1 & 1 \\
 1 & 0 & 1 \\
\end{array}
\right)$ & 2 &\includegraphics[width=1in,keepaspectratio]{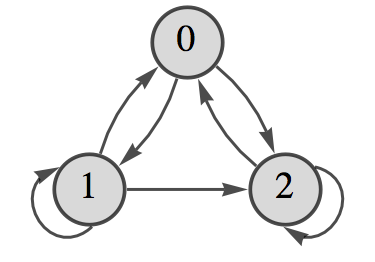}  & 0.810 &0.844& 0.490 & 0.136 \\
$\left(
\begin{array}{ccc}
 1 & 1 & 1 \\
 1 & 1 & 0 \\
 1 & 0 & 0 \\
\end{array}
\right)$ & 2 &\includegraphics[width=1in,keepaspectratio]{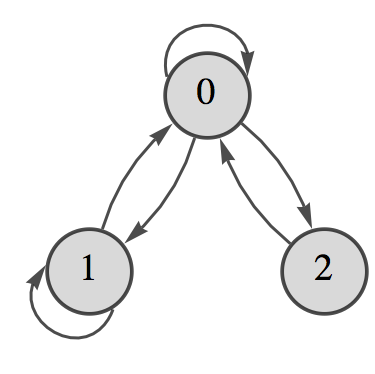}  & 0.810 & 0.830 & 0.472&0.114 \\

\bottomrule

\end{tabular}
\caption[Two shifts of finite type with the same entropy]{Two shifts of finite type with the same entropy but different average sample complexity and intricacy\label{2028table}}
\end{table}
 In the next example (see Table \ref{2028table}) we use Theorem~\ref{squarethm} to compare average sample complexity and intricacy of rank $0$ cylinder sets for two shifts of finite type with positive square adjacency matrices, which we denote by $X_{1}$ and $X_{2}$, respectively. These shifts both have the same entropy, but they have different average sample complexity and intricacy. Their complexity functions are different but have the same exponential growth rate. In this case, intricacy and average sample complexity tell us more than the entropy. The reason these quantities are smaller for $X_{2}$ than $X_{1}$ is that $|\mathscr{L}_{n^*}(X_{2})|<|\mathscr{L}_{n^*}(X_{1})|$ for all $n$. 
 \end{example}
 
   \begin{center}
\begin{longtable}{@{}m{.18\textwidth}m{.05\textwidth}m{.25\textwidth}m{.08\textwidth}m{.08\textwidth}m{.08\textwidth}m{.08\textwidth}@{}}
\toprule
\hspace{.08\textwidth}$M$ & $\rho(M)$
   &\ \hspace{.25in} \text{Graph} & \text{Entropy} & $H(10)$ & $\Asc(10)$ & $\Int(10)$
   \\
   \midrule

   \endfirsthead
   \multicolumn{7}{l}{{Continued from previous page}}\\
   \toprule
\hspace{.08\textwidth}$M$ & $\rho(M)$
   &\ \hspace{.25in} \text{Graph} & \text{Entropy} & $H(10)$ & $\Asc(10)$ & $\Int(10)$
   \\
   \midrule
   
   \endhead
   
   \bottomrule
  \multicolumn{7}{r}{{Continued on next page}}
   \endfoot
   \bottomrule
   \caption{Table with calculations for SFTs with the same entropy.\label{3lettersameentropytable}}
   \endlastfoot
   
$\left(
\begin{array}{ccc}
 1 & 1 & 0 \\
 0 & 1 & 1 \\
 1 & 0 & 1 \\
\end{array}
\right)$ & 2 &   \includegraphics[width=1in,keepaspectratio]{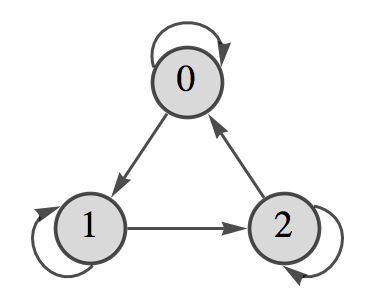}& 0.693 & 0.734 & 0.458 & 0.182 \\
$\left(
\begin{array}{ccc}
 0 & 1 & 1 \\
 1 & 0 & 1 \\
 1 & 1 & 0 \\
\end{array}
\right)$ & 2 & \includegraphics[width=1in,keepaspectratio]{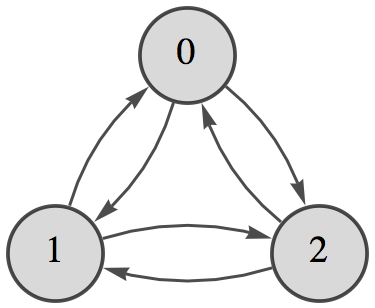}  & 0.693 & 0.734 & 0.458 & 0.182 \\
$\left(
\begin{array}{ccc}
 1 & 1 & 0 \\
 0 & 0 & 1 \\
 1 & 1 & 1 \\
\end{array}
\right)$ & 2 &  \includegraphics[width=1in,keepaspectratio]{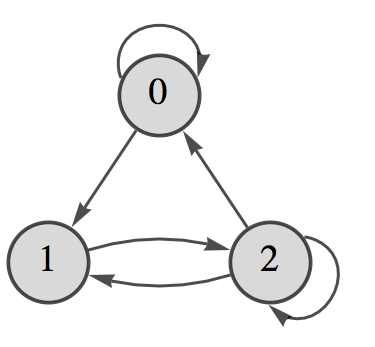} & 0.693 & 0.734 & 0.458 & 0.182 \\
$\left(
\begin{array}{ccc}
 1 & 1 & 0 \\
 0 & 1 & 1 \\
 1 & 1 & 0 \\
\end{array}
\right)$ & 2 &   \includegraphics[width=1in,keepaspectratio]{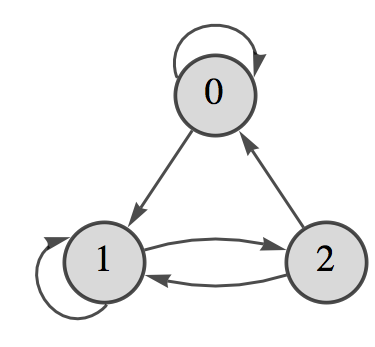}& 0.693 & 0.734 & 0.458 & 0.182 \\
 $\left(
\begin{array}{ccc}
 0 & 1 & 1 \\
 1 & 0 & 1 \\
 1 & 0 & 1 \\
\end{array}
\right)$ & 3 &  \includegraphics[width=1in,keepaspectratio]{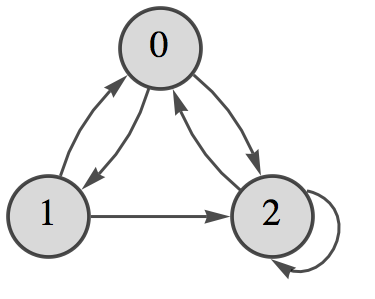} & 0.693 & 0.734 & 0.446 & 0.158 \\
 $\left(
\begin{array}{ccc}
 0 & 1 & 1 \\
 1 & 1 & 1 \\
 1 & 0 & 0 \\
\end{array}
\right)$ & 3 &   \includegraphics[width=1in,keepaspectratio]{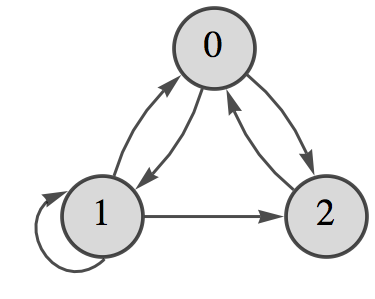}& 0.693 & 0.734 & 0.446 & 0.158 \\
$\left(
\begin{array}{ccc}
 1 & 1 & 1 \\
 1 & 0 & 0 \\
 1 & 0 & 0 \\
\end{array}
\right)$ & 2 &  \includegraphics[width=1in,keepaspectratio]{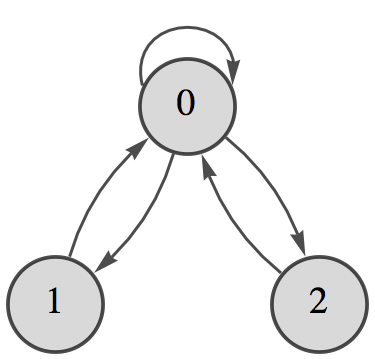} & 0.693 & 0.722 & 0.440 & 0.158 \\
\end{longtable}
\end{center}
   
\begin{example}[Comparison of SFTs with the same entropy]\label{ex:3}  
Table~\ref{3lettersameentropytable} shows seven SFTs with the same entropy but not all the same intricacy or average sample complexity functions. The smallest power for which the first four adjacency matrices and the last adjacency matrix are positive is $2$, while it is $3$ for the other two adjacency matrices. These two groups have the same $\Asc$ and $\Int$. The last SFT is unique among these seven in that it has the same entropy as the other six SFTs, the square of its adjacency matrix is positive, but it has lower intricacy and average sample complexity than the others (the rounding makes it appear to have the same intricacy in the table).
 \end{example}

\section{Measure-theoretic intricacy and average sample complexity}

We formulate definitions of measure-theoretic intricacy and measure-theoretic average sample complexity in analogy with measure-theoretic entropy. 
\begin{definition}\label{mtdefinitions}
Let $(X,\mathscr{B},\mu,T)$ be a measure-preserving system, $\alpha=\{A_1,\dots,A_n\}$ a finite measurable partition of $X$ and $c_S^n$ a system of coefficients. Recall $H_\mu(\alpha)=-\sum_{i=1}^n\mu(A_i)\log \mu(A_i)$ and for $S\subset n^*$
\be
\alpha_S=\bigvee_{i\in S}T^{-i}\alpha.
\en
The \emph{measure-theoretic intricacy of $T$ with respect to the partition $\alpha$} is
\begin{equation}
\Int_\mu(X,\alpha,T)=\lim_{n\rightarrow \infty}\frac{1}{n}\sum_{S\subset n^*}c_S^n\left[H_\mu(\alpha_S)+H_\mu(\alpha_{\setcomp{S}})-H_\mu(\alpha_{n^*})\right].
\end{equation}
The \emph{measure-theoretic average sample complexity of $T$ with respect to the partition $\alpha$} is 
\begin{equation}
\Asc_\mu(X,\alpha,T)=\lim_{n\rightarrow\infty}\frac{1}{n}\sum_{S\subset n^*}c_S^nH_\mu(\alpha_S).
\end{equation}
%The \emph{measure-theoretic average configuration complexity of $T$ with respect to the partition $\alpha$} is
%\begin{equation}
%\Acc_\mu(X,\alpha,T)=\lim_{n\rightarrow \infty}\frac{1}{n}\sum_{\substack{S\subset n^*\\0\in S}}c_S^nH_\mu(\alpha_S).
%\end{equation}
\end{definition}
If $(X,T)$ is a subshift, $\alpha$ the partition by rank $0$ cylinder sets and $\mathscr{U}(\alpha)$ the corresponding open cover of $X$, then $\Asc_\mu(X,\alpha,T)\le\Asc(X,\mathscr{U},T)$, since, for each $n$ and $S\subset n^*$, $H_\mu(\alpha_S)\le\log N(\mathscr{U}(\alpha)_S)$.
 We also define the \emph{measure-theoretic intricacy function} and the \emph{measure-theoretic average sample complexity function} as we did in the topological case.
\begin{definition}\label{mtnewcompdefs}
Let $(X,\mathscr{B},\mu,T)$ be a measure-preserving system, $\alpha$ a finite measurable partition of $X$ and a system of coefficients $c_S^n$. The \emph{measure-theoretic average sample complexity function of $T$ with respect to the partition $\alpha$} is given by
\be
\Asc_\mu(X,\alpha,T,n)=\frac{1}{n}\sum_{S\subset n^*}c_S^nH_\mu(\alpha_S).
\en
The \emph{measure-theoretic intricacy function of $T$ with respect to the partition $\alpha$} is given by
\be
\Int_\mu(X,\mathscr{U},T,n)=\frac{1}{n}\sum_{S\subset n^*}c_S^n\left[H_\mu(\alpha_S)+H_\mu(\alpha_{\setcomp{S}})-H_\mu(\alpha_{n^*})\right] .
\en
When the context is clear we may also write these as $\Asc_\mu(n)$ and $\Int_\mu(n)$.
As in Proposition \ref{accascprop1}, restricting $S$ to include $0$ gives us half of $\Asc_\mu(X,\alpha,T)$:
\be
\begin{aligned}
	\Acc_\mu(X,T,\alpha) &= \lim_{n \to \infty} \frac{1}{n} \sum_{\substack{S\subset n^*\\ 0\in S }} \frac{1}{2^n} H_\mu(\alpha_S) 
	=\lim_{n \to \infty} \frac{1}{n} \sum_{S' \subset (n-1)^*} \frac{1}{2^n} H_\mu (\alpha \vee \alpha_{S'}) \\
	&=\lim_{n \to \infty} \frac{1}{n} \sum_{S' \subset (n-1)^*} \frac{1}{2^n}[H_\mu(\alpha_{S'})+H_\mu(\alpha|\alpha_{S'})] = \frac{1}{2} \Asc_\mu(X,\alpha,T),
\end{aligned}
\en
since the limit of the second term is $0$.
\end{definition}

\begin{theorem}\label{ascmuinf}
Let $(X,\mathscr{B},\mu,T)$ be a measure-preserving system and $\alpha$ a finite measurable partition. For a system of coefficients $c_S^n$, $\Asc_\mu(X,\alpha,T)$ exists and equals $\inf_n(1/n)\sum_{S\subset n^*}c_S^nH_\mu(\alpha_S)$.
\end{theorem}
\begin{proof}
Let $b_n=\sum_{S\subset n^*}c_S^nH_\mu(\alpha_S)$.  For each $S\subset (n+m)^*$ define $U=U(S)$ and $V=V(S)$ as in the proof of Theorem~\ref{syscoeffsubadd}. We have
\begin{equation}\label{Hsubadd}
H_\mu(\alpha_S)\le H_\mu(\alpha_{U})+H_\mu(\alpha_{V}).
\end{equation}
The proof of subadditivity of $b_n$ follows in the same manner as in the proof of Theorem~\ref{syscoeffsubadd}. 
\end{proof}

\begin{corollary}
Let $(X,\mathscr{B},\mu,T)$ be a measure-preserving system and $\alpha$ a finite measurable partition. If $c_S^n$ is a system of coefficients, then the limit in the definition $\Int_\mu(X,\alpha,T)$ (Definition~\ref{mtdefinitions}) exists.
\end{corollary}
\begin{proof}
This follows from the fact that 
\begin{equation}\label{intmueq}
\Int_\mu(X,\alpha,T)=2\Asc_\mu(X,\alpha,T)-h_\mu(X,\alpha,T).
\end{equation}
\end{proof}

\begin{proposition}\label{ascmaxmeascor}
	Let $(X,T)$ be a topological dynamical system, $\alpha$ a fixed Borel measurable partition of $X$, and $c_S^n=2^{-n}$ for all $S\subset n^*$. There exist ergodic probability measures on $X$ that maximize $\Asc_\mu(X,\alpha,T)$.
\end{proposition}
\begin{proof}
	An adaptation of the argument in \cite[Prop.10.13, p. 61]{DGS} shows that $\Asc_\mu (X, \alpha, T)$ is an affine function of $\mu$. Since $\Asc_\mu(X,\alpha,T)$ is an infimum of continuous functions of $\mu$ (see Theorem~\ref{ascmuinf}), it is an upper semi-continuous function of $\mu$. The space of invariant  probability measures on $X$ is nonempty and compact in the weak$^*$-topology, and an upper semi-continuous function on a compact space attains its supremum. Therefore, the set of measures $\mu$ that maximize $\Asc_\mu(X,\alpha,T)$ is nonempty. It is convex because $\Asc_\mu(X,\alpha,T)$ is affine in $\mu$. The extreme points of this set coincide with the ergodic measures that maximize $\Asc_\mu(X,\alpha,T)$. (See Chapter 8 of \cite{waltersergodic} for more details and proofs of the properties of $h_{\mu}(X,\alpha,T)$.)
\end{proof}

%%%%%%%%%%%%%%%%%%%%%%%%%%%%%%%%%%%%%%%%%%%%%%%%
%
%					THEOREM RELATING ASC AND INT TO ENTROPY
%
%%%%%%%%%%%%%%%%%%%%%%%%%%%%%%%%%%%%%%%%%%%%%%%%

Now we establish a measure-theoretic analogue to Theorem~\ref{asctoentthm}: 
\be\sup_\alpha \Asc_\mu(X,\alpha,T)=\sup_\alpha\Int_\mu(X,\alpha,T)=h_\mu(X,T).
\en
 The proof follows the structure of the proof of Theorem~\ref{asctoentthm}, and, as in the topological case, the theorem motivates us to focus the study of measure-theoretic intricacy and measure-theoretic average sample complexity on particular partitions, for example the partition by time-zero cylinder sets in subshifts. 
\begin{theorem}\label{mtheoreticasctoent}
Let $(X,\mathscr{B},\mu,T)$ be a measure-preserving system and fix the system of coefficients $c_S^n=2^{-n}$. Then
\be
\sup_{\alpha}\Asc_{\mu}(X,\alpha,T)=h_\mu(X,T).
\en
\end{theorem}

\begin{lemma}\label{Nsetsmeasurelemma}
Let $(X,\mathscr{B},\mu,T)$ be a measure-preserving system and $\alpha$ a finite measurable partition of $X$. Given  $n\in\mathbb{N}$ and $S\subset n^*$, the following properties hold:
\begin{enumerate}[\normalfont 1.]
\item $H_{\mu}((\alpha_{k^*})_S)=H_{\mu}(\alpha_{S+k^*})$.
\item Given $S_1,S_2,\dots, S_m\subset n^*$,  $H_{\mu}(\alpha_{\cup_iS_i})\le\sum_{i}H_{\mu}(\alpha_{S_i})$.
\end{enumerate}
\end{lemma}

\begin{corollary}\label{cor:supintent}
Let $(X,\mathscr{B},\mu,T)$ be a measure-preserving system, $\alpha$ a finite measurable partition of $X$, and fix the system of coefficients $c_S^n=2^{-n}$. Then
\be
\sup_{\alpha}\Int_{\mu}(X,\alpha,T)=h_\mu(X,T).
\en
\end{corollary}
\begin{proof}
This follows from Theorem~\ref{mtheoreticasctoent} and Equation~\ref{intmueq}
{by considering an increasing sequence of partitions $\alpha^{(n)}$ with $\Asc_\mu(X, \alpha^{(n)},T) \nearrow h_\mu(X,T)$.}
\end{proof}

%%%%%%%%%%%%%%%%%%%%%%%%%%%%%%%%%%%%%%%%%%%%%%%%
%
%				THMS RELATING H_MU TO ASC_MU
%
%%%%%%%%%%%%%%%%%%%%%%%%%%%%%%%%%%%%%%%%%%%%%%%%

The next results give a relationship between measure-theoretic average sample complexity of a finite measurable partition $\alpha$, 
the first-return map on a cross product, and a series summed over $i$ involving the conditional entropies $H_\mu(\alpha\mid\alpha_i)$. We will take advantage of this in the next section to compute $\Asc_\mu$ and $\Int_\mu$ for $1$-step Markov shifts. One purpose of accurately computing $\Asc_\mu$ and $\Int_\mu$ is to look for measures $\mu$ that maximize these quantities. 

When the weights are $c_S^n=2^{-n}$, by considering subsets $S$ as being formed by random choices of elements of $n^*$ we obtain Theorem~\ref{Asccompenteqthm}, which relates $\Asc_\mu(X,\alpha,T)$ to half the entropy of the first return map $T_{X\times A}$ on a cross product $X\times A$ of $X$ with the cylinder $A=[1]$ in the full $2$-shift with respect to the finite measurable partition $\alpha\times A$. 

In the one-sided full $2$-shift, $\Sigma_2^+$, we define $A=[1]=\{\xi\in\Sigma_2^+:\xi_0=1\}$. Then, subsets $S\subset n^*$ correspond to occurrences of $1$ in the first $n$ elements of sequences $\xi\in A$. Denote by $\xi_0^{n-1}$ both the string $\xi_0\xi_1\dots\xi_{n-1}$ and the cylinder set $\{z\in\Sigma_2^+:z_i=\xi_i\text{ for all }i=0,1,\dots,n-1\}$. We denote the subsets $S\subset n^*$ or $\mathbb N$ corresponding to $\xi$ by $S(\xi_0^{n-1})=\{i\in n^*:\xi_i=1\}$ and $S(\xi)=\{i\in\mathbb{N}:\xi_i=1\}$. Since averaging $H_\mu(\alpha_S)$ over all $S$ with weights $2^{-n}$ amounts to picking random $S$ and taking the expectation of $H_\mu(\alpha_S)$, we make calculations by doing the latter. 

We introduce some notation and give some facts necessary for the statements and proofs of Theorem~\ref{Asccompenteqthm} and Proposition \ref{infocor}; see \cite{petersen1989ergodic} for more background information and details. Let $P$ denote the Bernoulli measure $\mathscr{B}(1/2,1/2)$ on $\Sigma_2^+$. Let $A$ be the subset of $\Sigma_2^+$ defined above and denote by $n_A:A\rightarrow \mathbb{N}$ the minimum return time of a sequence $\xi\in A$ to $A$ under the shift $\sigma$, i.e.,
\be
n_A(\xi)=\inf\{n\ge 1:\sigma^n\xi\in A\}=\inf\{n\ge 1:\xi_n=1\}.
\en
%8/7 define A_n
Define $A_n = \{ \xi \in A: n_A(\xi)=n\}, n\geq 1$.
Since $(\Sigma_2^+,\sigma,P)$ is ergodic, the expected recurrence time of a point $\xi\in A$ to $A$ is $1/P(A)$. Let $\sigma_A\xi=\sigma^{n_A(\xi)}\xi$. Given a positive integer $n$ and sequence $\xi\in A$, define $m_\xi(n)$ by
\be
m_\xi(n)=\sum_{i=0}^{n-1}n_A(\sigma_A^{i}\xi)=n_A(\xi)+n_A(\sigma_A\xi)+\cdots+n_A(\sigma_A^{n-1}\xi),
\en
the sum of the first $n$ return times of $\xi$ to $A$. Since the expected return time of $\xi$ to $A$ is $1/P(A)$, we have that
\begin{equation}\label{mP2eq}
\lim_{n\rightarrow\infty}\frac{m_\xi(n)}{n}= \frac{1}{P(A)}=2\text{ for } P_A\text{-a.e. } \xi\in A.
\end{equation}
Since $n_A\in L^1$, the Ergodic Theorem implies $m_\xi/n\rightarrow 2$ in $L^1$ as well.

For a measure-preserving system $(X,\mathscr{B},\mu,T)$ denote by $T_{X\times A}$ the first-return map on $X\times A$, so that $T_{X\times A}(x,\xi)=(T^{n_A(\xi)}x,\sigma_A\xi)$.
In the proof we also use the fact that for two countable measurable partitions $\alpha$ and $\gamma$ of $X$
\begin{equation}\label{halphabetapart}
H_\mu(\alpha\vee\gamma)=H_\mu(\alpha)+H_\mu(\gamma|\alpha).
\end{equation}

\begin{theorem}\label{Asccompenteqthm}
Let $(X,\mathscr{B},\mu,T)$ be an ergodic measure-preserving system and $\alpha$ a finite measurable partition of $X$. Let $A=[1]=\{\xi\in\Sigma_2^+:\xi_0=1\}$ and let $\beta=\alpha\times A$ be the related finite partition of $X\times A$. Denote by $T_{X\times A}$ the first-return map on $X\times A$ and let $P_A=P/P[1]$ denote the measure $P$ restricted to $A$ and normalized. Let $c_S^n=2^{-n}$ for all $S\subset n^*$. Then
\begin{equation}\label{Asccompenteq}
\Asc_\mu(X,\alpha,T)=
\frac{1}{2}\lim_{n \to \infty} \frac{1}{n} \int_A H_\mu (\alpha_{S(\xi_0^{m_\xi(n)-1})}) \, dP_A(\xi)
\leq \frac{1}{2}h_{\mu\times P_A}(X\times A,\beta,T_{X\times A}).
\end{equation}
\end{theorem}
\begin{proof}
Given $\varepsilon>0$, define the set $U_\varepsilon(n)\subset A$ by
\be
U_\varepsilon(n)=\left\{\xi\in A:\left\lvert \frac{m_\xi(n)}{n}-2\right\rvert>\varepsilon\right\}.
\en
By Equation~\ref{mP2eq}, $\lim_{n\rightarrow\infty}P(U_\varepsilon(n))= 0$. For $\xi\in A\setminus U_\varepsilon$, from Equation~\ref{halphabetapart}, if $m_\xi(n) \leq 2n-1$ we have
\be\label{epsprimeeq}
%8/26added -1 in exponen
\begin{aligned}
&\frac{1}{n}\left\lvert H_\mu\left(\alpha_{S\left(\xi_0^{2n-1}\right)}\right)-H_\mu\left(\alpha_{S\left(\xi_0^{m_\xi(n)-1}\right)}\right)\right\rvert 
=\frac{1}{n}\left\lvert H_\mu\left(\alpha_{S\left(\xi_{m_\xi(n)}^{2n-1}\right)}\middle|\alpha_{S\left(\xi_0^{m_\xi(n)-1}\right)}\right)\right\rvert\\ 
&\le \frac{1}{n}\left\lvert 2n-m_{\xi(n)}\right\rvert H_\mu(\alpha)
\leq \varepsilon H_\mu(\alpha) \text{  for large } n,
\end{aligned} 
\en
since $\alpha_{S(\xi_{m_\xi(n)}^{2n-1})}$ is the join of at most $2n-m_{\xi(n)}$ terms, each of entropy no more than $H_\mu(\alpha)$. The same estimate applies in case $m_\xi(n) > 2n-1$.

 Since $P_A\{\xi \in A:\xi_0^{2n-1}=S\} = 1/2^{2n-1}$ for each $S= \{0,s_1,s_2,\dots\} \subset (2n)^*$,
\be\label{eq:AscForm}
\Asc_\mu(X,\alpha,T) = 2 \lim_{n \to \infty} \frac{1}{2n}
\sum_{\substack{S\subset n^*\\ 0\in S }} \frac{1}{2^{2n}}H_\mu(\alpha_S) 
= \lim_{n \to \infty} \int_A \frac{1}{2n}H_\mu(\alpha_{S(\xi_0^{2n-1})})\, dP_A(\xi).
	\en

Since $m_\xi/n\rightarrow 2$ in $L^1$ and $P_A(U_\varepsilon(n))\rightarrow0$ we have
\begin{equation}
\frac{m_\xi(n)}{n}\chi_{U_{\varepsilon}(n)}(\xi)\rightarrow 0\text{ in }L^1.
\end{equation}
By the definition of $H_\mu$ we know $(1/m_\xi) H_\mu\left(\alpha_{S\left(\xi_0^{m_\xi-1}\right)}\right)$
is bounded, so
\begin{equation}\label{uepseq2}
\lim_{n\rightarrow\infty}\int_{U_\varepsilon}\frac{m_\xi}{n}\frac{1}{m_\xi}H_\mu\left(\alpha_{S\left(\xi_0^{m_\xi-1}\right)}\right)dP_A(\xi)=0.
\end{equation}
Similarly
\begin{equation}\label{uepsset2n0}
\lim_{n\rightarrow\infty}\int_{U_\varepsilon}\frac{1}{2n}H_\mu\left(\alpha_{S(\xi_0^{2n-1})}\right)dP_A(\xi)=0.
\end{equation}

Thus for large enough $n$,
\be
\begin{aligned}
	&\lvert\Asc_\mu(X,\alpha,T) -\int_A \frac{1}{2n}H_\mu(\alpha_{S(\xi_0^{m_\xi(n)-1})})\, dP_A(\xi)\rvert\\
	&\leq  \lvert\int_A \frac{1}{2n}H_\mu(\alpha_{S(\xi_0^{2n-1})})\, dP_A(\xi) -
	\int_A \frac{1}{2n}H_\mu(\alpha_{S(\xi_0^{m_\xi(n)-1})})\, dP_A(\xi)\rvert + \varepsilon\\
	&\leq \int_{A \setminus U_\varepsilon} \frac{1}{2n}\lvert H_\mu(\alpha_{S(\xi_0^{2n-1})}) - H_\mu(\alpha_{S(\xi_0^{m_\xi(n)-1})})\rvert\, dP_A +\\
		&\int_{U_\varepsilon(n)} \lvert\frac{1}{2n}H_\mu(\alpha_{S(\xi_0^{2n-1})}) + \frac{1}{2} \frac{m_\xi}{n} \frac{1}{m_\xi}H_\mu(\alpha_{S(\xi_0^{m_\xi(n)-1})})\rvert\, dP_A(\xi) + \varepsilon.
\end{aligned}
\en
For large $n$ the first term is bounded by $2 \varepsilon H_\mu(\alpha)$, and the second tends to $0$ as $n \to \infty$. Therefore
\be
\Asc_\mu(X,\alpha,T)=
\frac{1}{2}\lim_{n \to \infty} \frac{1}{n} \int_A H_\mu (\alpha_{S(\xi_0^{m_\xi (n)-1})}) dP_A(\xi).
\en

For each $(x,\xi)$ in $X\times A$, $\beta_{n^*}(x,\xi)$ denotes the element of $\beta_{n^*}=\bigvee_{i=0}^{n-1}T^{-i}_{X\times A}(\alpha\times A)$ to which $(x,\xi)$ belongs. 
Some caution is necessary here: although $\beta$ does not partition the second coordinate in $X \times A$, when $\beta$ is moved by $T_{X \times A}$ and the resulting partitions are joined, some partitioning of the second coordinate does take place, due to the return-times partition of $A$ with respect to $T_A$. Thus $\beta_{n^*}$ is a proper refinement of $\alpha_{S(\xi_0^{m_\xi (n)})} \times A$.

Turning to the remaining inequality in the statement of the theorem, 
by definition of the information function $I$ we have
\be\label{ascmucrossPAeq}
\begin{aligned}
	h_{\mu\times P_A}(X\times A,\beta,T_{X\times A})&=\lim_{n\rightarrow\infty}\frac{1}{n}H_{\mu\times P_A}(\beta_{n^*})\\
	&=\lim_{n\rightarrow\infty}\frac{1}{n}\int_{X\times A}I_{\beta_{n^*}}(x,\xi) \, d\mu(x)dP_A(\xi)\\
	&=-\lim_{n\rightarrow\infty}\frac{1}{n}\int_A\int_X\log\left[(\mu\times P_A)(\beta_{n^*}(x,\xi))\right] d\mu(x)dP_A(\xi).
\end{aligned}
\en
For each $x \in X$ and $\xi\in A$, $\beta_{n^*}(x, \xi) \subset \alpha_{S(\xi_0^{m_\xi (n)-1})} \times A$, so
\be
(\mu\times P_A)(\beta_{n^*}(x,\xi)) \leq (\mu \times P_A) (\alpha_{S(\xi_0^{m_\xi(n)-1})}(x) \times A),
\en
and hence (\ref{ascmucrossPAeq}) implies
\be
\begin{aligned}
	h_{\mu \times P_A}(X \times A, \beta, T_{X \times A})
	&\geq -\lim_{n\rightarrow\infty}\frac{1}{n}\int_A\int_X\log (\mu\times P_A) (\alpha_{S(\xi_0^{m_\xi(n)-1})}(x)\times A) \, d\mu(x) dP_A(\xi)\\
	&= \lim_{n\rightarrow\infty}\frac{1}{n}\int_A H_\mu (\alpha_{S(\xi_0^{m_\xi(n)-1})}) \, dP_A(\xi).\label{fullinteq}
\end{aligned}
\en
\end{proof}

\begin{proposition}\label{infocor}
Let $(X,\mathscr{B},\mu,T)$ be a 1-step Markov shift and $\alpha$ the finite time-0 generating partition of $X$. Let $c_S^n=2^{-n}$ for all $S\subset n^*$. Then
\begin{equation}\label{Ascseriesequal}
\Asc_\mu(X,\alpha,T)=\frac{1}{2}\sum_{i=1}^\infty\frac{1}{2^i}H_\mu\left(\alpha\mid\alpha_i\right).
\end{equation}
\end{proposition}
\begin{proof}
Using the above notation, let $S(\xi)=\{s_0(\xi),s_1(\xi).\dots\}$, abbreviated $\{s_0,s_1,\dots\}$.  For $\xi \in A$ we have $s_0=0$ and on $A_i$ we have $s_1(\xi)=n_A(\xi)=i$. 
For $\xi \in A$, there are $n$ hits of $A$ among $\xi, \sigma \xi, \dots, \sigma^{m_\xi(n)-1}\xi$. 
By the conditional entropy formula,
\be
H_\mu(\alpha_{s_0} \vee \dots \vee \alpha_{s_n})=H_\mu(\alpha_{s_0}|\alpha_{s_1} \vee \dots \vee \alpha_{s_n}) + H_\mu(\alpha_{s_1} \vee \dots \vee \alpha_{s_n})
\en
etc., and so by the Markov property
\be
H_\mu(\alpha_{s_0} \vee \dots \vee \alpha_{s_n}) - H_\mu(\alpha_{s_n}) = H_\mu(\alpha_{s_0}|\alpha_{s_1}) + H_\mu(\alpha_{s_1}|\alpha_{s_2}) + \dots + H_\mu(\alpha_{s_{n-1}}|\alpha_{s_n}).
\en
Since $s_1(\xi)= s_0(\sigma_A\xi), s_2(\xi)=s_1(\sigma_A\xi), 
\dots , s_n(\xi)=s_{n-1}(\sigma_A(\xi)=n_A(\sigma_A^{n-1}\xi)$
and $\sigma_A$ is measure-preserving on $A$, each of these latter terms has the same integral over $A$. 
Combining this with (\ref{Asccompenteq}) and (\ref{eq:AscForm}),
\be
\begin{aligned}
\Asc_\mu(X,\alpha,T) &= \frac{1}{2}\lim_{n \to \infty} \int_A\frac{1}{n}H_\mu(\alpha_{S(\xi_0^{m_\xi(n)})})\, dP_A\\
&= \frac{1}{2} \int_A H_\mu(\alpha|\alpha_{s_1(\xi)}\, dP_A(\xi)=\frac{1}{2} \sum_{i=1}^\infty \frac{1}{2^i}H_\mu(\alpha|\alpha_i). 
\end{aligned}
\en
\end{proof}

We are grateful to Jean-Paul Thouvenot for helping to analyze more precisely the relationship, described in the following theorem, between $\Asc_\mu(X,\alpha,T)$ and the entropy of the partition $\beta = \alpha \times A$ under the first-return map $T_{X \times A}$ expressed in Theorem \ref{Asccompenteqthm}. 
Continue with the notation and hypotheses of Theorem \ref{Asccompenteqthm}. The first-return system consisting of $T_{X \times A}$ on $X \times A$ may also be regarded as a skew product. The base system is $(A, \sigma_A, P_A)$, and the map is given by $T_{X \times A}(x,\xi)=(T^{n_A(\xi)}x,\sigma_A\xi)$.  (Here the base is written in the second coordinate.) We may write $T^{n_A(\xi)}=T_\xi x$. 
The base system $(A, \sigma_A, P_A)$ is isomorphic to the countable-state Bernoulli system 
%8/7
with states $A_i = \{ \xi \in A : n_A(\xi)=i\}, i\geq 1$, and probabilities $P_A(A_i)=1/2^i, i=1,2,\dots$. 
%8/7 Define return-times algebra
Let $\mathcal A= \{X \times A_n: n \geq 1\}$ denote the $\sigma$-algebra generated by the sets $A_I, i \geq 1$.

%8/7 correct last item
Since in our setting a partition or algebra may be moved either by the original transformation or by the first-return (or skew product) transformation, we adopt special notation for the latter:
\be
\begin{gathered}
\beta_{0,n-1}^*=\bigvee_{k=0}^{n-1}T_{X \times A}^{-k}\beta, \quad \mathcal A_{0,n-1}^*=\bigvee_{k=0}^{n-1}T_{X \times A}^{-k}\mathcal A,\\ \alpha_{1,n-1}^*(\xi)=T^{-n_A(\xi)}\alpha \vee \dots \vee T^{-n_A(\xi)-n_A(\sigma_A\xi) - \dots -n_A(\sigma_A^{n-1}\xi)}\alpha,
\end{gathered}
		\en
		etc., also for the $\sigma$-algebras generated as $n \to \infty$.
		
%8/24 add partitions to notation		
According to the formula for the entropy of a skew product,
\be
h_{P_A \times \mu}(\sigma_A \times \{T_\xi, \beta \})=h(A,\sigma_A,P_A) + h_{\sigma_A}(X,T,\mu,\alpha),
\en
where
%8/31 moved )
\be
 h_{\sigma_A}(X,T,\mu, \alpha)= \int_A H_\mu(\alpha|\alpha_{1,\infty}^*)(\xi)\, dP_A(\xi)
 \en
 is the fiber entropy of the skew product system with respect to the fixed partition $\alpha$. 
 
 %8/24 cut confused sentence about factor process and phrase about factor
 The following theorem identifies $\Asc_\mu(X,\alpha,T)$ as one-half of the conditional entropy of the partition $\beta$ of $X \times A$, moved by the first-return map $T_{X \times A}$, given the return-times algebra $\mathcal A_{-\infty,\infty}^*$ of the base. 
 The process $(\beta,T_{X \times A},\mu \times P_A)$ reads only the first coordinate (the cell of the partition $\alpha$ of $X$), not knowing the times at which the readings are being made; it must be given extra information about the return times to arrive at $\Asc_\mu(X,\alpha,T)$. 
 This is the reason for the inequality in (\ref{Asccompenteq}).
 
 %`8/24 add alpha
 \begin{theorem}\label{thm:FirstReturns}
 	With the notation and hypotheses of Theorem \ref{Asccompenteqthm},
 	\be
 	\begin{aligned}
 	\Asc_\mu(X,\alpha,T)&=\frac{1}{2} \lim_{n \to \infty} \frac{1}{n} H_{\mu \times P_A}\left(\bigvee_{k=0}^{n-1}T_{X \times A}^{-k}\beta \Big\vert \bigvee_{k=0}^{n-1}T_{X \times A}^{-k}\mathcal A\right) \\
 &	= \frac{1}{2} H_{\mu \times P_A}(\beta|\beta_{1,\infty}^* \vee \mathcal A_{-\infty, \infty}^*)
 	= \frac{1}{2} h_{\mu \times P_A}((\beta,T_{X \times A}, \mu \times P_A)|\mathcal A_{-\infty, \infty}^*)\\
 	&= \frac{1}{2}h_{\sigma_A}(X,T,\mu, \alpha).
 	\end{aligned}
 	\en
 	 \end{theorem}
 	 %8/31 tried to clarify proof
 	 \begin{proof}
 	 	Each cell $C$ of $\mathcal A_{0,n-1}^*$ corresponds to a choice of $S \subset m_\xi(n)^*$, so 
 	 \be
 	  H_{\mu \times P_A}(\beta_{0,n-1}^*|\mathcal A_{0,n-1}^*) = 
 	  \sum_{C \in \mathcal A_{0,n-1}^*} P_A(C)H_\mu(\beta_{0,n-1}^*|C)= 
 	   \int_A H_\mu (\alpha_{S(\xi_0^{m_\xi (n)-1})}) \, dP_A(\xi).
 	   \en
 	 From Theorem \ref{Asccompenteqthm},
 	 	\be
 	 	\begin{aligned}
 	 	\Asc_\mu(X,\alpha,T)&= \frac{1}{2}\lim_{n \to \infty} \frac{1}{n} \int_A H_\mu (\alpha_{S(\xi_0^{m_\xi (n)-1})}) \, dP_A(\xi)
 	 	 	 = \frac{1}{2}\lim_{n \to \infty} \frac{1}{n} H_{\mu \times P_A}(\beta_{0,n-1}^*|\mathcal A_{0,n-1}^*)\\
 	 	 	 &=\frac{1}{2}\lim_{n \to \infty} \frac{1}{n} H_{\mu \times P_A}(\beta_{0,n-1}^*|\mathcal A_{-\infty,\infty}^*) ,
 	 	\end{aligned}
 	 	\en
 	 	since $\mathcal A_{n,\infty}^* \vee \mathcal A_{-\infty,-1}^*$ is independent of $\mathcal A_{0,n-1}^*$ and $\beta_{0,n-1}^*$, so conditioning on $\mathcal A_{0,n-1}^*$ is the same as on $\mathcal A_{-\infty,\infty}^*$. 
 	 	
 	 	%8/24 correction, remove \alpha|
 	 	We follow the standard argument from entropy theory which uses the measure-preserving property to form a telescoping sum in order to show that $\lim (1/n) H_\mu(\alpha_0^{n-1})=H_\mu(\alpha|\alpha_1^\infty)$, just adding conditioning on $\mathcal A_{-\infty,\infty}^*$.
 	 For each $j=1,2,\dots$, 
 	 \be
 	 \begin{aligned}
 	 H_{\mu \times P_A}(\beta|\beta_{1,j}^* \vee \mathcal A_{-\infty,\infty}^*)
 	&=H_{\mu \times P_A}(\beta_{0,j}^*|\mathcal A_{-\infty,\infty}^*)-H_{\mu \times P_A}(\beta_{1,j}^*|\mathcal A_{-\infty,\infty}^*)\\
 &= H_{\mu \times P_A}(\beta_{0,j}^*|\mathcal A_{-\infty,\infty}^*)-H_{\mu \times P_A}(\beta_{0,j-1}^*|\mathcal A_{-\infty,\infty}^*).
 \end{aligned}
\en
Sum on $j=1,\dots,n$, divide by $n$, and take the limit to obtain 
\be
H_{\mu \times P_A}(\beta|\beta_{1,\infty}^* \vee \mathcal A_{-\infty,\infty}^*)= \lim_{n \to \infty} \frac{1}{n}H_{\mu \times P_A}(\beta_{0,n-1}^*|\mathcal A_{-\infty,\infty}^*).
\en

%8/24added clarifying detail
  Similarly, continuing with the above notation, remembering that $|S(\xi_0^{m_\xi(n)-1})|=n$, 
   again using the measure-preserving property of $\sigma_A$, and applying Theorem \ref{Asccompenteqthm},
  \be\label{eq:FibEnt}
  %8/26 added )
  %8/31 took out extra )
  \begin{aligned}
   h_{\sigma_A}(X,T,\mu, \alpha)&= \int_A H_\mu(\alpha|\alpha_{1,\infty}^*)(\xi)\, dP_A(\xi)
   =\lim_{n \to \infty} \frac{1}{n} \sum_{k=1}^{n-1} \int_A H_\mu (\alpha | \alpha_{s_1(\xi)} \vee \dots \vee \alpha_{s_k(\xi)}) \, dP_A\\
   &= \lim_{n \to \infty} \frac{1}{n} \sum_{k=1}^{n-1} \int_A [H_\mu (\alpha_{s_0(\xi)} \vee \alpha_{s_1(\xi)} \vee \dots \vee \alpha_{s_k(\xi)}) - H_\mu(\alpha_{s_1(\xi)} \vee \dots \vee \alpha_{s_{k}(\xi)})] \, dP_A\\
   &= \lim_{n \to \infty} \frac{1}{n} \sum_{k=1}^{n-1} \int_A  [H_\mu (\alpha_{s_0(\xi)} \vee \alpha_{s_1(\xi)} \vee \dots \vee \alpha_{s_k(\xi)}) - H_\mu(\alpha_{s_0(\xi)} \vee \dots \vee \alpha_{s_{k-1}(\xi)})] \, dP_A\\
   &= \lim_{n \to \infty} \int_A \frac{1}{n} H_\mu(\alpha_{s_0(\xi)} \vee \dots \vee \alpha_{s_{n-1}(\xi)}) \, dP_A\\
   	&= \lim_{n \to \infty} \int_A \frac{1}{n} H_\mu (\alpha_{S(\xi_0^{m_\xi(n)-1})}\, dP_A = 2 \Asc_\mu(X,\alpha,T). 
   \end{aligned}
\en
 \end{proof}
 
\begin{rem}
	The preceding formulas also yield formula (\ref{Ascseriesequal}) when $(X,\mu,T,\alpha)$ is a 1-step Markov process:
	\be
	\begin{aligned}
	 h_{\sigma_A}(X,T,\mu)&= \int_A H_\mu(\alpha|\alpha_{1,\infty}^*(\xi))\, dP_A(\xi)
	 = \int_A H_\mu(\alpha_{s_1(\xi)})\, dP_A(\xi)\\
	 &= \sum_{i=1}^\infty \frac{1}{2^n}H_\mu(\alpha|\alpha_i),
	  	\end{aligned}
 	\en
 	so 
 	\be
 	\Asc_\mu(X,\alpha,T)=\frac{1}{2}\sum_{i=1}^\infty \frac{1}{2^i}H_\mu(\alpha|\alpha_i).
 	\en
 	\end{rem}

%\subsection*{Measures of maximal average sample complexity}
%In Section~\ref{variationalprinciple} we give a relationship between topological entropy and measure-theoretic entropy known as the variational principle. In this section we discuss the analogue of measures of maximal entropy for measure-theoretic average sample complexity. Since 

%%%%%%%%%%%%%%%%%%%%%%%%%%%%%%%%%%%%%%%%%%%%%%%%
%
%				ANALYSIS OF MARKOV SHIFTS
%
%%%%%%%%%%%%%%%%%%%%%%%%%%%%%%%%%%%%%%%%%%%%%%%%

\section{Analysis of Markov shifts}\label{markovsec}
A (1-step) Markov shift, $(\mathcal{A}^\mathbb{Z},\mathscr{B},\mu_{P,p},\sigma)$  
consists of a finite alphabet which we take to be $\mathcal {A}=\{0,1,\dots, r-1\}$, the $\sigma$-algebra $\mathscr{B}$ generated by cylinder sets, a shift-invariant measure $\mu_{P,p}$ determined by an $r\times r$ stochastic matrix $P$ and a probability vector $p$ fixed by $P$, and the shift transformation $\sigma$. 
The measure of a cylinder set determined by consecutive indices is 
\be
\mu_{P,p}\{x:x_i=j_0,x_{i+1}=j_1,\dots, x_{i+k}=j_k\}=p_{j_0}P_{j_0j_1}P_{j_1j_2}\cdots P_{j_{k-1}j_k}.
\en
Thus $P_{ij}=\mu_{P,p}(x_1=j|x_0=i)$.
 A $k$-step Markov measure $P$ is a 1-step Markov measure on the recoding of the shift space  by $k$-blocks. Then the transition matrix is $r^k\times r^k$, and the states are the $k$-blocks. In some cases, whole rows or columns of $P$ will be $0$ and will be left out.

To apply Corollary~\ref{infocor} to Markov shifts, with $\alpha$ the partition into rank zero cylinder sets $A_i=\{x\in\mathcal{A}^\mathbb{Z}:x_0=i\}$,
we let $T=\sigma^{-1}$. Then
\begin{equation}
\mu_{P,p}(x\in\ A_j\mid x\in T^{-i}A_{k_i}\cap T^{-i-1}A_{k_{i+1}}\cap\cdots)=\mu_{P,p}(x\in\ A_j\mid x\in T^{-i}A_{k_i})= p_{k_i}(P^i)_{k_ij}.              % NOT p_j(P^i)_{jk_i}
\end{equation}
For Markov shifts,  the probability that $x_0=k$ if we know $x_{-i}=j$ does not depend on the entries $x_{-l}$ for $l>i$. Thus in this case
\be
H_{\mu_{P,p}}(\alpha\mid\alpha_i^\infty)=
H_\mu(\alpha\mid\alpha_i)
=-\sum_{j,k=0}^{r-1}p_j(P^i)_{jk}\log(P^i)_{jk},
\en 
so
\begin{equation}\label{1stepmarkasc}
\Asc_{\mu_{P,p}}(\mathcal{A}^\mathbb{Z},\alpha,\sigma)=-\frac{1}{2}\sum_{i=1}^\infty\frac{1}{2^i}\sum_{j,k=0}^{r-1}p_j(P^i)_{jk}\log(P^i)_{jk}.
\end{equation}

Corollary~\ref{infocor} applies to Markov shifts with memories larger than $1$ by first representing them as equivalent $1$-step Markov shifts via a higher block coding.  If $P$ is the stochastic matrix of the $1$-step Markov shift equivalent to a given higher step Markov shift, then $H_{\mu_{P,p}}(\alpha\mid\alpha_i^\infty)$ becomes more difficult to write in terms of entries of $P$ than for the case of $1$-step Markov shifts. This is because the entries of $P$ are probabilities of going from $2$-block states to $2$-block states, but, since $\alpha$ is the partition by rank zero cylinder sets, to find $H_{\mu_{P,p}}(\alpha\mid\alpha_i^\infty)$ we are required to find the probability of going from $2$-block states to $1$-block states. Denote by $P_{j``yz"}$ the entry of $P$ representing the probability of going from $2$-block state $j$ to $2$-block state $``yz"$ where $y,z\in\mathcal{A}$ are the two symbols that make up the terminal $2$-block. In this case
\be
H_{\mu_{P,p}}(\alpha\mid\alpha_i^\infty)=-\sum_{j\in\mathcal{A}^2}\sum_{z\in\mathcal{A}}p_j\left(\sum_{y\in\mathcal{A}}(P^i)_{j``yz"}\log\sum_{y\in\mathcal{A}}(P^i)_{j``yz"}\right),
\en 
so
\begin{equation}\label{2stepmarkasc}
\Asc_{\mu_{P,p}}(\mathcal{A}^\mathbb{Z},\alpha,\sigma)=-\frac{1}{2}\sum_{i=1}^\infty\frac{1}{2^i}\sum_{j\in\mathcal{A}^2}\sum_{z\in\mathcal{A}}p_j\left(\sum_{y\in\mathcal{A}}(P^i)_{j``yz"}\log\sum_{y\in\mathcal{A}}(P^i)_{j``yz"}\right).
\end{equation}
In the following sections we use Equations~\ref{1stepmarkasc} and \ref{2stepmarkasc} to compute the measure-theoretic average sample complexity for some examples of Markov shifts. In each example the matrix $P$ depends on at most two parameters, enabling us to plot in either $[0,1]\times\mathbb{R}$  or $[0,1]\times[0,1]\times \mathbb{R}$  these independent parameters versus measure-theoretic average sample complexity. Similarly, we can make plots of measure-theoretic entropy and measure-theoretic intricacy. 

We used \emph{Mathematica} \cite{mathematica} to make graphs and compute values. The calculations for measure-theoretic average sample complexity and measure-theoretic intricacy are found by taking the sum of the first $20$ terms of either (\ref{1stepmarkasc}) or (\ref{2stepmarkasc}), depending on the case. The measures in the tables give maximum values for either measure-theoretic entropy, measure-theoretic intricacy, or measure-theoretic average sample complexity. The bolded numbers in tables are the maxima for the given category. Tables show computations correct to $3$ decimal places. To simplify notation we denote $\mu_{P,p}$ by $\mu$ in this section.

%%%%%%%%%%%%%%%%%%%%%%%%%%%%%%%%%%%%%%%%%%%%%%%%
%
%				1-STEP MARKOV MEASURE ON FULL 2-SHIFT
%
%%%%%%%%%%%%%%%%%%%%%%%%%%%%%%%%%%%%%%%%%%%%%%%%

\subsection{$1$-step Markov measures on the full $2$-shift}
In this example we consider $1$-step Markov measures on the full $2$-shift. $P$ is dependent on two variables, $P_{00}$ and $P_{11}$. $P$ and $p$ are given by
\be
P=\left(\begin{array}{cc}
P_{00}&1-P_{00}\\
1-P_{11}&P_{11}\end{array}\right)\quad\text{and}\quad
 p=\left(
 \frac{1-P_{11}}{2-P_{00}-P_{11}}, \frac{1-P_{00}}{2-P_{00}-P_{11}}
\right).
\en
Table~\ref{f2markov2step} contains calculations for $1$-step Markov measures on the full $2$-shift. There are two measures that maximize $\Int_\mu$, both of which lie on a boundary plane. We know entropy has a maximum value of $\log 2$ when the measure is Bernoulli. This is also the measure that maximizes $\Asc_\mu$ with a value of $(\log 2)/2$. 

\begin{table}[h]
\begin{center}
\begin{tabular}{@{}lllll@{} }
\toprule
$P_{00}$&$P_{11}$&$h_\mu$&$\Asc_\mu$&$\Int_\mu$\\
\midrule
0.5&0.5&{\bf 0.693}&{\bf0.347}&0\\
0.216&0&0.292&0.208&{\bf0.124}\\
0&0.216&0.292&0.208&{\bf 0.124}\\
0.905&0.905&0.315&0.209&0.104\\
\bottomrule\\
\end{tabular}

\caption{$1$-step Markov measures on the full $2$-shift\label{f2markov2step}}
\end{center}
\end{table}

The left graph in Figure~\ref{1stepf2in} shows $\Asc_\mu$ for $1$-step Markov measures on the full $2$-shift. We observe that this plot is strictly convex and therefore has a unique measure of maximal average sample complexity occurring when $P_{00}=P_{11}=0.5$. This is the same as the measure of maximal entropy. The measure-theoretic average sample complexity for this measure on the full $2$-shift is $(\log 2)/2$, which is equal to the topological average sample complexity of the full $2$-shift with respect to the cover by rank $0$ cylinder sets.
\begin{figure}
\begin{center}
\includegraphics[width=.4\textwidth,keepaspectratio]{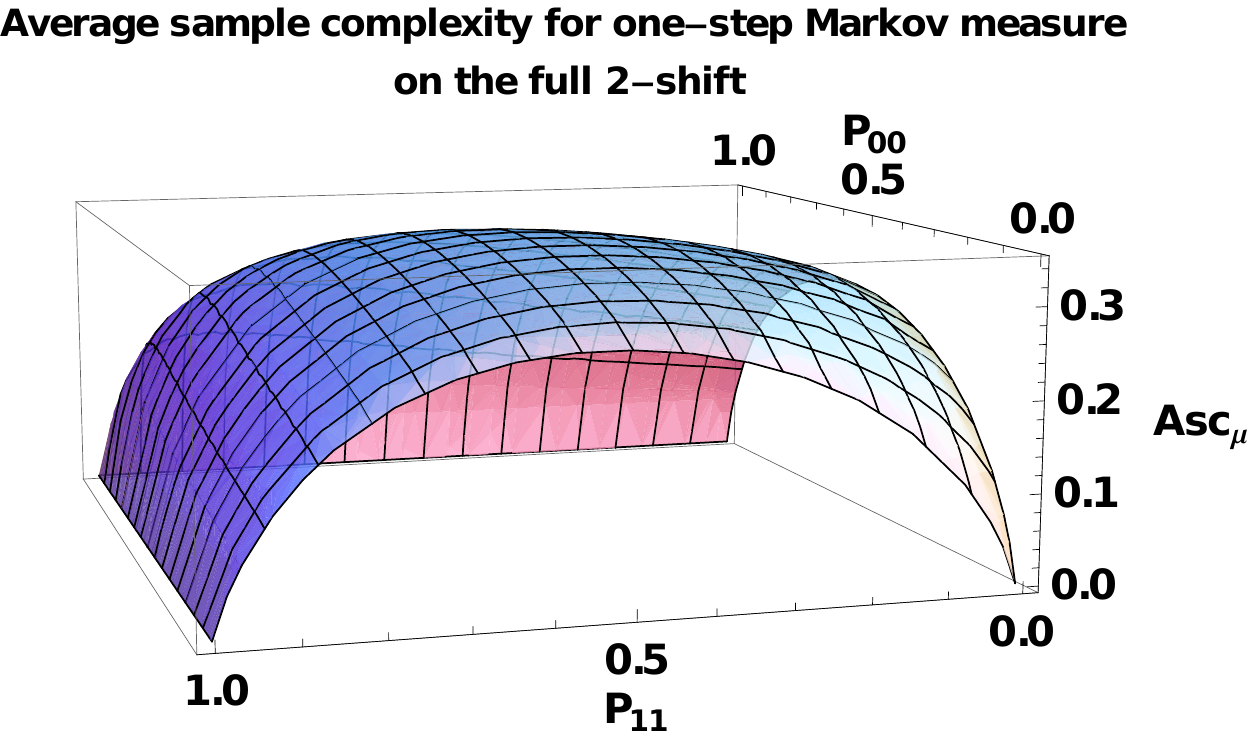}\quad\qquad
\includegraphics[width=.4\textwidth,keepaspectratio]{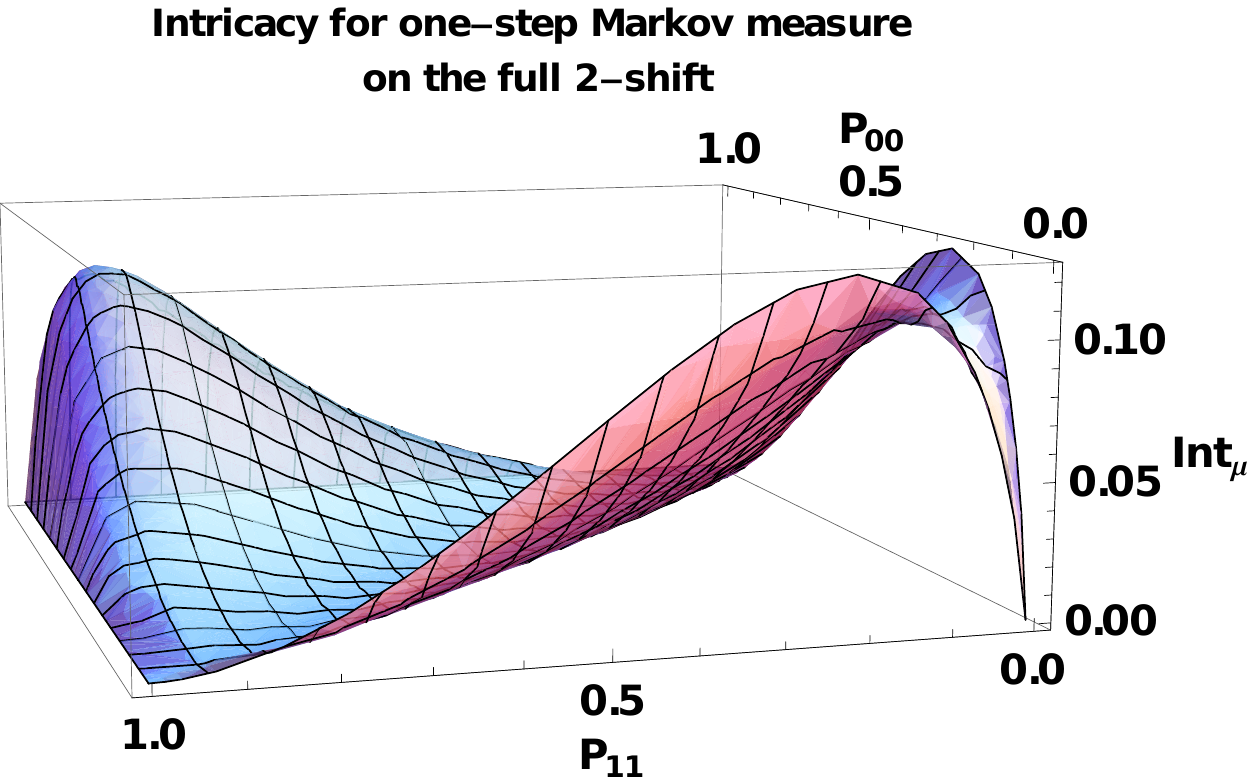}
\end{center}
\caption{$\Asc_\mu$ and $\Int_\mu$ for $1$-step Markov measures on the full $2$-shift\label{1stepf2in} }
\end{figure}
The fourth measure shown in Table~\ref{f2markov2step} is interesting because it is a fully supported local maximum  for $\Int_\mu$. This can be seen in the right graph of Figure~\ref{1stepf2in}, which shows $\Int_\mu$ for $1$-step Markov measures on the full $2$-shift. The absolute maxima of $\Int_\mu$ occur in the planes $P_{00}=0$ and $P_{11}=0$. The full $2$-shift restricted to these planes represents proper subshifts of the full  $2$-shift isomorphic to the golden mean shift, which we discuss in the next example. Figure~\ref{1stepf2intcaseplots} shows the boundary plane $P_{11}=0$  for the intricacy in order better to view the maximum. 
It appears that there is an {\em interior local maximum} for $\Int_\mu$ among $1$-step Markov measures.

We also observe that measure-theoretic intricacy is $0$ when $P_{00}=1-P_{11}$. We prove this using Equation~\ref{1stepmarkasc} with the simplified matrix and fixed vector
\be
P=\left(\begin{array}{cc}
P_{00}&1-P_{00}\\
P_{00}&1-P_{00}
\end{array}\right)\quad\text{and}\quad p=(P_{00},1-P_{00}).
\en
We show $2\Asc_\mu=h_\mu$ and thus $\Int_\mu=0$. Since $P^i=P$ for all $i=1,2,\dots$, and 
\be
\sum_{j,k=0}^{1}p_j(P^i)_{jk}\log(P^i)_{jk}=P_{00}\log P_{00}+(1-P_{00})\log(1-P_{00})=-h_{\mu_{P,p}}(\sigma),
\en 
we have
\be
\Asc_{\mu_{P,p}}(\mathcal{A}^\mathbb{Z},\alpha,\sigma)=-\frac{1}{2}\sum_{i=1}^\infty\frac{1}{2^i}\sum_{j,k=0}^{1}p_j(P^i)_{jk}\log(P^i)_{jk}=\frac{1}{2}\sum_{i=1}^\infty\frac{1}{2^i}h_{\mu_{P,p}}(\sigma)=\frac{1}{2}h_{\mu_{P,p}}(\sigma).
\en

Figure~\ref{entcombof21step} shows the graph of $h_\mu$ on the left and a combined plot on the right, which, in order from top to bottom, shows $h_\mu$, $\Asc_\mu$, and $\Int_\mu$. Each graph is symmetric about the plane $P_{00}=P_{11}$.

\begin{figure}
\begin{center}
\includegraphics[width=.4\textwidth,keepaspectratio]{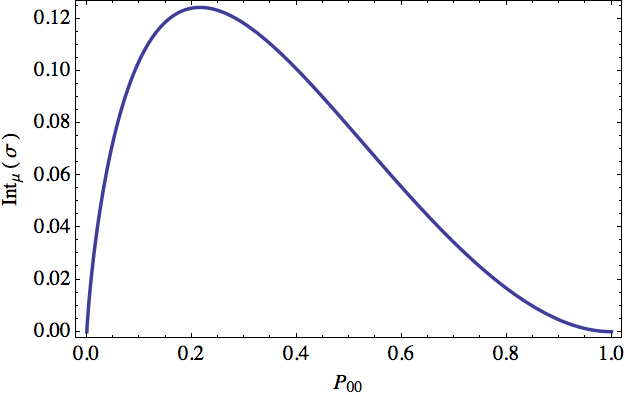}
\caption{$\Int_\mu$ for $1$-step Markov measures on the full $2$-shift with $P_{11}=0$\label{1stepf2intcaseplots}}
\end{center}
\end{figure}

\begin{figure}\label{entcomboplotsfulltwoshift}
\includegraphics[width=.4\textwidth,keepaspectratio]{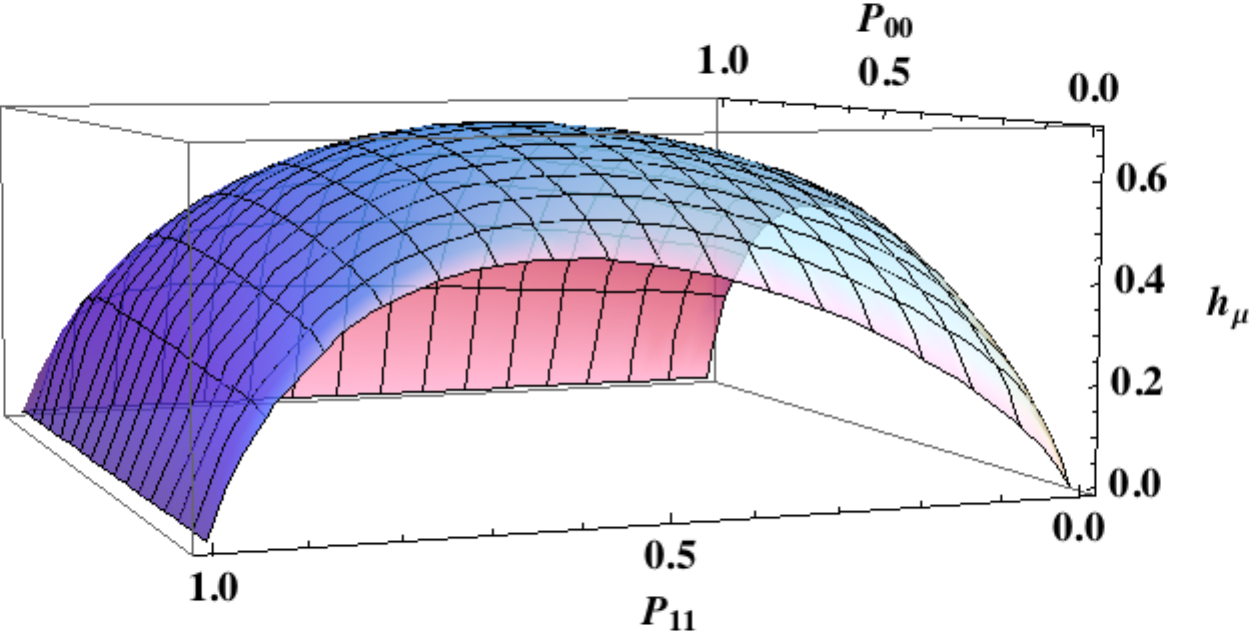}\quad\qquad
\includegraphics[width=.4\textwidth,keepaspectratio]{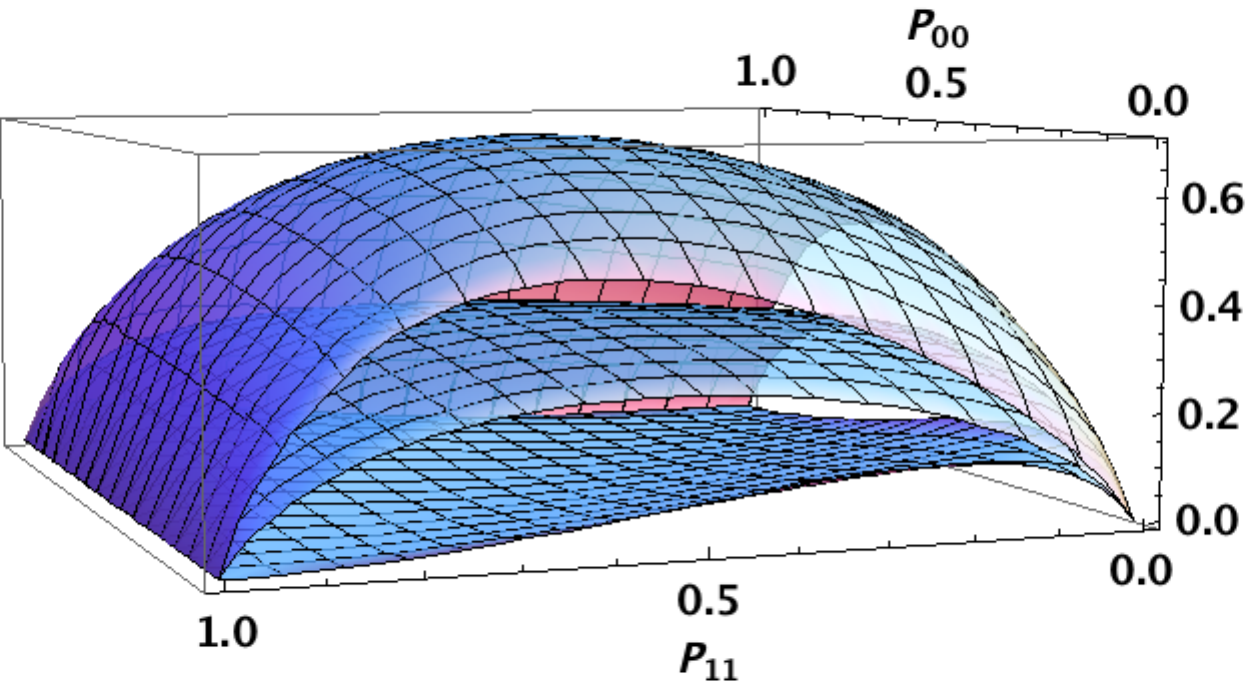}
\caption{$h_\mu$ for $1$-step Markov measures on the full $2$-shift\label{entcombof21step}}
\end{figure}

The unique measure of maximal entropy occurs when $P_{00}=P_{11}=0.5$ and has entropy $\log 2$. Analysis of the graphs of $\Asc_\mu$ and $\Int_\mu$ for $1$-step Markov measures on the full $2$-shift leads to the following conjectures.

\begin{conjecture}
For each $k\ge1$, there is a unique $k$-step Markov measure $\mu_k$ on the full $2$-shift that maximizes $\Asc_{\mu}$ among all $k$-step Markov measures.
\end{conjecture}
We base this conjecture on the observation of convexity in the graph of $\Asc_\mu$ for $1$-step Markov measures on the full $2$-shift.

\begin{conjecture}
For each $k\ge1$, there are two $k$-step Markov measures on the full $2$-shift that maximize $\Int_{\mu}$ among all $k$-step Markov measures. They are not fully supported.
\end{conjecture}

\begin{conjecture}
There is a $1$-step Markov measure on the full $2$-shift that gives a fully supported local maximum for $\Int_{\mu}$ among all $1$-step Markov measures.
\end{conjecture}

%%%%%%%%%%%%%%%%%%%%%%%%%%%%%%%%%%%%%%%%%%%%%%%%
%
%				1-STEP MARKOV MEASURES ON GMS
%
%%%%%%%%%%%%%%%%%%%%%%%%%%%%%%%%%%%%%%%%%%%%%%%%

\subsection{$1$-step Markov measures on the golden mean shift}
For $1$-step Markov measures on the golden mean shift, $P$ and $p$ depend on the single parameter $P_{00}$: 
\be
P=\left(\begin{array}{cc}
P_{00}&1-P_{00}\\
1&0\end{array}\right)\quad\text{and}\quad
 p=\left(
 \frac{1}{2-P_{00}},1-\frac{1}{2-P_{00}}\right).
\en
The measure of maximal entropy occurs when $P_{00}=1/\phi$, where $\phi$ is the golden mean, and the measure-theoretic entropy for this measure is $h_{\mu_{P,p}}(\sigma)=\log \phi$. 

Table~\ref{gmsmarkov1step} contains calculations for different $1$-step Markov measures on the golden mean shift. 

Figure~\ref{onestepmarkovgmsmulticurves} includes two graphs for $1$-step Markov measures on the golden mean shift with $P_{00}$ as the horizontal axis. The graph on the left includes six curves. Five curves are plots of the measure-theoretic average sample complexity function of $n$ for $n=2,\dots,6$ computed using Definition~\ref{mtnewcompdefs}. The sixth is a plot using Equation~\ref{1stepmarkasc}. This graph shows that the average sample complexity functions quickly approach their limit $\Asc_\mu$. As $P_{00}$ approaches $1$, the functions become better approximations for $\Asc_\mu$.

The graph on the right has plots of $h_\mu$, $\Asc_\mu$ and $\Int_\mu$ found using Equation~\ref{1stepmarkasc}. Circles mark what appear to be the unique maxima of each curve.  The maxima among $1$-step Markov measures of $\Asc_\mu$, $\Int_\mu$, and $h_\mu$ all seem to be achieved by different measures $\mu$.

\begin{table}
\begin{center}
\begin{tabular}{@{}llll@{} }
\toprule
$P_{00}$&$h_\mu$&$\Asc_\mu$&$\Int_\mu$\\
\midrule
0.618&{\bf 0.481}&0.266&0.051\\
0.533&0.471&{\bf 0.271}&0.071\\
0.216&0.292&0.208&{\bf 0.124}\\
\bottomrule\\
\end{tabular}
\caption{$1$-step Markov measures on the golden mean shift \label{gmsmarkov1step}}
\end{center}
\end{table}

\begin{figure}
\begin{center}
\includegraphics[width=.48\textwidth,keepaspectratio]{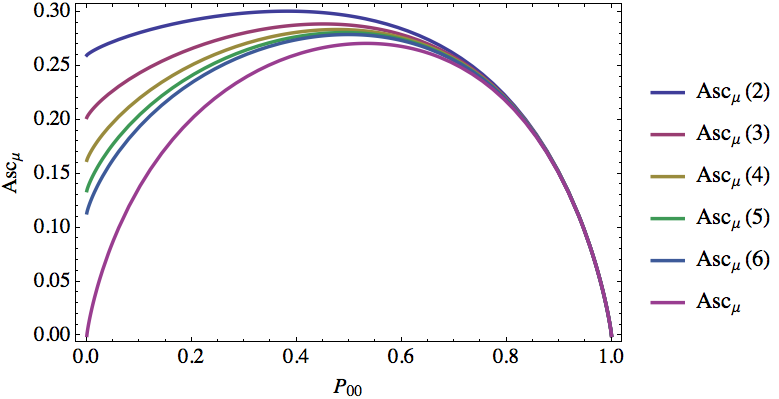}
\includegraphics[width=.45\textwidth,keepaspectratio]{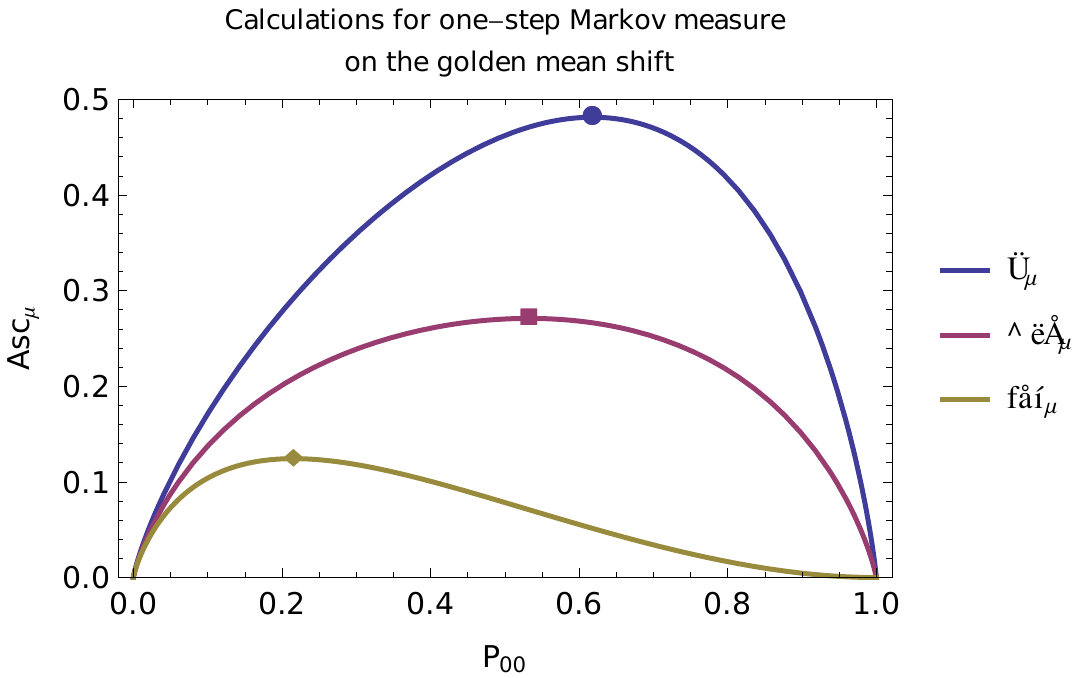}
\caption{$1$-step Markov measures on the golden mean shift \label{onestepmarkovgmsmulticurves}}
\end{center}
\end{figure}

%%%%%%%%%%%%%%%%%%%%%%%%%%%%%%%%%%%%%%%%%%%%%%%%
%
%				2-STEP MARKOV MEASURES GMS
%
%%%%%%%%%%%%%%%%%%%%%%%%%%%%%%%%%%%%%%%%%%%%%%%%

\newpage
\subsection{$2$-step Markov measures on the golden mean shift}
Here we consider $2$-step Markov measures on the golden mean shift. In this case we have two parameters. We let $P_{000}$ and $P_{100}$ be the probability of going from $00$ to $00$ and from $10$ to $00$ respectively. $P$ and $p$ are given by 
\be
P=\left(\begin{array}{ccc}
P_{000}&1-P_{000}&0\\
0&0&1\\
P_{100}&1-P_{100}&0
\end{array}\right)
\en
and
\be p=\left(
 -\frac{P_{100}}{2 P_{000}-P_{100}-2},  \frac{P_{100}}{2 \left(2 P_{000}-P_{100}-2\right)}+0.5, 
   \frac{P_{100}}{2 \left(2 P_{000}-P_{100}-2\right)}+0.5 \\
\right)
\en
\begin{table}
\begin{center}
\begin{tabular}{@{}lllll@{} }
\toprule
$P_{000}$&$P_{100}$&$h_\mu$&$\Asc_\mu$&$\Int_\mu$\\
\midrule
0.618&0.618&{\bf 0.481}&0.266&0.051\\
0.483&0.569&0.466&{\bf 0.272}&0.078\\
0&0.275&0.344&0.221&{\bf 0.167}\\
\bottomrule\\
\end{tabular}

\caption{$2$-step Markov measures on the golden mean shift \label{gmsmarkov2step}}
\end{center}
\end{table}

Table~\ref{gmsmarkov2step} and the plots in Figures~\ref{ascint2stepgms} and \ref{entcombo2stepgms} are similar to those in the previous examples. As expected, the maximal $h_\mu$ is $\log \phi$ as it was for  $1$-step Markov measures on the golden mean shift. 

\begin{figure}
\begin{center}
\includegraphics[width=.4\textwidth,keepaspectratio]{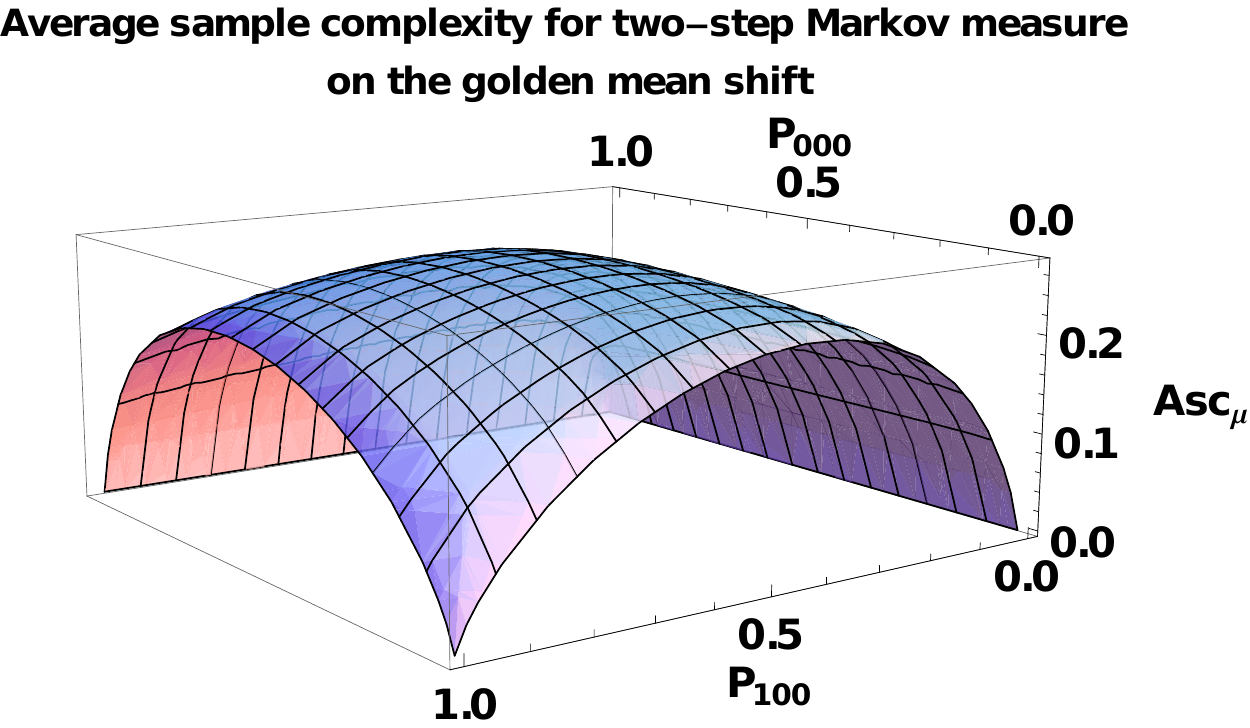}\quad
\includegraphics[width=.4\textwidth,keepaspectratio]{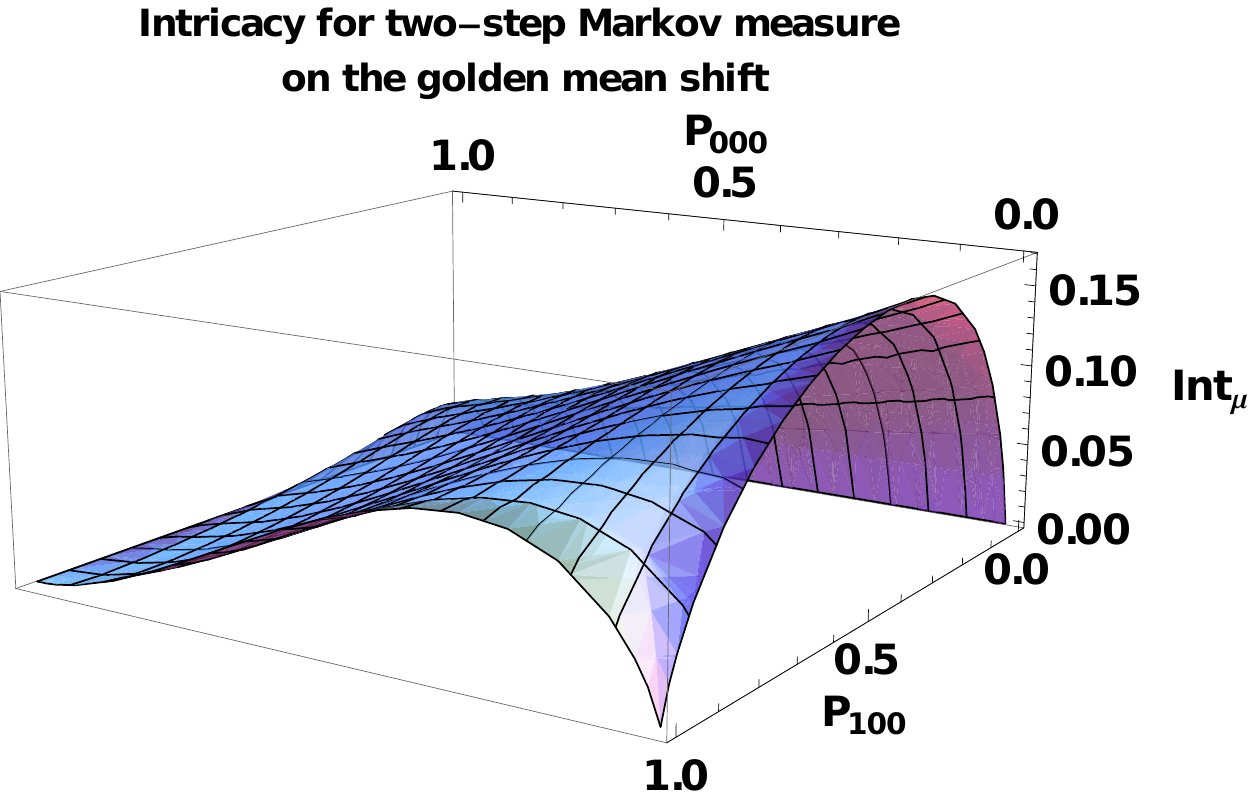}
\end{center}
\caption{$\Asc_\mu$ and $\Int_\mu$ for $2$-step Markov measures on the golden mean shift\label{ascint2stepgms}}
\end{figure}

The graph of $\Asc_\mu$ as a function of the parameters of $2$-step Markov shifts appears strictly convex, as was the case for $1$-step Markov measures on the full $2$-shift; this gives evidence for the existence of a unique maximizing measure. The maximum for $\Int_\mu$ is not fully supported and  occurs on the plane $P_{000}=1$. The maximum values of both $\Asc_\mu$ and $\Int_\mu$ strictly increase as we go from $1$-step Markov measures on the golden mean shift to $2$-step Markov measures on the golden mean shift.  There is no reason to expect that these values will not continue to increase as we move to higher $k$-step Markov measures on the golden mean shift, leading to the following conjectures.
\begin{figure}
\begin{center}
\includegraphics[width=.43\textwidth,keepaspectratio]{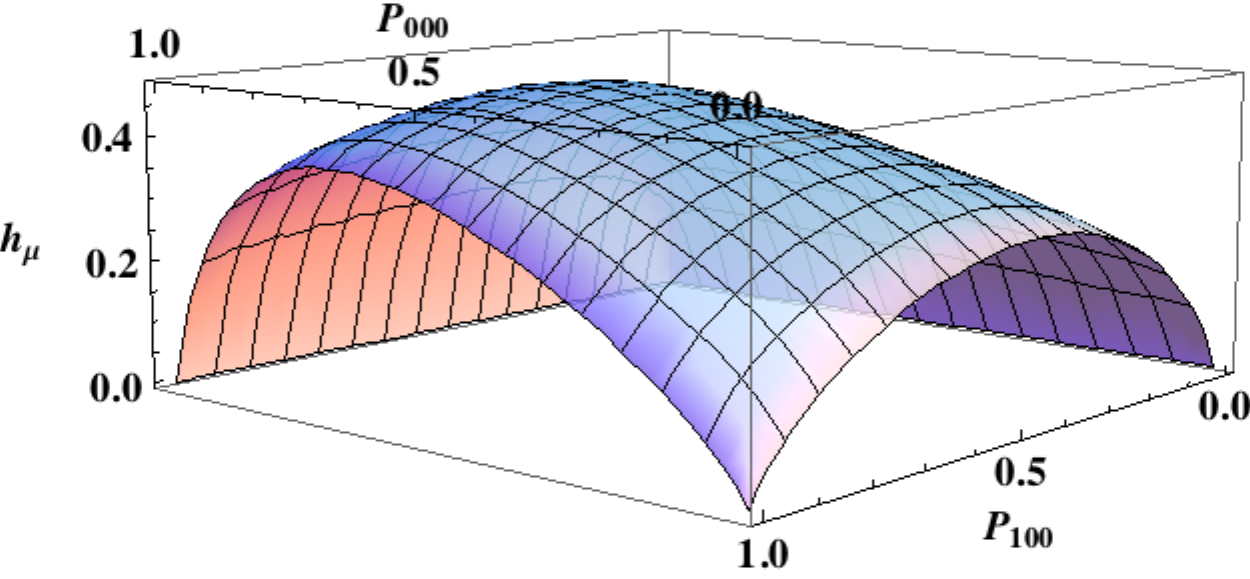}\qquad
\includegraphics[width=.4\textwidth,keepaspectratio]{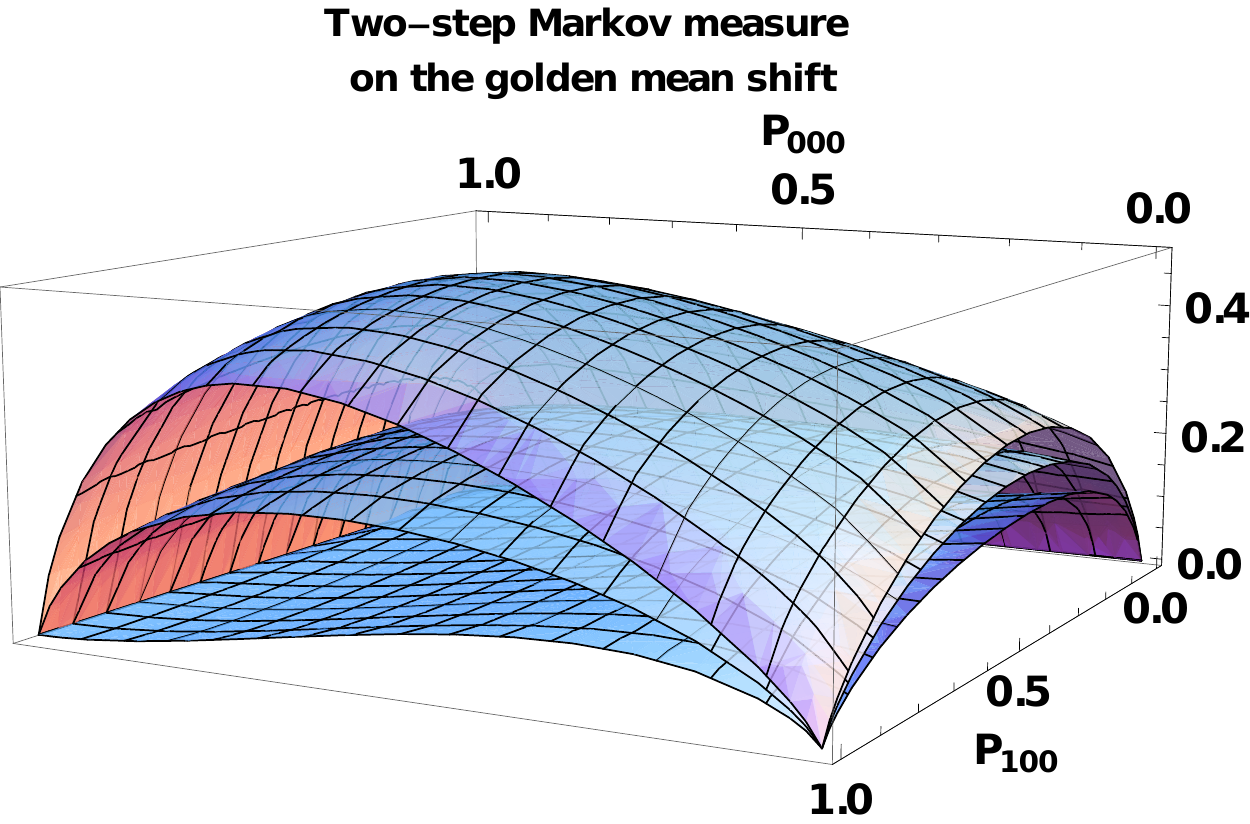}
\end{center}
\caption{$h_\mu$ for $2$-step Markov measures on the golden mean shift\label{entcombo2stepgms}}
\end{figure}
\begin{conjecture}
For each $k\ge1$, there is a unique $k$-step Markov measure $\mu_k$ on the golden mean shift that maximizes $\Asc_{\mu}$ among all $k$-step Markov measures. Furthermore, if $k_1\ne k_2$ then $\Asc_{\mu_{k_1}}\ne\Asc_{\mu_{k_2}}$.
\end{conjecture}

\begin{conjecture}
On the golden mean shift there is a unique measure of maximum $\Asc_\mu$, and it is not Markov of any order.
\end{conjecture}

%%%%%%%%%%%%%%%%%%%%%%%%%%%%%%%%%%%%%%%%%%%%%%%%
%
%				FUTURE DIRECTIONS
%
%%%%%%%%%%%%%%%%%%%%%%%%%%%%%%%%%%%%%%%%%%%%%%%%

\section{Questions}
\subsection{Maximal measures}
On every  irreducible shift of finite type there is a unique measure of maximal entropy, the Shannon-Parry measure. It is a {Markov} measure determined by the transition matrix of the SFT. In the preceding section we saw evidence that perhaps shifts of finite type (and maybe also many other topological dynamical systems) have unique measures of maximal $\Asc_\mu$, these measures might not be Markov of any order, and measures of maximal $\Int_\mu$ may not be unique and may not be fully supported. Moreover, $\Int_\mu$ might have local maxima which are not global maxima. Is $\Asc_\mu$ a convex function of the parameters defining a $k$-step Markov measure, or maybe even of $\mu$ itself? Is there a variational principle, which might say that for a subshift $(X,T)$ with partition $\alpha$ into rank $0$ cylinder sets (and corresponding cover $\mathscr{U}(\alpha)$),
$\sup_\mu\Asc_\mu(X,\alpha,T)=\Asc(X,\mathscr{U}(\alpha),T)$? 
 In the topological case the first-return map $T_{X \times A}$ is not continuous nor expansive nor even defined on all of $X \times A$ in general, so known results about measures of maximal entropy and equilibrium states do not apply.
 To maximize $\Int$, there is the added problem of the minus sign in 
$\Int(X,\mathscr{U},T)=2\Asc(X,\mathscr{U},T)-\htop(X,\mathscr{U},T)$.
If high intricacy, which we have proposed to think of as organized flexibility, really is a desirable property, presumably systems that can evolve through trial and error will seek to maximize it; thus understanding maximizing measures, and their generalizations when a potential function is present, can have important applications.

\subsection{Improved formulas and computational methods}
Equation~\ref{zrasca} gives us a formula for computing the average sample complexity of rank zero cylinder sets for certain shifts of finite type and the uniform system of coefficients. We need formulas to calculate average sample complexity for all shifts of finite type, other subshifts, and indeed other types of dynamical systems. We also need methods aside from brute force to get good approximations, in less computation time, for average sample complexity and intricacy. 

%%%%%%%%%%%%%%%%%%%%%%%%%%%%%%%%%%%%%%%%%%%%%%%%
%
%					Further analysis of SFTs				
%
%%%%%%%%%%%%%%%%%%%%%%%%%%%%%%%%%%%%%%%%%%%%%%%%

\subsection{Further analysis of shifts of finite type}
Suppose a shift of finite type, $X$, has square positive adjacency matrix. We know that the intricacy of $X$ with respect to rank $0$ cylinder sets using the uniform system of coefficients depends only on $|\mathscr{L}_{n^*}(X)|$, i.e. its complexity function (see Theorem~\ref{squarethm}), and therefore two shifts of finite type with square positive adjacency matrices and the same complexity functions have the same intricacy (and intricacy functions). 
We have examples of shifts of finite type with the same complexity functions but different intricacy functions. In these examples the smallest power for which the adjacency matrices are positive differ. Under what conditions will two shifts of finite type with the same complexity functions have the same intricacy functions? Maybe two shifts of finite type will have the same intricacy functions (with respect to rank $0$ cylinder sets and the uniform system of coefficients) exactly when they have the same complexity functions and the same smallest power for which their adjacency matrices are positive.

%%%%%%%%%%%%%%%%%%%%%%%%%%%%%%%%%%%%%%%%%%%%%%%%
%
%				HIGHER-DIMENSIONAL SHIFTS
%
%%%%%%%%%%%%%%%%%%%%%%%%%%%%%%%%%%%%%%%%%%%%%%%%

\subsection{Higher-dimensional shifts and general group actions}
The definitions of average sample complexity and intricacy generalize naturally to higher-dimensional subshifts,  general group (or semigroup) actions, and networks. Systems such as these would have to be considered if one wanted to take into account underlying system geometry (e.g., connections among neurons), but since the computation and understanding of entropy is already problematic in these settings, progress is not likely to be easy. 

\subsection{Relative $\Asc$ and $\Int$}
When there is a factor map $\pi : (X,T) \to (Y,S)$, there are definitions of relative entropy, relatively maximal measures, relative equilibrium states, etc. \cite{LedrappierWalters1977,BoyleTuncel1984,Walters1986,PetersenQuasShin2003,AHJ2014}. 
All of this could be generalized to $\Asc$ and $\Int$.

%%%%%%%%%%%%%%%%%%%%%%%%%%%%%%%%%%%%%%%%%%%%%%%%
%
%				SUBSETS OF MAXIMAL ASC AND INT
%
%%%%%%%%%%%%%%%%%%%%%%%%%%%%%%%%%%%%%%%%%%%%%%%%

\subsection{Maximizing subsets}
For a topological system $(X,T)$ and a cover $\mathscr{U}$ of $X$, for each $n\ge 1$ we would like to find the subset(s) $S\subset n^*$ that maximize $\log N(\mathscr{U}_S)$ and $\log (N(\mathscr{U}_S)N(\mathscr{U}_{\setcomp{S}})/N(\mathscr{U}_{n^*}))$. For shifts of finite type with positive square adjacency matrix and covers by rank zero cylinder sets, it is a consequence of Proposition~\ref{nsprop1} that $\log N(\mathscr{U}_S)$ is maximized for the subset $S\subset n^*$, $S=\{0,2,4,6,\dots,n-1\}$ for $n$ even and $S=\{0,2,\dots,n-2\}$ or $S=\{1,3,\dots, n-1\}$ for $n$ odd. 
Similarly, for a measure-preserving system $(X,\mathscr{B},\mu,T)$ and partition $\alpha$ of $X$, for each $n\ge 1$ we would like to find the subset(s) $S\subset n^*$ that maximize $H_\mu(\alpha_S)$ and $H_\mu(\alpha_S)+H_\mu(\alpha_{\setcomp{S}})-H_\mu(\alpha_{n^*})$. Finding maximizing subsets $S\subset n^*$ could lead to improved computational methods for estimating average sample complexity and intricacy by allowing us to focus  on subsets that have the greatest effect. 
%%%%%%%%%%%%%%%%%%%%%%%%%%%%%%%%%%%%%%%%%%%%%%%%
%
%					ANALYSIS OF MORE EXAMPLES
%
%%%%%%%%%%%%%%%%%%%%%%%%%%%%%%%%%%%%%%%%%%%%%%%%

%%%%%%%%%%%%%%%%%%%%%%%%%%%%%%%%%%%%%%%%%%%%%%%%
%
%					ENTROPY IS ONLY FINITELY OBSERVABLE
%
%%%%%%%%%%%%%%%%%%%%%%%%%%%%%%%%%%%%%%%%%%%%%%%%

\subsection{Entropy is the only finitely observable invariant}\label{finobs}
In~\cite{OrnsWeiss07}, D. Ornstein and B. Weiss show that any finitely observable measure-theoretic isomorphism invariant is necessarily a continuous function of entropy. (A function $J$ with values in some metric space, defined for all finite-valued, stationary, ergodic processes, is said to be \emph{finitely observable} if there is a sequence of functions $S_n(x_1,\dots, x_n)$ that for every process $\mathscr{X}$ converges to $J(\mathscr{X})$ for almost every realization $x_1,x_2,\dots)$ of $\mathscr{X}$.) 
In view of this result it is perhaps not surprising that the invariants $\sup_\alpha\Asc_\mu(X,\alpha,T)$ and $\sup_{\alpha}\Int(X,\alpha,T)$ are both equal to measure-theoretic entropy.
But we do not know whether $\Asc_\mu$ and $\Int_\mu$ are finitely observable---for one thing, we lack an analogue of the Shannon-McMillan-Breiman Theorem for these quantities.
Further, that
the supremum over open covers of $\Asc$ and $\Int$ yields the usual topological entropy suggests that there might be some kind of topological analogue of this Ornstein-Weiss theorem. 
%%%%%%%%%%%%%%%%%%%%%%%%%%%%%%%%%%%%%%%%%%%%%%%%
%
%					ALTERNATE ASC
%
%%%%%%%%%%%%%%%%%%%%%%%%%%%%%%%%%%%%%%%%%%%%%%%%

\subsection[\vspace{-1.5ex}Alternate definition of the average sample complexity function analogous\\ to the complexity function of a sequence]{Alternate definition of the average sample complexity function}\label{altasc}
Recall that for a subshift $N(S)$ gives the number of words seen at the places in the set $S \subset n^*$ among all sequences in $X$, and the average sample complexity function is an average over all $S$ of $\log N(S)$ (Definition \ref{def:ascintfunctions}). In analogy with the complexity function $p_X(n)$ of a subshift, which gives the number of words of length $n$ found among all sequences in the subshift, one could consider an {\em alternate sample complexity function}
\be
\Alt_X(n)= \frac{1}{2^n}\sum_{S \subset n^*} N(S).
\en
This quantity, and its corresponding alternate intricacy function, would provide a different measure of the complexity of a system.

\subsection{Average sample complexity of an infinite or finite word}
The average sample complexity function of a fixed sequence $x$ on a finite alphabet may be defined to be that of its orbit closure, and similarly for the alternate average sample complexity just defined. For a finite sequence $u=u_0 \dots u_{m-1}$, for each $n=1,2,\dots,m$ and $S \subset n^*$ one could define $N_u(S)$ to be the number of different words seen along the places in $S$ in all the subwords $u_i \dots u_{i+n-1}, 0\leq i \leq m-n$ and form averages
\be
\Asc_u(n)=\frac{1}{n}\frac{1}{2^n} \sum_{S\subset n^*} \log N_u(S) \quad\text{ or }\quad 
\Alt(n) = \frac{1}{2^n}\sum_{S \subset n^*} N_u(S).
\en
One could compare finite words according to these and other complexity measures, and for infinite words one might ask, for example, which are the aperiodic sequences with minimum $\Asc(n)$ or $\Alt(n)$. 

\subsection{Partition $n^*$ into $m$ subsets}
Our definition for intricacy in dynamical systems is based on partitioning the set $n^*$ into a subset $S$ and its complement $\setcomp{S}$. We could also consider partitioning $n^*$ into more than two subsets, disjoint $S_1,S_2,\dots, S_m$  whose union is $n^*$. Let $\mathscr{S}(m)\subset n^*$ denote the set of all partitions of $n^*$ into $m$ subsets and denote by $c_{\mathscr{S}(m)}$ a new weighting factor depending on the partition $\mathscr{S}(m)$. For $(X,T)$ a topological dynamical system and $\mathscr{U}$ an open cover of $X$, one may define the \emph{$m$-intricacy of $X$ with respect to $\mathscr{U}$} to be 
\begin{equation}
m\text{-}\Int(X,\mathscr{U},T)=\lim_{n\rightarrow\infty}\frac{1}{n}\sum_{\mathscr{S}(m)\subset n^*} c_{\mathscr{S}(m)} \log\left(\frac{\prod_{S_i\in\mathscr{S}(m)}N(\mathscr{U}_{S_i})}{N(\mathscr{U}_{n^*})}\right).
\end{equation}
The analogous generalization to intricacy functionals was already proposed in \cite{tononi1994measure, BZ12}.
One could also average over $m$.

%%%%%%%%%%%%%%%%%%%%%%%%%%%%%%%%%%%%%%%%%%%%%%%%
%
%					DEFINITIONS BASED ON ROKHLIN ENTROPY
%
%%%%%%%%%%%%%%%%%%%%%%%%%%%%%%%%%%%%%%%%%%%%%%%%
\subsection{Definition based on Rokhlin entropy}
In \cite{seward2014krieger}, Rokhlin entropy is defined for probability preserving group actions as the infimum of the measure-theoretic entropies over countable generating partitions; Rokhlin \cite{Rokhlin1967} had shown that for free ergodic $\mathbb Z$ actions it coincides with ordinary measure-theoretic entropy. 
We may define the \emph{measure-theoretic Rokhlin average sample complexity} based on Rokhlin entropy to be
\be
\Asc_\mu^{\text{Rok}}(X,T)=\inf_\alpha\left\{\Asc_\mu(X,\alpha,T):\alpha \text{ is a countable generating partition}\right\},
\en 
and the topological version,
\be
\Asc^{\text{Rok}}(X,T)=\inf\left\{\Asc(X,\mathscr{U},T):\mathscr{U} \text{ is a topological generator}\right\}.
\en
Since these are now invariants of isomorphism and conjugacy, what are they? Always the ordinary entropy? Always $0$?

%(We thank Tomasz Downarowicz for the suggestion to examine these definitions.)

%%%%%%%%%%%%%%%%%%%%%%%%%%%%%%%%%%%%%%%%%%%%%%%%
%
%					KOSLICKI THOMPSON
%
%%%%%%%%%%%%%%%%%%%%%%%%%%%%%%%%%%%%%%%%%%%%%%%%
\subsection{Application of topological average sample pressure to coding sequence density}
In \cite{koslicki2015coding}, Koslicki and Thompson give a new approach to coding sequence density estimation in genomic analysis based on topological pressure. They use topological pressure as a computational tool for predicting the distribution of coding sequences and identifying gene-rich regions.
In their study, they consider finite sequences on the alphabet $\{A,C,G,T\}$, weight each word of length $3$, and compute the topological pressure as one would for a $3$-block coding of the full $4$-shift. The weighting function (potential function) is found by training parameters so that the topological pressure fits the observed coding sequence density on the human genome.
If one were to make similar computations but replace topological pressure with topological average sample pressure, which takes into account mutual influences among sets of sites, these finer measurements might better detect coding regions or otherwise help to understand the structure of genomes.

%%%%%%%%%%%%%%%%%%%%%%%%%%%%%%%%%%%%%%%%%%%%%%%%
%
%					VARIATIONAL PRINCIPLE
%
%%%%%%%%%%%%%%%%%%%%%%%%%%%%%%%%%%%%%%%%%%%%%%%%

\begin{ack*}
	This paper is based on the UNC-Chapel Hill Ph.D. dissertation of the second author \cite{Wilson2015}, written under the direction of the first author.
	We thank Mike Boyle, J\'{e}r\^{o}me Buzzi, Tomasz Downarowicz, Kevin McGoff, Jean-Paul Thouvenot, and Lorenzo Zambotti for illuminating conversations. 
\end{ack*}

%------------------------------------------------------------------------------------------------------------------------
%								END 
%------------------------

------------------------------------------------------------------------------------------------
\appendix
\section{Average Sample Pressure}\label{asp}
  Given a potential function $f$ we define topological average sample pressure by 
  restricting observations to selected subsets of $n^*$ and averaging, thus
  generalizing topological average sample complexity. We use notation based on the notation for topological pressure in ~\cite[Chapter 9]{waltersergodic}, which should be consulted for background and any unexplained terminology or notations, such as $r(n,\varepsilon), s(n,\varepsilon), p_n, q_n, P(T,f)$, etc., and we follow the scheme of the arguments found there. Recall that a set $E\subset X$ is $(S,\varepsilon)$ \emph{separated} if for each pair of distinct points $x,y\in E$, $d(T^{s_i}x,T^{s_i}y)>\varepsilon$ for some $i=0,\dots, |S|-1$, and $s(S,\varepsilon)$ is the maximum cardinality of a set $E\subset X$ such that $E$ is $(S,\varepsilon)$ separated. Also, a set $E\subset X$ is $(S,\varepsilon)$ \emph{spanning} if for each $x\in X$ there is $y\in E$ with $d(T^{s_i}x,T^{s_i}y)\le\varepsilon$ for all $i=0,\dots, |S|-1$, and $r(S,\varepsilon)$ is the minimum cardinality of an $(S,\varepsilon)$ spanning set of $X$. 
  \begin{definition}
  	Let $T:X\rightarrow X$ be a continuous transformation on a compact metric space. Let $f\in C(X,\mathbb{R})$ and $S\subset n^*$. Define
  	\begin{equation}\label{bigQs}
  	Q_S(T,f,\varepsilon)=\inf\left\{\sum_{x\in F}\exp\left(\sum_{i\in S}f(T^ix)\right):F\text{ is an }(S,\varepsilon)\text{ spanning set for }X\right\}.
  	\end{equation}
  	Then, for a fixed system of coefficients $c_S^n$, the \emph{average sample pressure of $T$ given $f$ and $\varepsilon$}, $\Asp_\varepsilon(T,f)$, is
  	\begin{equation}
  	\Asp_\varepsilon(T,f)=\limsup_{n\rightarrow\infty}\frac{1}{n}\sum_{S\subset n^*}c_S^n\log Q_S(T,f,\varepsilon).
  	\end{equation}
  \end{definition}
  
  \begin{definition}
  	We also define average sample pressure in terms of $(S,\varepsilon)$ separated sets. Let
  	\begin{equation}\label{bigPs}
  	P_S(T,f,\varepsilon)=\sup\left\{\sum_{x\in E}\exp\left(\sum_{i\in S}f(T^ix)\right):E\text{ is an }(S,\varepsilon).\text{ separated set for }X\right\},
  	\end{equation}
  	Then, for a fixed system of coefficients $c_S^n$, the \emph{average sample pressure of $T$ given $f$ and $\varepsilon$}, $\Asp'_\varepsilon(T,f)$, is
  	\begin{equation}
  	\Asp'_\varepsilon(T,f)=\limsup_{n\rightarrow\infty}\frac{1}{n}\sum_{S\subset n^*}c_S^n\log P_S(T,f,\varepsilon).
  	\end{equation}
  	$P_S(T,f,\varepsilon)$ and $\Asp'_\varepsilon(T,f)$ are nondecreasing as $\varepsilon$ decreases.
  	We use $\Asp$ to denote the definition which uses $(S,\varepsilon)$ spanning sets and $\Asp'$ to denote the definition which uses $(S,\varepsilon)$ separated sets since, in general, for a given $\varepsilon$ these may not be equal. 
  \end{definition}
  \begin{proposition} \label{QPprop}
  	Let $T:X\rightarrow X$ be a continuous transformation on a compact metric space. Let $f\in C(X,\mathbb{R})$ and $S\subset n^*$. Given $\varepsilon>0$,
  	\begin{equation}
  	Q_S(T,f,\varepsilon)\le P_S(T,f,\varepsilon).
  	\end{equation}
  \end{proposition}
  \begin{proof}
  	We first notice that in Equation~\ref{bigPs} we can take the supremum over $(S,\varepsilon)$ separated sets, $E$, that are maximal, i.e. if we were to add another point to $E$ then it would no longer be an $(S,\varepsilon)$ separated set. This is because $\exp\left(\sum_{i\in S}f(T^ix)\right)>0$ for all $x\in X$, so the more points $x$ we sum this value over, the larger it gets. 
  	
  	Now, if $E$ is a maximal $(S,\varepsilon)$ separated set then it must be an $(S,\varepsilon)$ spanning set for $X$. This is because given $x\in X\setminus E$ and $y\in E$, if $d(T^{s_i}x,T^{s_i}y)>\varepsilon$, for all $i=0,\dots, |S|-1$, then we could add $x$ to $E$ and $E$ would still be $(S,\varepsilon)$ separated, contradicting that $E$ is a maximal $(S,\varepsilon)$ separated set.
  \end{proof}
  \begin{corollary}\label{asppressupboundcor}
  	\begin{equation}
  	\Asp_\varepsilon(T,f)\le\Asp'_\varepsilon(T,f)\le P(T,f).
  	\end{equation}
  \end{corollary}
  \begin{proof}
  	The first inequality follows directly from Proposition~\ref{QPprop}. The second inequality is true because 
  	\be
  	P_S(T,f,\varepsilon)\le P_n(T,f,\varepsilon)
  	\en
  	for all $S\subset n^*$.
  \end{proof}
  
  Next we give the definition of average sample pressure in terms of open covers.
  \begin{definition}
  	Let $T:X\rightarrow X$ be a continuous transformation on a compact metric space. Let $f\in C(X,\mathbb{R})$ and $S\subset n^*$. If $\mathscr{U}$ is an open cover of $X$, then we define
  	\begin{equation}
  	p_S(T,f,\mathscr{U})=\inf\left\{\sum_{V\in\mathscr{V}}\sup_{x\in V}\exp\left(\sum_{i\in S}f\left(T^ix\right)\right):\mathscr{V} \text{ is a finite subcover of } \mathscr{U}_S\right\}
  	\end{equation}
  	and
  	\begin{equation}
  	q_S(T,f,\mathscr{U})=\inf\left\{\sum_{V\in\mathscr{V}}\inf_{x\in V}\exp\left(\sum_{i\in S}f\left(T^ix\right)\right):\mathscr{V} \text{ is a finite subcover of } \mathscr{U}_S\right\}.
  	\end{equation}

  	Define the the \emph{average sample pressure} of $(X,T)$ and the open cover $\mathscr{U}$ of $X$, given $f$ and a system of coefficients $c_S^n$, by
  	\begin{equation}\label{limaspeq}
  	\Asp(T,f,\mathscr{U})=\limsup_{n\rightarrow\infty}\frac{1}{n}\sum_{S\subset n^*}c_S^n\log q_S(T,f,\mathscr{U}).
  	\end{equation}
  	Similarly, we define another average sample pressure by
  	\begin{equation}\label{eq:limprime}
  	\Asp'(T,f,\mathscr{U})=\limsup_{n\rightarrow\infty}\frac{1}{n}\sum_{S\subset n^*}c_S^n\log p_S(T,f,\mathscr{U}).
  	\end{equation}
  \end{definition}
  
  \begin{proposition}
  	Let $T:X\rightarrow X$ be a continuous transformation on a compact metric space. Let $f\in C(X,\mathbb{R})$ and fix a system of coefficients $c_S^n$. If $\mathscr{U}$ is an open cover of $X$ then
  	\begin{equation}
  	\Asp(T,f,\mathscr{U})\le \Asp'(T,f,\mathscr{U})\le  \lim_{n\rightarrow\infty}\frac{1}{n}\log p_n(T,f,\mathscr{U}).
  	\end{equation}
  \end{proposition}
  \begin{proof}
  	Since $q_S(T,f,\mathscr{U})\le p_S(T,f,\mathscr{U})$, we have $\Asp(T,f,\mathscr{U})\le \Asp'(T,f,\mathscr{U})$. Since $p_S(T,f,\mathscr{U})\le p_n(T,f,\mathscr{U})$ for every $S\subset n^*$, we get 
  	\be
  	\Asp'(T,f,\mathscr{U})\le \limsup_{n\rightarrow\infty}\frac{1}{n}\log p_n(T,f,\mathscr{U}).
  	\en
  \end{proof}
  \begin{lemma}\label{presdissub}
  	Let $T:X\rightarrow X$ be a continuous transformation on a compact metric space, $f\in C(X,\mathbb{R})$, and $S_1,S_2\subset n^*$ disjoint. Let $S=S_1\cup S_2$. If $\mathscr{U}$ is an open cover of $X$, then 
  	\begin{equation}
  	\log p_S(T,f,\mathscr{U})\le \log p_{S_1}(T,f,\mathscr{U})+\log p_{S_2}(T,f,\mathscr{U}).
  	\end{equation}
  \end{lemma}
  \begin{proof}
  	First we show 
  	\be
  	p_S(T,f,\mathscr{U})\le p_{S_1}(T,f,\mathscr{U})\cdot p_{S_2}(T,f,\mathscr{U}).
  	\en
  	For each finite open subcover $\mathscr{V}_1$ of $\mathscr{U}_{S_1}$ and $\mathscr{V}_2$ of $\mathscr{U}_{S_2}$, $\mathscr{V}_1\vee \mathscr{V}_2$ is a finite subcover of $\mathscr{U}_S$ and
  	\be
  	\sum_{A\in\mathscr{V}_1\vee\mathscr{V}_2}\sup_{x\in A}\exp\left(\sum_{i\in S}f(T^ix)\right)\le \sum_{B\in\mathscr{V}_1}\sup_{x\in B}\exp\left(\sum_{i\in S_1}f(T^ix)\right)\cdot\sum_{C\in\mathscr{V}_2}\sup_{x\in C}\exp\left(\sum_{i\in S_2}f(T^ix)\right).
  	\en
  	This shows $p_S(T,f,\mathscr{U})\le p_{S_1}(T,f,\mathscr{U})\cdot p_{S_2}(T,f,\mathscr{U})$ which implies $\log p_S(T,f,\mathscr{U})\le \log p_{S_1}(T,f,\mathscr{U})+\log p_{S_2}(T,f,\mathscr{U})$.
  \end{proof}
  \begin{proposition}
  	Let $\lambda_c$ be a symmetric probability measure on $[0,1]$. Then the limit in Equation~\ref{eq:limprime} exists for $c_S^n=\int_{[0,1]}x^{|S|}(1-x)^{n-|S|}\lambda_c(dx)$ and
  	\be
  	\lim_{n\rightarrow\infty}\frac{1}{n}\sum_{S\subset n^*}c_S^n\log p_S(T,f,\mathscr{U})=\inf_n\frac{1}{n}\sum_{S\subset n^*}c_S^n\log p_S(T,f,\mathscr{U}).
  	\en
  \end{proposition}
  
  The following theorem is needed to show, in Theorem \ref{asppresssupthm}, Corollary \ref{asppresssupcor}, Theorem \ref{aspepslimthm}, and Corollary \ref{ascepslimcor}, following the plan in \cite[pp. 209--212]{waltersergodic}, that taking the limit on $\varepsilon$ yields the same result as taking the supremum over open covers, namely the ordinary topological pressure $P(T,f)$ or ordinary topological entropy $\htop$. The proof follows the proof in \cite[Theorem 9.2, p. 210]{waltersergodic}. We denote the diameter of a cover by $\diam(\mathscr{U})=\sup_{U\in\mathscr{U}}\diam(U)$.
  
  \begin{theorem}\label{qQPpthm}
  	Let $T:X\rightarrow X$ be a continuous transformation on a compact metric space. Let $f\in C(X,\mathbb{R})$ and $S\subset n^*$.
  	\begin{enumerate}[{\normalfont (i)}]
  		\item If $\mathscr{U}$ is an open cover of $X$ with Lebesgue number $\delta$, then $q_S(T,f,\mathscr{U})\le Q_S(T,f,\delta/2)\le P_S(T,f,\delta/2)$.\\
  		\item If $\varepsilon>0$ and $\mathscr{U}$ is an open cover of $X$ such that $\diam(\mathscr{U})\le \varepsilon$, then $Q_S(T,f,\varepsilon)\le P_S(T,f,\varepsilon)\le p_S(T,f,\mathscr{U})$.
  		
  	\end{enumerate}
  \end{theorem}
  %-------------------------------------------------------------------------------------------------------------------------
  \begin{lemma}\label{asplem}
  	Let $T: X\rightarrow X$ be a  continuous transformation on the compact metric space $(X,d)$. Let $f\in C(X,\mathbb{R})$. The following are equal to $\lim_{\varepsilon\rightarrow0^+}\Asp'_\varepsilon(T,f)$:
  	\begin{enumerate}[{\normalfont (i)}]
  		\item $\displaystyle{\lim_{\delta\rightarrow 0^+}\sup_{\mathscr U}\left\{\Asp'(T,f,\mathscr{U}):\diam(\mathscr{U})\le \delta\right\}}$.\\
  		\item $\displaystyle{\lim_{\delta \to 0^+}\sup_{\mathscr U}\{\liminf_{n\rightarrow\infty}\frac{1}{n}\sum_{S\subset n^*} c_S^n\log q_S(T,f,\mathscr{U}): \diam(\mathscr U) \leq \delta\}}$.\\
  		\item $\displaystyle{\lim_{\delta \to 0^+}\sup_{\mathscr U}\{\limsup_{n\rightarrow\infty}\frac{1}{n}\sum_{S\subset n^*} c_S^n\log q_S(T,f,\mathscr{U}): \diam(\mathscr U) \leq \delta\}}=
  		\lim_{\delta \to 0^+} \sup_{\mathscr U}\{\Asp(T,f,\mathscr U): \diam(\mathscr U) \leq \delta\}$.\label{limsupaspprop}\\
  		\item $\sup_{\mathscr{U}}\Asp (T,f,\mathscr{U})$.
  	\end{enumerate}
  \end{lemma}
  \begin{proof}
  	(i) Given $\delta>0$, for any open cover $\mathscr{U}$ of $X$ with $\diam(\mathscr{U})\le\delta$, $P_S(T,f,\delta)\le p_S(T,f,\mathscr{U})$ by Theorem~\ref{qQPpthm} (ii). 
  	Thus $\Asp'_\delta (T,f)\le\sup_{\mathscr{U}}\{\Asp'(T,f,\mathscr{U}):\diam(\mathscr{U})\le \delta\}$. 
  	
  	Conversely, let $\mathscr{U}$ be an open cover with Lebesgue number $\delta$. 
  	Fix $n$ and $S \subset n^*$.
  	Then by Theorem~\ref{qQPpthm} (i), $q_S(T,f,\mathscr{U})\le P_S(T,f,\delta/2)$. If
  	\be
  	\uptau_\mathscr{U}=\sup\{|f(x)-f(y)|:d(x,y)\le \diam(\mathscr{U})\}, \text{ then }
  	\en
  	\be
  	p_S(T,f,\mathscr{U})\le e^{|S|\uptau_\mathscr{U}}q_S(T,f,\mathscr{U}).
  	\en
  	Hence,
  	\be
  	p_S(T,f,\mathscr{U})\le e^{|S|\uptau_\mathscr{U}}P_S(T,f,\delta/2).
  	\en
  	This implies 
  	\be
  	\log p_S(T,f,\mathscr{U})\le |S|\uptau_\mathscr{U}+\log P_S(T,f,\delta/2) \leq |S| \uptau_\mathscr{U} + \lim_{\delta \to 0^+} \log P_S (T,f, \delta/2) ,
  	\en
  	and therefore, since the average $S$ (using general coefficienmts $c_S^n$) has size $n/2$,
  	\be
  	\Asp'(T,f,\mathscr{U})\le \frac{\uptau_\mathscr{U}}{2}+\lim_{\varepsilon\rightarrow 0^+}\Asp'_\varepsilon(T,f).
  	\en
  	Since 
  	$\sup_\mathscr{U}\uptau_\mathscr{U} \to 0$ as $\diam(\mathscr U) \to 0$, 
  	\be
  	\lim_{\delta\rightarrow 0^+}\sup_{\mathscr{U}}\{\Asp'(T,f,\mathscr{U}):\diam(\mathscr{U})\le \delta\}\le\lim_{\varepsilon\rightarrow 0^+}\Asp'_\varepsilon(T,f).
  	\en

  	(ii) and (iii)  We know $q_S(T,f,\mathscr{U})\le p_S(T,f,\mathscr{U})\le e^{|S|\uptau_\mathscr{U}}q_S(T,f,\mathscr{U})$ for all open covers $\mathscr{U}$ of $X$, so
  	\be
  	e^{-|S|\uptau_\mathscr{U}}p_S(T,f,\mathscr{U})\le q_S(T,f,\mathscr{U})\le p_S(T,f,\mathscr{U}).
  	\en
  	Therefore,
  	\be\label{aspsuplem}\
  	\begin{aligned}
  		-\frac{\uptau_\mathscr{U}}{2}+\Asp'(T,f,\mathscr{U})&\le \liminf_{n\rightarrow\infty}\frac{1}{n}\sum_{S\subset n^*} c_S^n\log q_S(T,f,\mathscr{U})\\
  		&\le \limsup_{n\rightarrow\infty}\frac{1}{n}\sum_{S\subset n^*} c_S^n\log q_S(T,f,\mathscr{U})\\
  		&\le \Asp'(T,f,\mathscr{U}).
  	\end{aligned}
  	\en
  	Thus, using (i) above,
  	\be
  	\begin{aligned}
  		\lim_{\varepsilon\rightarrow 0^+}\Asp'_\varepsilon (T,f)&=\lim_{\delta \to 0^+} \sup_{\mathscr U}\{\Asp'(T,f,\mathscr U): \diam(\mathscr U) \leq \delta\}\\	 
  		&=\lim_{\delta \to 0^+}\sup_{\mathscr U}\{\liminf_{n\rightarrow\infty}\frac{1}{n}\sum_{S\subset n^*} c_S^n\log q_S(T,f,\mathscr{U}): \diam(\mathscr U) \leq \delta\}\\
  		&=\lim_{\delta \to 0^+}\sup_{\mathscr U}\{\limsup_{n\rightarrow\infty}\frac{1}{n}\sum_{S\subset n^*} c_S^n\log q_S(T,f,\mathscr{U}): \diam(\mathscr U) \leq \delta\}.
  	\end{aligned}
  	\en

  	(iv) Let $\mathscr{U}$ be an open cover of $X$ with Lebesgue number $2\varepsilon$. Then $q_S(T,f,\mathscr{U})\le Q_S(T,f,\varepsilon)$ by Theorem~\ref{qQPpthm} (i), so 
  	\begin{align*}
  	\Asp(T,f,\mathscr{U})&\le\Asp_\varepsilon(T,f)\le \Asp'_\varepsilon(T,f)\\
  	&\le\lim_{\varepsilon\rightarrow0^+}\Asp'_\varepsilon(T,f)\\
  	&=\lim_{\delta\rightarrow 0^+}\sup_{\mathscr U}\left\{\Asp(T,f,\mathscr{U}):\diam(\mathscr{U})\le \delta\right\} \quad\text{(by (iii) above)}\\
  	&\le \sup_{\mathscr{U}}\Asp(T,f,\mathscr{U}).
  	\end{align*}
  	
  \end{proof}
  
  The next theorem gives a relationship between average sample pressure and topological pressure when we fix $c_S^n=2^{-n}$, similar to Theorem~\ref{asctoentthm}, which gives a relationship between average sample complexity and topological entropy.

  \begin{theorem}\label{asppresssupthm}
  	Let $T: X\rightarrow X$ be a  continuous transformation on the compact metric space $X$. Let $f\in C(X,\mathbb{R})$ and $S\subset n^*$. For the fixed system of coefficients $c_S^n=2^{-n}$ for all $n\in\mathbb{N}$ and $S\subset n^*$,
  	\begin{equation}
  	\lim_{\varepsilon\rightarrow0^+}\Asp_\varepsilon'(T,f)=P(T,f).
  	\end{equation}
  \end{theorem}
  
  \begin{proof}
  	By Corollary~\ref{asppressupboundcor} we know $\Asp'_\varepsilon(T,f)\le P(T,f)$, so it suffices to show the opposite inequality. 
  	
  	We will proceed in a similar manner as we did in the proof of Theorem~\ref{asctoentthm}. Most of the calculations in this proof can be found in the proof of the theorem either directly or by replacing $N(\mathscr{U}_S)$ in that proof by $p_S(T,f,\mathscr{U})$. For that reason, many of the details have been left out. Let $\mathscr{U}$ be an open cover of $X$. Recall that
  	\be
  	p_n(T,f,\mathscr{U})=\inf\left\{\sum_{V\in\mathscr{V}}\sup_{x\in V}\exp\left(\sum_{i=0}^{n-1}f(T^ix)\right):\mathcal{V} \text{ is a finite subcover of }\bigvee_{i=0}^{n-1} T^{-i}\mathscr{U}\right\}.
  	\en
  	Denote $P(T,f,\mathscr{U})=\lim_{n\rightarrow\infty}(1/n)\log p_n(T,f,\mathscr{U})$. Then, given $\varepsilon>0$, choose $k_0\in\mathbb{N}$ large enough that for every $k>k_0$
  	\begin{equation}
  	0 \leq \frac{1}{k}\log p_k(T,f,\mathscr{U})-P(T,f,\mathscr{U})<\varepsilon.
  	\end{equation}
  	Let $k>\max\{2k_0,2/\varepsilon\}$ and assume $k$ is even. Choose $n>k$ such that $k/2$ divides $n$ and form the family of bad sets $\mathcal{B}(n,k,\varepsilon)$ as in the statement of Lemma~\ref{goodsetlemma}. Let $S \notin \mathcal B$ and $E= n^* \setminus (S+k^*)$. Note that $|E| \leq n \varepsilon$ and
  	\be	
  	p_E(T,f,\mathscr{U}) \leq |\mathscr{U}_E|
  	e^{|E| \cdot ||f||_\infty} 
  	\leq (e^{||f||_\infty} |\mathscr{U}|)^{|E|}. 
  	\en
  	Thus if $S \notin \mathcal B$, by Lemma~\ref{presdissub} 
  	\begin{equation}\label{aspdissetseq}
  	\begin{aligned}
  	\log p_{S+k^*}(T,f,\mathscr{U})&\ge  \log p_{n^*}(T,f,\mathscr{U}) - \log p_E(T,f,\mathscr{U})\\
  	&\geq \log p_{n^*}(T,f,\mathscr{U}) - n \varepsilon( ||f||_\infty + \log |\mathscr{U}|),
  	\end{aligned}
  	\end{equation}
  	so that
  	\be
  	\frac{1}{n} \log p_S(T,f,\mathscr{U}_{k^*}) = \frac{1}{n} \log p_{S+k^*}(T,f,\mathscr{U}) \geq \frac{1}{n} \log p_{n^*}(T,f,\mathscr{U}) - \varepsilon( ||f||_\infty +\log |\mathscr{U}|).
  	\en
  	This implies (remembering that for large $n$ we have  $|\mathcal B^c|/2^n \geq 1 - \varepsilon$)
  	\be\label{eq:asp1}
  	\begin{aligned}
  		\Asp'(T,f,\mathscr{U}_{k^*})&=\limsup_{n\rightarrow\infty}\frac{1}{2^n}\sum_{S\subset n^*}\frac{\log p_S(T,f,\mathscr{U}_{k^*})}{n}\\
  		&\geq \limsup_{n \to \infty} \left[\frac{1}{2^n}\sum_{S \notin \mathcal B}[P(T,f,\mathscr{U})- \varepsilon (||f||_\infty + \log |\mathscr{U}|)] + \frac{1}{2^n} \sum_{S \in \mathcal B}  \frac{\log p_S(T,f,\mathscr{U}_{k^*})}{n}\right] \\
  		&\geq (1-\varepsilon)[P(T,f,\mathscr{U}) - \varepsilon( ||f||_\infty + \log |\mathscr{U}|)],
  	\end{aligned}
  	\en
  	since 
  	\be\label{eq:asp2}
  	\Big|\frac{1}{2^n} \sum_{S \in \mathcal B}  \frac{\log p_S(T,f,\mathscr{U}_{k^*})}{n}\Big|
  	\leq  (||f||_\infty + k \log |\mathscr{U}|) \frac{1}{2^n} \sum_{S \in \mathcal B} \frac{|S|}{n}
  	\leq (||f||_\infty + k \log |\mathscr{U}|) \frac{|\mathcal B|}{2^n} \to 0.
  	\en
  	Thus 
  	\be
  	\begin{gathered}
  	\lim_{k \to \infty} \Asp'(T,f,\mathscr{U}_{k^*}) \geq  (1-\varepsilon)[P(T,f,\mathscr{U}) - \varepsilon( ||f||_\infty + \log |\mathscr{U}|)], \quad\text{and hence}\\ 
  	\lim_{k \to \infty} \Asp'(T,f,\mathscr{U}_{k^*}) \geq P(T,f,\mathscr{U}).
  	\end{gathered}
  	\en
  	
  	Since $\diam(\mathscr{U}_{k^*}) \leq \diam(\mathscr{U})$, it follows that
  	\be\label{supaspsuppineq}
  	\sup_{\mathscr{U}}\left\{\Asp'(T,f,\mathscr{U}):\diam(\mathscr{U})\le \delta\right\}\ge
  	\sup_{\mathscr{U}}\left\{P(T,f,\mathscr{U}):\diam(\mathscr{U})\le \delta\right\}.
  	\en
  	Combining Equation \ref{supaspsuppineq}, Lemma~\ref{asplem} (i), and Theorem 9.4 in \cite{waltersergodic} we have
  	\be
  	\begin{aligned}
  		P(T,f)&\ge\lim_{\varepsilon\rightarrow0^+}\Asp'_\varepsilon(T,f)=\lim_{\delta\rightarrow0^+}\sup_{\mathscr U}\left\{\Asp'(T,f,\mathscr{U}):\diam(\mathscr{U})\le \delta\right\}\\
  		&\ge \lim_{\delta\rightarrow0^+}\sup_{\mathscr U}\left\{P(T,f,\mathscr{U}):\diam(\mathscr{U})\le \delta\right\}=P(T,f),
  	\end{aligned}
  	\en
  	so
  	\be
  	\lim_{\varepsilon\rightarrow0^+}\Asp'_\varepsilon(T,f)=P(T,f).
  	\en

  \end{proof}
  %-------------------------------------------------------------------------------------------------------------------------
  \begin{corollary}\label{asppresssupcor}
  	$\displaystyle{
  		\sup_{\mathscr{U}}\Asp(T,f,\mathscr{U})=P(T,f).}$
  \end{corollary}
  We note that $\sup_\mathscr{U}\Asp'(T,f,\mathscr{U})$ might be strictly larger than $P(T,f)$ (see \cite[Remark (18), p. 212]{waltersergodic}).
  \begin{theorem}\label{aspepslimthm}
  	Let $T: X\rightarrow X$ be a  continuous transformation on a compact metric space and $f\in C(X,\mathbb{R})$. For the system of coefficients $c_S^n=2^{-n}$ for all $n\in\mathbb{N}$ and $S\subset n^*$,
  	\begin{equation}
  	\lim_{\varepsilon\rightarrow0^+}\Asp_\varepsilon(T,f)= \lim_{\varepsilon\rightarrow0^+}\Asp'_\varepsilon(T,f)=P(T,f).
  	\end{equation}
  \end{theorem}
  \begin{proof}
  	Given $\varepsilon>0$, let $\mathscr{U}_\varepsilon$ be the open cover of $X$ by balls of radius $2\varepsilon$, and let $\mathscr{V}_\varepsilon$ be the open cover of $X$ by balls of radius $\varepsilon/2$. $\mathscr{U}_\varepsilon$ has $2\varepsilon$ for a Lebesgue number, so by Theorem~\ref{qQPpthm}, for each $S\subset n^*$
  	\begin{equation}\label{qQPpcompeq}
  	q_S(T,f,\mathscr{U}_\varepsilon)\le Q_S(T,f,\varepsilon)\le P_S(T,f,\varepsilon).
  	\end{equation}
  	Combining Equation~\ref{qQPpcompeq} with Theorem~\ref{asppresssupthm} and  Lemma \ref {asplem} gives the result.
  \end{proof}
  \begin{corollary}\label{ascepslimcor}
  	If $X$ is a compact metric space and $T:X\rightarrow X$ is a continuous map, for the system of coefficients $c_S^n=2^{-n}$ for all $n\in\mathbb{N}$ and $S\subset n^*$,
  	\begin{equation}
  	\lim_{\varepsilon\rightarrow0^+}\Asc_\varepsilon(X,T)=\lim_{\varepsilon\rightarrow0^+}\Asc'_\varepsilon(X,T)=\htop(X,T).
  	\end{equation}
  \end{corollary}

  In practice, we will fix an open cover $\mathscr{U}$ or an $\varepsilon>0$ when doing calculations to find values of $\Asp(T,f,\mathscr{U})$, $\Asp'(T,f,\mathscr{U})$, $\Asp_\varepsilon(T,f)$, and $\Asp'_\varepsilon(T,f)$. As with average sample complexity and intricacy, we define the \emph{average sample pressure function} using $(S,\varepsilon)$ spanning sets by
  \begin{equation}
  \Asp_\varepsilon(T,f,n)=\frac{1}{n}\sum_{S\subset n^*}c_S^n\log Q_S(T,f,\varepsilon),
  \end{equation}
  or with $(S,\varepsilon)$ separated sets by
  \begin{equation}
  \Asp'_\varepsilon(T,f,n)=\frac{1}{n}\sum_{S\subset n^*}c_S^n\log P_S(T,f,\varepsilon).
  \end{equation}
  We also define these functions using open covers:
  \begin{equation}
  \Asp(T,f,\mathscr{U},n)=\frac{1}{n}\sum_{S\subset n^*}c_S^n\log q_S(T,f,\mathscr{U}).
  \end{equation}
  and
  \begin{equation}
  \Asp'(T,f,\mathscr{U},n)=\frac{1}{n}\sum_{S\subset n^*}c_S^n\log p_S(T,f,\mathscr{U}),
  \end{equation}
  
  Now we compute average sample pressure for some shifts of finite type and potential functions.
  Given a shift of finite type $(X,\sigma)$ and a subset $S\subset n^*$, recall that $\mathscr{L}_S(X)$ denotes the set of words seen at the places in $S$ for all legal words in $X$. Recall also that the metric, $d$, we put on subshifts is defined by $d(x,y)=1/(m+1)$, where $m=\inf\{|k|:x_k\ne y_k\}$. We consider the case when the potential function $f\in C(X,\mathbb{R})$ is a function of a single coordinate, i.e. $f(x)=f(x_0)$. Letting $\varepsilon=1$ and $c_S^n$ be a system of coefficients, the average sample pressure of a shift of finite type $X$ and potential function $f$  is given by
  \begin{equation}
  \Asp_1(\sigma,f)=\limsup_{n\rightarrow\infty}\frac{1}{n}\sum_{S\subset n^*}c_S^n\log\sum_{w\in\mathscr{L}_S(X)}\exp\left(\sum_{i=1}^{|S|}f(w_i)\right).
  \end{equation}
  Notice if $f(x)\equiv0$, then we get
  \be
  \Asp_1(\sigma,0)=\limsup_{n\rightarrow\infty}\frac{1}{n}\sum_{S\subset n^*}c_S^n\log |\mathscr{L}_S(X)|=\Asc(X,\mathscr{U}_0,\sigma).
  \en

  \begin{example}
  	Consider  the full $r$-shift $\Sigma_r$ with a function $f$ that depends on a single coordinate. We have
  	\be
  	\sum_{S\subset n^*}\log\sum_{w\in\mathscr{L}_S(X)}\exp\left(\sum_{i=1}^{|S|}f(w_i)\right)=\sum_{k=0}^n\binom{n}{k}\log\left(\sum_{i=0}^re^{f(i)}\right)^k.
  	\en
  	Thus, if we fix $c_S^n=2^{-n}$, then for the full $r$-shift
  	\be
  	\begin{aligned}
  		\Asp_1(\sigma,f)&=\lim_{n\rightarrow\infty}\frac{1}{n}\frac{1}{2^n}\sum_{S\subset n^*}\log\sum_{w\in\mathscr{L}_S(X)}\exp\left(\sum_{i=1}^{|S|}f(w_i)\right)\\
  		&=\lim_{n\rightarrow\infty}\frac{1}{n}\frac{1}{2^n}\sum_{k=0}^n\binom{n}{k}\log\left(\sum_{i=0}^{r-1}e^{f(i)}\right)^k=\frac{1}{2}\log\left(\sum_{i=0}^{r-1}e^{f(i)}\right) = \frac{1}{2} P(\sigma,f).
  	\end{aligned}
  	\en
  	
  \end{example}

  \begin{center}
  	\begin{table}
  		\begin{tabular}{@{}m{.25\textwidth}m{.05\textwidth}m{.22\textwidth}m{.08\textwidth}m{.12\textwidth}m{.12\textwidth}@{}}
  			\toprule
  			\hspace{.4in}$M$ & $\rho(M)$
  			&\ \hspace{.25in} \text{Graph} & \text{Entropy} & $\Asc(X,\sigma)$ & $\Int(X,\sigma)$
  			\\
  			\midrule
  			$\left(
  			\begin{array}{ccc}
  			0 & 1 & 1 \\
  			1 & 0 & 1 \\
  			1 & 1 & 0 \\
  			\end{array}
  			\right)$ & 2 &\includegraphics[width=1in,keepaspectratio]{AdjA5}  & 0.693 & 0.448 & 0.203  \\
  			$\left(
  			\begin{array}{ccc}
  			1 & 1 & 0 \\
  			0 & 0 & 1 \\
  			1 & 1 & 1 \\
  			\end{array}
  			\right)$ & 2 &\includegraphics[width=1in,keepaspectratio]{AdjA14}  & 0.693 & 0.448 & 0.203 \\
  			\bottomrule
  		\end{tabular}
  		\caption{Two shifts that have the same entropy, $\Asc$, and $\Int$.\label{a10a11table}}
  	\end{table}
  \end{center}
  \begin{example}
  	Consider the shifts of finite type in Table~\ref{a10a11table}. We see that they are very similar and indistinguishable by the measures of complexity we have considered previously: entropy, average sample complexity, and intricacy. Suppose $f_1,f_2$ are two functions of a single coordinate on each shift of finite type defined by
  	\be
  	f_1(x)=\left\{\begin{array}{rr}
  		0,&x_0=0\\
  		0,&x_0=1\\
  		1,&x_0=2
  	\end{array}\right.\quad{\text{and}}\quad
  	f_2(x)=\left\{\begin{array}{rr}
  		0,&x_0=0\\
  		1,&x_0=1\\
  		0,&x_0=2
  	\end{array}\right.
  	\en
  	
  	\begin{table}
  		\begin{center}
  			\begin{tabular}{@{}m{.25\textwidth}m{.25\textwidth}m{.2\textwidth}m{.2\textwidth}@{}}
  				\toprule
  				\hspace{.4in}$M$ 
  				&\ \hspace{.25in} \text{Graph} & $\Asp_1(\sigma,f_1,10)$ & $\Asp_1(\sigma,f_2,10)$
  				\\
  				\midrule
  				$\left(
  				\begin{array}{ccc}
  				0 & 1 & 1 \\
  				1 & 0 & 1 \\
  				1 & 1 & 0 \\
  				\end{array}
  				\right)$  &\includegraphics[width=1in,keepaspectratio]{AdjA5}  & 0.660 & 0.660  \\
  				$\left(
  				\begin{array}{ccc}
  				1 & 1 & 0 \\
  				0 & 0 & 1 \\
  				1 & 1 & 1 \\
  				\end{array}
  				\right)$  &\includegraphics[width=1in,keepaspectratio]{AdjA14} & 0.722 & 0.633 \\
  				\bottomrule
  			\end{tabular}
  			\caption[Calculations of average sample pressure]{Calculations of $\Asp$ for two shifts that have the same entropy, $\Asc$, and $\Int$.\label{a10a11table2}}
  		\end{center}
  	\end{table}

  	Table~\ref{a10a11table2} shows the calculations of $\Asp_1(\sigma,f_1,10)$ and $\Asp_1(\sigma,f_2,10)$ for these two subshifts. Notice that $f_1$ places more weight on the symbol $2$, whereas $f_2$ places more weight on the symbol $1$. It is not surprising that the second SFT has a larger value for $\Asp_1(\sigma,f_1,10)$ than it does for $\Asp_1(\sigma,f_2,10)$, since every time a $1$ appears in a sequence of of this shift, a $2$ must follow it. 
  \end{example}

\section{Extensions to General Intricacy Weights|}\label{sec:appendix2}
  Using the characterization in Theorem \ref{coeffprop} we show that the main theorems (\ref{asctoentthm}, \ref{asppresssupthm}, \ref{mtheoreticasctoent}, \ref{Asccompenteqthm}, \ref{infocor}) extend to arbitrary intricacy weights $c_S^n$. 
 
 We preserve the setup and notation of of Sections \ref{sec:supequalsent} and \ref{asp}. Given $0 < \varepsilon <1, k$, and $n>k$, form the sets $K_i$ and the family $\mathcal B(n,k,\varepsilon)$ as in Section \ref{sec:supequalsent}. Let $\lambda_c$ be a symmetric probability measure on $[0,1]$ and form the weights $c_S^n$ according to Formula \ref{eq:intricacyweights}. Noting that $c_S^n$ depends only on $|S|$, and of necessity small and large $|S|$ must be treated separately, we fix $\delta >0$ and estimate separately the sums of $c_S^n$ over the sets $S$ with $|S| \geq n \delta$ and $|S| < n \delta$. 
 
 \begin{rem}
 	Any possible point mass of $\lambda_c$ at $0$ (hence also at $1$) contributes nothing to any $c_S^n$, so we assume that $\lambda_c\{0,1\}=0$.
 \end{rem}
 
 The following lemma will substitute in the general situation for Lemma \ref{goodsetlemma}. 
 \begin{lemma}
 	\label{lem:cestimate}
 	Fix $0 < \varepsilon < 1, k \in \mathbb N$, and $n>k$ and form the sets $K_i$ and the family $\mathcal B= \mathcal B(n,k,\varepsilon)$ as before. For each $j=0,\dots,n$ let $\mathcal B_j=\mathcal B_j (n,k,\varepsilon)=\{S \in \mathcal B: |S|=j\}$. Let $0< \delta < 1/2$. Then there is a constant $C$ such that 
 	\be
 	{\sum_{j\geq n \delta} \sum _{S \in \mathcal B_j} c_S^n \leq  C\left[ e^{2H(\varepsilon)/k}(1-\delta)^\varepsilon\right] ^n, \quad\text{and}}
 	\en
 	\be
 	\sum_{|S| < n \delta} c_S^n \leq \lambda_c[0,3\delta/2] + t_n, \quad\text{where } t_n \to 0 \text{ as } n \to \infty.
 	\en
 \end{lemma}
 \begin{proof}
 	Write $m=\ceil{2n/k}$. Using Stirling's approximation, writing out the factorials, and estimating gives
 	\be\label{eq:bjest}
 	|\mathcal B_j| \leq \binom{m}{\ceil{m \varepsilon}} \binom{\ceil{n(1-\varepsilon)}}{j} \leq C e^{mH(\varepsilon)}\left(\frac{n-j}{n}\right)^{n \varepsilon}
 	\en
 	for some constant $C$. Thus 
 	\be
 	\begin{aligned}
 		\sum_{j \geq n\delta}\sum_{S \in \mathcal B_j} c_S^n &\leq 
 		C e^{mH(\varepsilon)}(1-\delta)^{n \varepsilon} \int_0^1\sum_{j \geq n \delta} \binom{n}{j} x^j(1-x)^{n-j}\lambda_c(dx)\\
 		&\leq 	C e^{mH(\varepsilon)}(1-\delta)^{n \varepsilon} \int_0^1\sum_{j=0}^n \binom{n}{j} x^j(1-x)^{n-j}\lambda_c(dx)\\
 		&= 	C e^{mH(\varepsilon)}(1-\delta)^{n \varepsilon} .
 	\end{aligned}
 	\en
 	
 	For the second formula, recall that by the Law of Large Numbers, if $1\geq x \geq 3 \delta/2$ then 
 	\be
 	\sum_{j < n \delta} \binom{n}{j} x^j (1-x)^{n-j}   \to 0	\quad\text{as $n \to \infty$ (and is of course bounded)}.
 	\en
 	Therefore also
 	\be
 	t_n=\int_{3\delta	/2}^1\sum_{j < n \delta} \binom{n}{j} x^j (1-x)^{n-j} \lambda_c(dx) \to 0 \quad\text{as } n \to \infty\en
 	and
 	\be
 	\begin{aligned}
 		\sum_{|S| <n\delta}c_S^n &= \int_0^1 \sum_{j<n \delta} \binom{n}{j} x^j(1-x)^{n-j}\lambda_c(dx)\\
 		&=\int_0^{3\delta/2}\sum_{j<n \delta} \binom{n}{j} x^j(1-x)^{n-j}\lambda_c(dx) +\int_{3\delta/2}^1\sum_{j<n \delta} \binom{n}{j} x^j(1-x)^{n-j}\lambda_c(dx)\\
 		&\leq \lambda_c[0,3\delta/2] + t_n.
 	\end{aligned}\en
 	
 \end{proof}
 
 \begin{lemma}
 	\label{lem:csumest}Given $\varepsilon >0$, there is $k(\varepsilon)$ such that for fixed $k>k(\varepsilon)$ and all large enough $n>k$
 	\be\label{eq:csumest}
 	\sum_{S \in \mathcal B(n,k,\varepsilon)} c_S^n < \varepsilon.
 	\en
 \end{lemma}
 \begin{proof}
 	Since $\lambda_c$ does not have a point mass at $0$, given $\varepsilon>0$ there is $\delta(\varepsilon)$ such that $\lambda_c[0,3\delta(\varepsilon)/2]<\varepsilon$. Let 
 	\be
 	k(\varepsilon) = \frac{2H(\varepsilon)}{-\varepsilon\log(1-\delta(\varepsilon))}.
 	\en
 	If $k>k(\varepsilon)$, then
 	\be
 	e^{2H(\varepsilon)/k}(1-\delta(\varepsilon))^\varepsilon) < 1,
 	\en
 	so by Lemma \ref{lem:cestimate}
 	\be
 	\lim_{n \to \infty} \sum_{S \in \mathcal B(n,k,\varepsilon)} c_S^n \leq \lim_{n \to \infty} \left[ C \left( e^{2H(\varepsilon)/k}(1- \delta(\varepsilon)^\varepsilon\right) ^n +t_n \right] + \lambda_c[0,3\delta(\varepsilon)/2] = \lambda_c[0,3\delta(\varepsilon)/2] < \varepsilon. 
 	\en
 \end{proof}
 
 Now we extend Theorem \ref{asppresssupthm} from the special weights $c_S^n=1/2^n$ to general systems of coefficients $c_S^n$.
 \begin{theorem}\label{thm:asppresssup}
 	Let $T: X\rightarrow X$ be a  continuous transformation on the compact metric space $X$. Let $f\in C(X,\mathbb{R})$ and $S\subset n^*$. For any system of coefficients $c_S^n$ (see the definition at the beginning of Section \ref{Sec:IntAsc}), for all $n\in\mathbb{N}$ and $S\subset n^*$,
 	\begin{equation}
 	\lim_{\varepsilon\rightarrow0^+}\Asp_\varepsilon'(T,f)=P(T,f).
 	\end{equation}
 \end{theorem}
 \begin{proof}
 	The proof is largely the same until just before Equation \ref{eq:asp1}. Replace the lines starting at the line preceding that equation and continuing through Equation \ref{eq:asp2} by the following:
 	
 	Let $\delta(\varepsilon)$ and 	$k(\varepsilon) = {2H(\varepsilon)}/[-\varepsilon\log(1-\delta(\varepsilon))]$ be as in Lemma \ref{lem:csumest}, so that $\sum_{S \in \mathcal B} c_S^n < \varepsilon$. Then
 	\be
 	\begin{aligned}
 		\Asp'(T,f,\mathscr{U}_{k^*})&=\limsup_{n\rightarrow\infty}\sum_{S\subset n^*}c_S^n \frac{\log p_S(T,f,\mathscr{U}_{k^*})}{n}\\
 		&\geq \limsup_{n \to \infty} \left[\sum_{S \notin \mathcal B}c_S^n[P(T,f,\mathscr{U})- \varepsilon (||f||_\infty + \log |\mathscr{U}|)] +  \sum_{S \in \mathcal B} c_S^n \frac{\log p_S(T,f,\mathscr{U}_{k^*})}{n}\right] \\
 		&\geq (1-\varepsilon)[P(T,f,\mathscr{U})-\varepsilon(||f||_\infty + \log|\mathscr{U})] - \varepsilon( ||f||_\infty + k \log |\mathscr{U}|),
 	\end{aligned}
 	\en
 	since 
 	\be
 	\begin{aligned}
 		\Big| \sum_{S \in \mathcal B} c_S^n \frac{\log p_S(T,f,\mathscr{U}_{k^*})}{n}\Big|
 		&\leq  (||f||_\infty + k \log |\mathscr{U}|)  \sum_{S \in \mathcal B} c_S^n \frac{|S|}{n}\\
 		&\leq (||f||_\infty + k \log |\mathscr{U}|) \sum_{S \in \mathcal B} c_S^n < (||f||_\infty + k \log |\mathscr{U}|) \varepsilon.
 	\end{aligned}
 	\en
 	Then the rest of the argument ($\lim_{k \to \infty} \Asp'(T,f,\mathscr{U}_{k^*}) \geq \dots$) proceeds as before.
 \end{proof}
 
 Then the extensions to arbitrary systems of coefficients of Theorem \ref{asctoentthm}, Corollary \ref{asppresssupcor}, Theorem \ref{aspepslimthm}, and Corollary \ref{ascepslimcor} are direct consequences of Theorem \ref{thm:asppresssup}.
 
 We turn now to the extension to arbitrary systems of coefficients of Theorems  \ref{mtheoreticasctoent} and \ref{Asccompenteqthm} on measure-theoretic average sample complexity. 
 
 \begin{theorem}\label{thm:mtheoreticasctoent}
 	Let $(X,\mathscr{B},\mu,T)$ be a measure-preserving system and fix a system of coefficients.  Then
 	\be
 	\sup_{\alpha}\Asc_{\mu}(X,\alpha,T)=h_\mu(X,T).
 	\en
 \end{theorem}
 \begin{proof}
 	Follow the proof of Theorem  \ref{mtheoreticasctoent}. Use $H_\mu(\alpha_E) \leq |E| H_\mu(\alpha)$,  replace $1/2^n$ by $c_S^n$ and move it inside the sums, replace $\log N(S+k^*)$ by $H_\mu(\alpha_{S+k^*})$ and similarly for other refinements along subsets of $n^*$, and replace $\htop(X,\mathscr{U},T)$ by $h_\mu(X,\alpha,T)$.
 \end{proof}
 
 The extensions to arbitrary weights of Theorem \ref{Asccompenteqthm} and Proposition \ref{infocor} are not so straightforward, since averaging with general weights does not translate directly to sampling along a set of first-return times in a measure-preserving system. 
 Indeed, in the general situation the analogue of (\ref{eq:AscForm}) is
 %8/26replaced Acc by Asc  throughout, added -1s in exponents, corrected calculations in Th 8.8
 \be
 \begin{aligned}
 	\Asc_\mu^\lambda (X,T,\alpha)
 	&= \int_0^1 \lim_{n \to \infty} \frac{1}{n} \sum_{\substack{S\subset n^*\\ 0\in S }} H_\mu(\alpha_S) P_A^{(p)}\{\xi \in A: S(\xi_0^{n-1})=S\}\, d\lambda(p) \\
 	&= \int_0^1 \lim_{n \to \infty} \frac{1}{n} \sum_{\substack{S\subset n^*\\ 0\in S }} H_\mu(\alpha_S) p^{|S|-1}(1-p)^{n-|S|} \,d\lambda(p) = \lim_{n \to \infty} \frac{1}{n} \sum_{\substack{S\subset n^*\\ 0\in S }} b_S^nH_\mu(\alpha_S),
 \end{aligned}
 \en
 where again $P_A^{(p)}$ is the normalized restriction to $A=[1]$ of Bernoulli $(p,1-p)$ measure, $\lambda$ is a symmetric probability measure on $[0,1]$, and the weights 
 \be
 b_S^n= \int_0^1p^{|S|-1}(1-p)^{n-|S|} d\lambda(p)
 \en
 sum to $1$ as $S$ runs through all subsets of $n^*$ which contain $0$. 
 The limits exist because the last one does, by subadditivity. 
 Note that in this situation, since we are conditioning on $A$,
 when $\lambda$ is the point mass at $1/2$, $\Asc_\mu^\lambda(X,T,\alpha)=\Asc_\mu(X,\alpha,T)$.
 
 \begin{theorem}\label{thm:Asccompenteqthm}
 	Let $(X,\mathscr{B},\mu,T)$ be an ergodic measure-preserving system and $\alpha$ a finite measurable partition of $X$. 
 	Let $A=[1]=\{\xi\in\Sigma_2^+:\xi_0=1\}$ and let $\beta=\alpha\times A$ be the related finite partition of $X\times A$. 
 	Denote by $T_{X\times A}$ the first-return map on $X\times A$, and for each $0<p<1$ let $P^{(p)}_A=P^{(p)}/P^{(p)}[1]$ denote the measure $P^{(p)}=\mathscr B(p,1-p)$ restricted to $A$ and normalized. Let $\lambda$ be a symmetric probability measure on $[0,1]$.
 	Then
 	%8/7 removed _c twice
 	\begin{equation}\label{eq:Asccompenteq}
 	\Asc_\mu^\lambda (X,\alpha,T)=\lim_{n \to \infty} \frac{1}{n} \int_0^1p  \int_A H_\mu (\alpha_{S(\xi_0^{m_\xi(n)-1})}) \,dP_A^{(p)}(\xi) \, \lambda(dp)
 	\leq \int_0^1p \,  h_{\mu\times P^{(p)}_A}(X\times A,\beta,T_{X\times A}) \lambda(dp).
 	\end{equation}
 \end{theorem}
 \begin{proof}
 	Consider first a fixed $p\in(0,1)$. 
 	Now $m_\xi(n)/n \to 1/p$ as $n \to \infty$ for $P_A^{(p)}$-a.e. $\xi \in A$. 
 	Repeat the proof of Theorem \ref{Asccompenteqthm} with $U_\varepsilon (n)=\{ \xi \in A:|m_\xi/n - 1/p|> \varepsilon\}$ and $j=\floor{n/p}$. Again $P_A^{(p)}(U_\varepsilon) \to 0$ as $n \to \infty$. As in (\ref{epsprimeeq}), for $\xi \in A \setminus U_\varepsilon$,
 	\be
 	\frac{1}{n}\left\lvert H_\mu\left(\alpha_{S\left(\xi_0^{j}\right)}\right)-H_\mu\left(\alpha_{S\left(\xi_0^{m_\xi(n)-1}\right)}\right)\right\rvert 
 	\leq 2\varepsilon H_\mu(\alpha)\text{  for large } n.
 	\en
 	Thus for fixed $p$,
 	\be
 	\lim_{n \to \infty} \frac{1}{n} \int_A H_\mu(\alpha_{S(\xi_0^{m_\xi(n)-1})}) \, dP_A^{(p)}
 	= \lim_{n \to \infty} \frac{1}{n} \int_A H_\mu(\alpha_{S(\xi_0^{\floor{n/p}})}) \, dP_A^{(p)}
 	= \lim_{n \to \infty} \frac{1}{pj} \int_A H_\mu(\alpha_{S(\xi_0^{j-1})}) \, dP_A^{(p)}.
 	\en
 	As in the proof of Theorem \ref{Asccompenteqthm} it follows that
 	\be
 	\lim_{n \to \infty} \frac{1}{n} p \int_A H_\mu(\alpha_{S(\xi_0^{m_\xi(n)})}) \, dP_A^{(p)}
 	\leq  p \, h_{\mu \times P^{(p)}_A}(X \times A, \beta, T_{X \times A}). 
 	\en
 	Then integrate with respect to $\lambda$ and apply the Bounded Convergence Theorem.
 \end{proof}

 \begin{proposition}\label{cor:infocor}
 	Let $(X,\mathscr{B},\mu,T)$ be a 1-step Markov shift, $\alpha$ the finite time-0 generating partition of $X$, and $\lambda$ a symmetric probability measure on $[0,1]$. Then
 	\begin{equation}%\label{Ascseriesequal}
 	\Asc_\mu^\lambda (X,\alpha,T)=  \sum_{i=1}^\infty \int_0^1 p^2 (1-p)^{i-1} \lambda (dp)   H_\mu\left(\alpha\mid\alpha_{i}\right).
 	\end{equation}
 \end{proposition}
 \begin{proof}
 	Using the proof of Proposition \ref{infocor},
 	%8/7 added ^{(p)} twice
 	%8/26 changed to dlambda(p)
 	\be\begin{aligned}
 		\Asc_\mu^\lambda (X,\alpha,T)&=\int_0^1  \lim_{n \to \infty} \frac{1}{n} \int_A H_\mu (\alpha_{S(\xi_0^{n-1})}) \,dP_A^{(p)}(\xi) \, d\lambda(p)\\ 
 		&=\int_0^1  \lim_{n \to \infty} \frac{1}{n}p \int_A H_\mu (\alpha_{S(\xi_0^{m_\xi(n)-1})}) \,dP_A^{(p)}(\xi) \, d\lambda(p)\\
 		&= \int_0^1  p\int_AH_\mu(\alpha|\alpha_{S_1(\xi)}) \, dP_A^{(p)} \, d\lambda(p)
 		=\int_0^1  \sum_{i=1}^\infty p\, P_A^{(p)}(A_i)H_\mu(\alpha|\alpha_i)\,  d\lambda(p) \\
 		&=  \sum_{i=1}^\infty \int_0^1 p^2 (1-p)^{i-1} \, d\lambda(p)H_\mu(\alpha|\alpha_i).
 	\end{aligned}
 	\en
 \end{proof}
 
 \begin{theorem}
 	With the notation and hypotheses of Theorems \ref {thm:Asccompenteqthm} and \ref{thm:FirstReturns},
 	%8/7 added int_0^1, ^{(p)}
 	%9/22 added p in lines 1 and 2
 	\be
 	\begin{aligned}
 		\Asc_\mu^\lambda (X,\alpha,T)&=
 		\lim_{n \to \infty} \frac{1}{n} \int_0^1 p  H_{\mu \times P_A^{(p)}} (\beta_{0,n-1}^*|\mathcal A_{-\infty,\infty}^*)\, d\lambda(p)\\
 		&= \int_0^1 p  H_{\mu \times P_A^{(p)}}(\beta|\beta_{1,\infty}^* \vee \mathcal A_{-\infty,\infty}^*) \, d\lambda(p)\\
 		%8/24 added factor of 1/2
 		&= \int_0^1 p  h_{\sigma_A}^{(p)}(X,T,\mu)\, d\lambda(p).
 	\end{aligned}\en
 \end{theorem}
 
 \begin{proof}
 	We have
 	%8/26changed exp from n-1 to m xi (n) -1
 	\be
 	\begin{aligned}
 		\Asc_\mu^\lambda (X,T,\alpha)&=\int_0^1 p \lim_{n \to \infty} \frac{1}{n}\int_A H_\mu(\alpha_{S(\xi_0^{m_{\xi(n)}-1})})\, dP_A^{(p)} \, d\lambda(p)\\
 		&=\lim_{n \to \infty} \frac{1}{n} \int_0^1 p H_{\mu \times P_A^{(p)}}(\beta_{0,n-1}^*|\mathcal A_{-\infty,\infty}^*) \, d\lambda(p)\\
 		&=\int_0^1 p H_{\mu \times P_A^{(p)}}(\beta|\beta_{1,\infty}^* \vee \mathcal A_{-\infty,\infty}^*) \, d\lambda(p).
 	\end{aligned}
 	\en
 	%8/24changed pf
 	Also, as in (\ref{eq:FibEnt}) and using (\ref{eq:Asccompenteq}),
 	%8/7 added int, ^p
 	\be
 	\begin{gathered}
 	\int_0^1 p	h_{\sigma_A}^{(p)}(X,T,\mu)\, d\lambda(p)= \int_0^1\int_A p H_\mu(\alpha|\alpha_{1,\infty}^*) \, dP_A^{(p)}(\xi) \, d\lambda (p)\\
 	=\lim_{n \to \infty} \int_0^1 \int_A \frac{p}{n} 
 	H_\mu (\alpha_{S(\xi_0^{m_\xi (n)-1})}) \, dP_A^{(p)}(\xi) \,d\lambda (p)
 	=\Asc_\mu^\lambda(X,T,\alpha).
 	\end{gathered}
 	\en
 \end{proof}

\bibliographystyle{amsplain}
\bibliography{PW}
\end{document}